\documentclass[reqno]{amsart}
\usepackage[doi=false, isbn=false, url=false, eprint=false, maxnames=10, giveninits=true, style=alphabetic]{biblatex}
\usepackage{graphicx}
\usepackage[usenames]{color}
\usepackage{amsmath}
\usepackage[all]{nowidow}
\usepackage{mathtools}
\usepackage{amsfonts}
\usepackage[font={small},labelfont={bf}]{caption}
\usepackage{parskip}
\usepackage[normalem]{ulem}
\usepackage{bbm}
\usepackage{quiver}
\usepackage[T1]{fontenc}
\usepackage{palatino}
\usepackage{mathrsfs}
\usepackage{hyperref}
\usepackage[shortlabels]{enumitem}
\usepackage{amssymb}
\usepackage{enumitem}
\usepackage{amsthm}
\usepackage{wasysym}
\usepackage{cancel}
\usepackage{tikz}
\usepackage{tikz-cd}
\usepackage{adjustbox}
\usepackage[capitalize,nameinlink]{cleveref}
\usepackage{todonotes}
\usepackage{centernot}
\usepackage{fullpage}

\theoremstyle{definition}
\newtheorem{thm}{Theorem}[section]

\theoremstyle{definition}
\newtheorem{lem}[thm]{Lemma}

\theoremstyle{definition}
\newtheorem{defn}[thm]{Definition}

\theoremstyle{definition}
\newtheorem{prop}[thm]{Proposition}

\theoremstyle{definition}
\newtheorem{rmk}[thm]{Remark}

\theoremstyle{definition}
\newtheorem{cor}[thm]{Corollary}
\theoremstyle{definition}
\newtheorem{exm}[thm]{Example}

\newcommand{\twoheadrightarrowtail}{\mathrel{\mathrlap{\rightarrowtail}}\mathrel{\mkern2mu\twoheadrightarrow}}

\newcommand{\opc}{\raisebox{\depth}{\(\ocircle\)}}

\newcommand{\crefdefpart}[2]{%
  \hyperref[#2]{\namecref{#1}~\labelcref*{#1}(\ref*{#2})}%
  }

\newcommand{\tricref}[3]{%
  \hyperref[#3]{\namecref{#1}~\labelcref*{#1}(\ref*{#3})}%
  }

\newcommand{\crefpart}[3]{%
  {\namecref{#1}~\labelcref*{#1}(\labelcref{#2},\labelcref{#3})}%
  }

\crefname{lem}{Lemma}{Lemmas}
\crefname{prop}{Proposition}{Propositions}
\crefname{thm}{Theorem}{Theorems}

\tikzcdset{every label/.append style = {font = \normalsize}}
\title{The Monadic Grzegorczyk Logic}

\author{G. Bezhanishvili}
\address{New Mexico State University}
\email{guram@nmsu.edu}

\author{M. Khan}
\address{New Mexico State University}
\email{mashyk@nmsu.edu}
\subjclass[2020]{03B45, 06E25, 06E15, 03C90, 03F45}
\keywords{Modal logic, predicate 
logic, monadic fragment, Grzegorczyk logic, 
finite model property}
\date{}

\begin{document}

\begin{abstract}
    %
    We develop a semantic criterion 
    for determining whether a given monadic modal logic axiomatizes the one-variable fragment of a predicate modal logic. 
    We show that the criterion applies to the monadic Grzegorczyk logic \textbf{MGrz}, thus
    establishing that \textbf{MGrz} axiomatizes the one-variable 
    fragment of the predicate Grzegorczyk logic \textbf{QGrz}. This we do by proving the finite model property 
    of \textbf{MGrz}, which is achieved by strengthening the notion of a maximal point of a descriptive \textbf{MGrz}-frame and by refining the existing selective filtration methods. 
\end{abstract}

\maketitle

\tableofcontents


\section{Introduction}


The Grzegorczyk logic \textbf{Grz} plays a fundamental role in the study of the G\"{o}del translation. Indeed, it is the largest modal companion of 
the intuitionistic propositional calculus \textbf{IPC}, and the lattice of extensions of \textbf{IPC} is isomorphic to the lattice of normal extensions of \textbf{Grz} (see, e.g., \cite[Sec.~9.6]{CZ}). 


Our understanding of the G\"{o}del translation in the predicate case is much more limited. One reason for this is the lack of adequate semantics for predicate logics. Indeed, predicate modal logics are rarely Kripke complete \cite{SG}, and we also have frequent incompleteness with respect to the more general Kripke bundle semantics \cite{isoda, nagaoka}. For example, letting \textbf{QL} denote the predicate extension of a normal modal logic \textbf{L}, we have that \textbf{QS4} is Kripke complete (see, e.g., \cite[Thm.~15.3]{H&G}), while 
\textbf{QGrz} is not even Kripke bundle complete \cite{isoda, nagaoka}. 

It is known \cite{ras-sik} (see also  \cite[p.~156]{GSS}) 
that \textbf{QS4} is a modal companion of the  intuitionistic predicate calculus \textbf{IQC}. There have been claims in the literature that \textbf{QGrz} is also a modal companion of \textbf{IQC} \cite{pan89}, but that it is not the largest such \cite{nau91}. However, the proofs rely on the Flagg-Friedman translation \cite{flagg}, which turned out not to be faithful \cite{ino92}. Thus, these claims need to be reexamined (see \cite[Rem.~2.11.13]{GSS} 
and \cite{GLfe}).


The difficulties with predicate logics can often be avoided by restricting to fragments with one fixed variable, which we refer to as {\em monadic fragments}. On the one hand, such fragments minimally capture the complexity of modal quantification. On the other, they can be viewed as propositional modal logics with two modal operators (the second being an $\textbf{S5}$-modality  capturing the monadic quantification), and hence are amenable to the standard semantic tools of modal logic.  



The study of the monadic fragment of the classical predicate calculus was initiated by Hilbert and Ackermann \cite{hilbertackermann}. Wajsberg \cite{wajsberg} showed that this fragment is axiomatized by 
the modal logic \textbf{S5}. Prior \cite{prior} introduced the monadic intuitionistic propositional calculus \textbf{MIPC}, and Bull \cite{bull} proved that it axiomatizes the monadic fragment of \textbf{IQC}. For further results in this direction see
\cite{ono,onoszk,szk,BV1}. 
Monadic fragments of predicate modal logics were studied in \cite{FS,esa,japaridze88,japaridze,GBm}. 
In particular, Fischer Servi \cite{FS} introduced \textbf{MS4} (monadic \textbf{S4}) and showed that the G\"{o}del translation embeds \textbf{MIPC} into \textbf{MS4} faithfully, while Esakia \cite{esa} introduced \textbf{MGrz} (monadic \textbf{Grz}) and proved the same for \textbf{MGrz}. In the same paper, Esakia introduced \textbf{MGL} (monadic G\"{o}del-L\"{o}b logic),
and Japaridze \cite{japaridze88, japaridze} proved that Solovay's embedding \cite{solovay} of \textbf{GL} into Peano Arithmetic extends to \textbf{MGL}.  This was done by establishing the finite model property (fmp for short) of \textbf{MGL},  but the same question for \textbf{MGrz} remained open. 

Using predicate Kripke frames, Ono and Suzuki \cite{onoszk} developed a semantic criterion to determine when a given monadic intuitionistic logic axiomatizes the monadic fragment of a given predicate intuitionistic logic. This was further generalized by Suzuki \cite{szk} to Kripke bundles. We generalize this criterion to the setting of monadic modal logics (see \Cref{translation}). Among other things, this yields that \textbf{MS4} is the monadic fragment of \textbf{QS4}, 
and likewise that \textbf{MGrz} is the monadic fragment of \textbf{QGrz} (see \cite{FS,esa}, but note that full proofs were not given in these papers). The latter we do by establishing that \textbf{MGrz} has the 
fmp, which is our main result  and resolves in the positive the 
issue of completeness of $\textbf{MGrz}$ 
(see \cite{esa}). 
This result, in conjunction with our semantic criterion in \Cref{translation}, shows that the monadic fragment of $\textbf{QGrz}$ is Kripke bundle complete, in spite of the fact that the full predicate logic $\textbf{QGrz}$ is not. 

To establish our main result, we use a modified form of selective filtration---a standard technique in modal logic for proving completeness and decidability. For instance, selective filtration can be used to show that \textbf{GL} and \textbf{Grz} have the fmp (see, e.g., \cite[pp.~150--152]{CZ}). Our approach combines the methods of \cite{Grefe} and \cite{GBm}. More specifically, the selection method in \cite{Grefe} was used to show that Fischer Servi's  intuitionistic modal logic 
has the fmp (see also \cite[Thm.~10.19]{mdim}). This method hinges on the existence of certain maximal points, 
which are selected in the construction. In \cite{GBm}, this method was refined 
to prove that the logics $\textbf{M}^+\textbf{IPC}$ and $\textbf{M}^+\textbf{Grz}$ have the fmp (see \Cref{conclusion} for relevant definitions).  The importance of these logics stems from the fact that they embed faithfully into \textbf{MGL} \cite{GBm}, which admits a provability interpretation by Japaridze's theorem 
\cite{japaridze88,japaridze}. 
%
%


The construction we propose here is a further refinement of the one in \cite{GBm}.  
The key new notion is that of a strongly maximal point (see \Cref{strongly_maximal}), the existence of which is ensured by sharpening the well-known Fine-Esakia principle for \textbf{Grz} (see \Cref{fe_principle,smax}).
We show that our technique specializes to yield  the fmp for both $\textbf{M}^+\textbf{Grz}$ and $\textbf{MGL}$. The details are outlined in \Cref{conclusion}, together with further examples, which suggest that an appropriate modification of the construction could be used for a variety of other cases as well.

We conclude the introduction by a brief discussion of the organization of the paper. 
We assume 
familiarity with modal logic (see, e.g., \cite{CZ}), topology (see, e.g., \cite{engelking}), and category theory (see, e.g., \cite{maclane}).  In \Cref{predicate}, we provide a brief introduction to predicate modal logics (pm-logics for short), and in \Cref{monadic} to monadic modal logics (mm-logics for short). We also introduce the appropriate semantics for each: 
Kripke bundles for pm-logics and monadic Kripke frames 
for mm-logics. In \Cref{fnb}, we prove that there is an equivalence between the categories of Kripke bundles and monadic Kripke frames.
In \Cref{correspondence}, we develop a criterion 
for establishing whether a given mm-logic is the monadic fragment of a given pm-logic. In \Cref{MGrz}, we concentrate on \textbf{MGrz}, and provide an overview of its algebraic and descriptive frame semantics. The novelty of this section is in the introduction of strongly maximal points. These points are the crucial ingredients of our selection method, which is developed in \Cref{construction} to prove the fmp for \textbf{MGrz}. Finally, in \Cref{conclusion}, we provide concluding remarks, as well as further observations and open questions. 

\section{Predicate modal logic}{\label{predicate}}


Let $\mathcal{L}_Q$ be a predicate modal language with a countable set of individual variables $V$ and a countable set of predicates $\Sigma$. There are no function symbols or individual constants in $\mathcal{L}_Q$. We take $\neg$, $\vee$, $\exists$, and $\Diamond$ as our basic connectives and treat $\wedge$, $\to$, $\forall$, and $\Box$ as standard abbreviations. 
Let $\text{Form}(\mathcal{L}_Q)$ be the set of all formulae in $\mathcal{L}_Q$. 

For each propositional modal logic $\textbf{L}$, let $\textbf{QL}$ be the least predicate extension of $\textbf{L}$.   
We recall \cite[Def.~2.6.1]{GSS} that $\textbf{QL}$ is the least set of formulae in $\mathcal{L}_Q$ containing
\begin{itemize}
    \item the axioms of the classical predicate calculus,
    \item the theorems of $\textbf L$,
\end{itemize}
and closed under the inference rules

\begin{center}
    \renewcommand{\arraystretch}{2}
    \begin{tabular}{l  l}
        $\dfrac{\varphi, \varphi\to\psi}{\psi}$ & (Modus Ponens),\\
        $\dfrac{\varphi}{\Box\varphi}$ & (necessitation),\\
        $\dfrac{\varphi}{\forall x \varphi}$ & (generalization),\\
        $\dfrac{\varphi}{\varphi[\psi/P(\Vec{x})]}$ & (uniform substitution).\\
    \end{tabular}
\end{center}


Here, $P\in\Sigma$, $\Vec{x}=(x_1,\dots,x_n)$, and $\varphi[\psi/P(\Vec{x})]$ stands for the formula obtained by replacing all occurrences of $P(\Vec{y})$ in $\varphi$ with $\psi[\Vec{y}/\Vec{x}]$ so that the free variables in $\psi[\Vec{y}/\Vec{x}]$ are not
bounded after substitution. For a rigorous treatment of this rule, we refer to \cite[Sec.~2.5]{GSS}.

In particular, $\textbf{QK}$ is the least predicate extension of $\textbf K$, $\textbf{QS4}$ is the least predicate extension of $\textbf{S4}$, etc.

\begin{defn} \cite[Def.~2.6.1]{GSS}
A {\emph{predicate modal logic}} or simply a {\emph{pm-logic}} 
is any set of formulae of $\mathcal{L}_Q$ containing \textbf{QK} and closed under the above inference rules.
\end{defn}

\begin{rmk}
    The formula $\exists x\Diamond \varphi\to\Diamond\exists x\varphi$ is derivable in any pm-logic (see, e.g., \cite[Lem.~2.6.18]{GSS}).
\end{rmk}

Since the standard Kripke semantics is often inadequate for 
pm-logics (see, e.g., \cite[Thm.~7.3]{SG}), we will work with the more general Kripke bundle semantics, and view predicate Kripke frames as special Kripke bundles. For this, we first recall the notions of a Kripke frame and a p-morphism (see, e.g., \cite[Defs.~1.19,~2.10]{venema}). 




\begin{defn}\ \label{basics}
    \begin{enumerate}
        \item A {\em Kripke frame} is a pair $\mathfrak F = (X,R)$, where $X$ is a nonempty set and $R$ is a binary relation on $X$. For $x\in X$, let $R[x] = \{ y \in X \mid x \mathrel{R} y \}$. \label{basics1}
        \item A {\em p-morphism} between two Kripke frames $\mathfrak F = (X,R)$ and $\mathfrak F' = (X',R')$ is a map $f\colon X \to X'$ such that $f(R[x])=R'[f(x)]$ for each $x\in X$. 
    \end{enumerate}
\end{defn}

%
%

\begin{defn}\ {\label{kb}}
\begin{enumerate}
    \item \cite[Def.~5.2.5]{GSS} A {\em Kripke bundle} is a triple $\mathfrak{B}=(\mathfrak{F},\pi,\mathfrak{F}_0)$, where $\mathfrak{F}=(X,R)$ and $\mathfrak{F}_0=(X_0,R_0)$ are Kripke frames and $\pi\colon \mathfrak{F}\twoheadrightarrow \mathfrak{F}_0$ is an onto p-morphism. \label{kb1}
    \item \cite[Def.~5.4.1]{GSS} A {\em morphism} between Kripke bundles $\mathfrak B = (\mathfrak{F},\pi,\mathfrak{F}_0)$ and $\mathfrak B' = (\mathfrak{F}',\pi',\mathfrak{F}_0')$ is a pair $(f,g)$ such that \label{kb2}
\begin{enumerate}
    \item $f\colon\mathfrak{F}\to\mathfrak{F}'$ and $g\colon\mathfrak{F}_0\to\mathfrak{F}_0'$ are p-morphisms,
    \item the following diagram commutes
        \[\begin{tikzcd}
        	{\mathfrak{F}} && {\mathfrak{F}_0} \\
        	\\
        	{\mathfrak{F}'} && {\mathfrak{F}_0'}
        	\arrow["\pi", two heads, from=1-1, to=1-3]
        	\arrow["f"', from=1-1, to=3-1]
        	\arrow["g", from=1-3, to=3-3]
        	\arrow["{\pi'}"', two heads, from=3-1, to=3-3]
        \end{tikzcd}\]
    \item $f$ is a {\em fiberwise surjection}; that is, the restriction 
    \[\left.f\right|_{\pi^{-1}(w)} \colon \pi^{-1}(w)\to (\pi')^{-1}(g(w))\] is onto for each point $w$ in $\mathfrak{F}_0$. \label{kbc}
\end{enumerate}
\end{enumerate}
\end{defn}

Clearly, Kripke bundles and morphisms between them form a category, which we denote by $\mathsf{KBn}$. 
The composition in $\mathsf{KBn}$ of two composable morphisms $(f,g)$ and $(h,k)$ is given by $(h,k)\circ(f,g)=(h\circ f, k\circ g)$ and identity morphisms are given by $\mathbbm{1}_\mathfrak{B}=(\mathbbm{1}_\mathfrak{F},\mathbbm{1}_{\mathfrak{F}_0})$, where $\mathfrak{B}=(\mathfrak{F},\pi,\mathfrak{F}_0)$ is an arbitrary Kripke bundle and $\mathbbm{1}_\mathfrak{F},\mathbbm{1}_{\mathfrak{F}_0}$ are the respective identity functions.   

\begin{rmk}\label{morphisms}
\begin{enumerate}
    \item[]
    
    \item The fiberwise surjection requirement in \tricref{kb}{kb2}{kbc} is equivalent to a more recognizable condition: 
    for any $x\in X$ and $y\in X'$, 
    $$\pi'(f(x)) = \pi'(y) \implies (\exists z \in X)(\pi(z)=\pi(x)\text{ and } f(z)=y).$$ {\label{fiberwise_onto}}
    \item Recall that isomorphisms of Kripke frames are bijections $f\colon(X,R)\to
    (X',R')$ that preserve and reflect the relations $R$ and $R'$   (that is, for any $x,y\in X$, we have 
    $x\mathrel{R}y\iff f(x)\mathrel{R'}f(y)$).  Isomorphisms in $\mathsf{KBn}$ are 
    pairs $(f,g)$ for which both $f$ and $g$ are isomorphisms of Kripke frames. {\label{b-iso}}
\end{enumerate}
\end{rmk}

A philosophical intuition behind how Kripke bundles generalzie predicate Kripke frames can be briefly outlined as follows. For a Kripke bundle  $\pi\colon(X,R)\twoheadrightarrow(X_0,R_0)$, we view $X_0$ as the set of ``possible worlds'', and 
$R_0$ as the 
accessibility relation between 
these possible worlds. For each ${w\in X_0}$, 
$\pi^{-1}(w)$ is the set of ``individuals'' residing in the ``world'' $w$.
If $w\mathrel{R_0}u$ and $a\in \pi^{-1}(w)$, then 
the elements of the set $R[a]\cap\pi^{-1}(u)$ are 
the ``inheritors'' of the individual $a$ in the ``world'' $u$, and represent 
different ways in which an individual can manifest when transitioning to an accessible world. With this intuition, 
we next recall how to formally interpret $\text{Form}(\mathcal{L}_Q)$ in a Kripke bundle. To this end, we denote the powerset operator by $\wp$.

\begin{defn} \cite[Def.~3.2.4]{GSS}\label{valuations}
A {\em valuation} on a Kripke bundle $\mathfrak{B}=(\mathfrak{F},\pi,\mathfrak{F}_0)$ is a function $I$ that assigns to each $m$-ary predicate $P$ an element of $\prod\limits_{w\in X_0} \wp([\pi^{-1}(w)]^m)$. 
We call $\mathfrak M=(\mathfrak{B},I)$ a {\em model} based on $\mathfrak{B}$. 
\end{defn}

We present valuations as $X_0$-indexed interpretations of the predicates; this means that  $I(P)=\left(I_w(P)\right)_{w\in X_0}$, where $I_w(P)\subseteq[\pi^{-1}(w)]^m$ for an $m$-ary predicate $P$. 


%

Following \cite{GSS}, at each possible world $w\in X_0$, we only interpret $\pi^{-1}(w)$-sentences (these are formulae with no free variables) obtained  by enriching the language $\mathcal{L}_Q$ with a set of constants corresponding to the individuals from $\pi^{-1}(w)$,
and define satisfaction in a model $\mathfrak M = (\mathfrak{B},I)$ at a world $w\in X_0$ inductively as follows:

\begin{center}
\begin{tabular}{l c l}
     $w\models_I P(\Vec{a})$ & $\iff$ & $\Vec{a}\in I_w(P)$;\\
     $w\models_I \varphi \vee \psi$ & $\iff$ & $w\models_I \varphi \text{ or } w\models_I \psi$;\\
     $w\models_I\neg\varphi$ & $\iff$ & $w\not\models_I \varphi$;\\
     $w\models_I \exists x \varphi$ & $\iff$ & $w\models_I \varphi[a/x]$ for some $a\in\pi^{-1}(w)$;\\
     $w\models_I \Diamond\varphi[\Vec{a}/\Vec{x}]$ & $\iff$ & $v\models_I \varphi[\Vec{b}/\Vec{x}]$ for some $v\in R_0[w]$ \\
      & & and $b_i\in R[a_i]\cap\pi^{-1}(v)$,\\ 
      & & where $a_i=a_j$ implies $b_i=b_j$.
\end{tabular}
\end{center}

For a given formula $\varphi$, we write $\mathfrak M\models \varphi$ provided $w\models_I\forall\Vec{x}\varphi$ for all $w\in X_0$, where $\forall\Vec{x}\varphi$ denotes the universal closure of $\varphi$. Further, we write $\mathfrak{B}\models\varphi$ provided $\mathfrak M\models\varphi$ for any model $\mathfrak M$ based on $\mathfrak{B}$. When $\mathfrak{B}\models \varphi$, we say that $\varphi$ is {\em valid} on $\mathfrak{B}$.


\begin{rmk}{\label{inheritors}}
    In the $\Diamond$-clause above, 
    the additional provision that $a_i=a_j$ 
    implies $b_i=b_j$ 
    is motivated by the fact that individuals in domains may not have unique inheritors in an accessible world, so each occurrence of an individual $a$ in a formula must be \emph{uniformly} replaced by an inheritor $b$. If we interpret $\Diamond\varphi[a/x]$ without any restriction, 
    theorems of \textbf{QK} may become invalid; see \cite[Sec.~5.1]{GSS} for details.
\end{rmk} 

In the propositional case, the set of formulae valid on a frame constitutes a normal modal logic. However, for Kripke bundles, the set of valid formulae is not a predicate logic since it may not be closed under uniform substitution. For example, consider the Kripke bundle $\pi \colon (X,R) \twoheadrightarrow (X_0,R_0)$, where $X=\{a,b\}$, $X_0=\{w\}$, $R=\{(a,a),(b,b),(a,b)\}$, and $R_0=\{(w,w)\}$. The formula $\Diamond p \to p$ is valid on the Kripke bundle, while the substitution instance $\Diamond P(x) \to P(x)$ is not: consider an interpretation $I$ where $I_w(P)=\{b\}$ (this set is highlighted in red). Then in the resulting model $\mathfrak{M}$, we have $w\models_I\Diamond P(a)$, while $w\not\models_I P(a)$. 

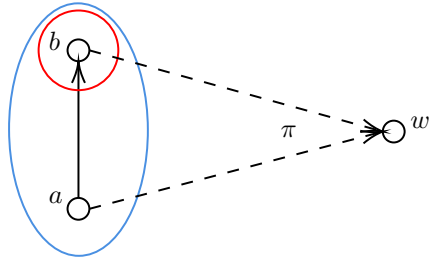
\begin{figure}[ht]
\centering
\tikzset{every picture/.style={line width=0.75pt}} 
\begin{tikzpicture}[x=0.75pt,y=0.75pt,yscale=-1,xscale=1,scale=0.8]
\draw  [color={rgb, 255:red, 74; green, 144; blue, 226 }  ,draw opacity=1 ] (165.19,91.78) .. controls (184.52,91.68) and (200.33,120.14) .. (200.5,155.33) .. controls (200.68,190.52) and (185.14,219.13) .. (165.81,219.22) .. controls (146.48,219.32) and (130.67,190.86) .. (130.5,155.67) .. controls (130.32,120.48) and (145.86,91.87) .. (165.19,91.78) -- cycle ;
\draw   (160,115.5) .. controls (160,112.46) and (162.46,110) .. (165.5,110) .. controls (168.54,110) and (171,112.46) .. (171,115.5) .. controls (171,118.54) and (168.54,121) .. (165.5,121) .. controls (162.46,121) and (160,118.54) .. (160,115.5) -- cycle ;
\draw   (160,195.5) .. controls (160,192.46) and (162.46,190) .. (165.5,190) .. controls (168.54,190) and (171,192.46) .. (171,195.5) .. controls (171,198.54) and (168.54,201) .. (165.5,201) .. controls (162.46,201) and (160,198.54) .. (160,195.5) -- cycle ;
\draw    (165.5,190) -- (165.5,123) ;
\draw [shift={(165.5,121)}, rotate = 90] [color={rgb, 255:red, 0; green, 0; blue, 0 }  ][line width=0.75]    (10.93,-3.29) .. controls (6.95,-1.4) and (3.31,-0.3) .. (0,0) .. controls (3.31,0.3) and (6.95,1.4) .. (10.93,3.29)   ;
\draw   (319,156.5) .. controls (319,153.46) and (321.46,151) .. (324.5,151) .. controls (327.54,151) and (330,153.46) .. (330,156.5) .. controls (330,159.54) and (327.54,162) .. (324.5,162) .. controls (321.46,162) and (319,159.54) .. (319,156.5) -- cycle ;
\draw  [color={rgb, 255:red, 255; green, 0; blue, 0 }  ,draw opacity=1 ] (145.4,115.5) .. controls (145.4,104.4) and (154.4,95.4) .. (165.5,95.4) .. controls (176.6,95.4) and (185.6,104.4) .. (185.6,115.5) .. controls (185.6,126.6) and (176.6,135.6) .. (165.5,135.6) .. controls (154.4,135.6) and (145.4,126.6) .. (145.4,115.5) -- cycle ;
\draw  [dash pattern={on 4.5pt off 4.5pt}]  (171,115.5) -- (317.07,155.97) ;
\draw [shift={(319,156.5)}, rotate = 195.48] [color={rgb, 255:red, 0; green, 0; blue, 0 }  ][line width=0.75]    (10.93,-3.29) .. controls (6.95,-1.4) and (3.31,-0.3) .. (0,0) .. controls (3.31,0.3) and (6.95,1.4) .. (10.93,3.29)   ;
\draw  [dash pattern={on 4.5pt off 4.5pt}]  (171,195.5) -- (317.07,157.01) ;
\draw [shift={(319,156.5)}, rotate = 165.24] [color={rgb, 255:red, 0; green, 0; blue, 0 }  ][line width=0.75]    (10.93,-3.29) .. controls (6.95,-1.4) and (3.31,-0.3) .. (0,0) .. controls (3.31,0.3) and (6.95,1.4) .. (10.93,3.29)   ;
\draw (149,185.4) node [anchor=north west][inner sep=0.75pt]    {$a$};
\draw (149,104.4) node [anchor=north west][inner sep=0.75pt]    {$b$};
\draw (331.5,147.6) node [anchor=north west][inner sep=0.75pt]    {$w$};
\draw (266,152) node [anchor=north west][inner sep=0.75pt]    {$\pi $};
\end{tikzpicture}
\end{figure}

Because of this, we use the notion of strong validity (see, e.g., \cite[Def.~5.2.11]{GSS}):

\begin{defn}\ \label{svalid} 
\begin{enumerate}
    \item A formula $\varphi$ is {\em strongly valid} in a Kripke bundle $\mathfrak{B}$ if every substitution instance of $\varphi$ is valid on $\mathfrak{B}$. We denote this by $\mathfrak{B}\models^+\hspace{-3pt}\varphi$.{\label{svalid1}}
    \item For a set $\Phi$ of $\mathcal{L}_Q$-formulae, we let $\mathfrak{B}\models^+\hspace{-3pt}\Phi$ indicate that $\mathfrak{B}\models^+\hspace{-3pt}\varphi$ for all $\varphi\in \Phi$. If $\mathfrak{B}\models^+\hspace{-3pt}\Phi$, we say that $\Phi$ is \emph{strongly sound} with respect to $\mathfrak{B}$.{\label{svalid2}}
    \item For a class $\mathsf{C}$ of Kripke bundles, $\Phi$ is said to be {\em strongly sound} with respect to $\mathsf{C}$ if $\mathfrak{B}\models^+\hspace{-3pt}\Phi$ for each $\mathfrak{B}\in\mathsf{C}$. {\label{svalid3}}
\end{enumerate}
\end{defn}

\begin{prop}{\label{sv}} \cite[Prop.~5.2.12]{GSS}
    The set of strongly valid formulae in a Kripke bundle $\mathfrak{B}$ forms a 
    pm-logic.
\end{prop}

    For a Kripke bundle $\mathfrak B$, we let $\mathbf{L}(\mathfrak{B})$ denote the pm-logic of strongly valid formulae in $\mathfrak B$, and for a class $\mathsf{C}$ of Kripke bundles, we let $\mathbf{L}(\mathsf{C})=\bigcap_{\mathfrak{B}\in\mathsf{C}}\mathbf{L}(\mathfrak{B})$. We call
    a pm-logic $\textbf{Q}$ {\em Kripke bundle complete} if 
    there is a class $\mathsf{C}$ of Kripke bundles such that $\mathbf{Q}=\mathbf{L}(\mathsf{C})$. Kripke bundle completeness provides a proper generalization of Kripke completeness (see, e.g., \cite[p.~113]{SS}). Therefore, since
    $\textbf{QK}$ and $\textbf{QS4}$ are Kripke complete (see, e.g., \cite[Thm.~15.3]{H&G}), they are also 
    Kripke bundle complete. 
    However, there exist pm-logics that are Kripke bundle incomplete. For example, $\textbf{QGrz}$ is such
    \cite{isoda, nagaoka}. 




%


\section{Monadic modal logic}{\label{monadic}}

Let $\mathcal L$ be a propositional modal language with one modality $\Diamond$, and let $\mathcal{L}_\exists$ be the extension of $\mathcal L$ with another modality $\exists$. As usual, $\Box$ abbreviates $\lnot\Diamond\lnot$ and $\forall$ abbreviates $\lnot\exists\lnot$. We use $\text{Form}(\mathcal{L}_\exists)$ to denote the set of all formulae in $\mathcal{L}_\exists$. 

For each modal logic $\textbf L$ in $\mathcal L$, let $\textbf{ML}$ be the least monadic extension of $\textbf L$; that is, $\textbf{ML}$ is the least set of formulae of $\mathcal L_\exists$ containing
\begin{itemize}
    \item the theorems of $\textbf L$,
    \item the \textbf{S5} axioms for $\exists$,
    \item the connecting axiom $\exists\Diamond p\to \Diamond\exists p$,
\end{itemize}
and closed under 
\begin{center}

\setlength{\tabcolsep}{6pt}
\renewcommand{\arraystretch}{2}
\begin{tabular}{l l}
     $\dfrac{\varphi,\varphi\to\psi}{\psi}$ & (Modus Ponens),\\ 
     $\dfrac{\varphi}{\Box \varphi}$ & ($\Box$-necessitation),\\
     $\dfrac{\varphi}{\forall\varphi}$ & ($\forall$-necessitation),\\
     $\dfrac{\varphi}{\varphi[\psi/p]}$ & (substitution).
\end{tabular}
\end{center}


In particular, \textbf{MK} is the least monadic extension of $\textbf K$, $\textbf{MS4}$ is the least monadic extension of $\textbf{S4}$, etc.

\begin{defn} \cite[Def.~2.13]{GBm}
A {\em monadic modal logic} or simply an {\em mm-logic} is any set of formulae of $\mathcal L_\exists$ containing \textbf{MK} and closed under the above inference rules. 
\end{defn}



\begin{rmk}
Monadic modal logics can alternatively be viewed as semicommutators of propositional modal logics with $\textbf{S5}$ (see, e.g., \cite[Def.~2.14]{semicomm}), and have been studied in the context of \emph{expanding relativized products} of Kripke frames (see, e.g., \cite[p.~432]{mdim}). As such, the relationship between semicommutators and predicate modal logics has been a subject of recent investigation (see, e.g., \cite[Sec.~3]{semicomm}).
\end{rmk}

A Kripke-style semantics for \textbf{MS4} was introduced by Esakia \cite{esa}. Our treatment below follows that of \cite{GBm}, where the semantics of \cite{esa} is generalized to all mm-logics.

\begin{defn}\ \label{mk_stuff}
\begin{enumerate}
    { \item A} {\em monadic Kripke frame} or simply an $\textbf{MK}$-{\em frame} is a triple $\mathfrak F = (X,R,E)$ such that $(X,R)$ is a Kripke frame,  \label{mk_1}
$E$ is an equivalence relation on $X$, and the following commutativity condition is satisfied:
\[
(\forall x,y,z\in X)(x\mathrel{E}y \mbox{ and } y\mathrel{R}z \Longrightarrow \exists u \in X : x\mathrel{R}u \mbox{ and } u\mathrel{E}z).
\]
\[\begin{tikzcd}
	u && z \\
	\\
	x && y
	\arrow["E"{description}, dashed, tail reversed, from=1-1, to=1-3]
	\arrow["R"{description}, dashed, from=3-1, to=1-1]
	\arrow["E"{description}, tail reversed, from=3-1, to=3-3]
	\arrow["R"{description}, from=3-3, to=1-3]
\end{tikzcd}\]

\item A {\em morphism} between two $\textbf{MK}$-frames $\mathfrak F = (X,R,E)$ and $\mathfrak F' = (X',R',E')$ is a map $f\colon X\to X'$ that is a p-morphism with respect to both $R$ and $E$.
\end{enumerate}
\end{defn}

Clearly, {\textbf{MK}}-frames and morphisms between them form a category, which we denote by $\mathsf{MKF}$. Composition in $\mathsf{MKF}$ is usual function composition and identity morphisms are identity functions $\mathbbm{1}_\mathfrak{F}$.

\begin{rmk}{\label{iso}}
    Isomorphisms of $\textbf{MK}$-frames are bijections $f \colon (X,R,E) \twoheadrightarrowtail (X',R',E')$ that preserve and reflect both $R$ and $E$. 
\end{rmk}

We interpret $\text{Form}(\mathcal L_\exists)$ in an $\textbf{MK}$-frame $\mathfrak F = (X,R,E)$ by interpreting $\Diamond$ using $R$ and $\exists$ using $E$. More precisely, 
a {\em valuation} on $\mathfrak F$ is a function $v$ that assigns to each propositional letter $p$ a subset $v(p)$ of $X$. Then, for each $x\in X$, we have:

\begin{center}
\begin{tabular}{l c l}
    $x\models_v p$ & $\iff$ & $x\in v(p)$; \\
    $x\models_v \varphi \vee \psi$ & $\iff$ & $x\models_v \varphi \text{ or } x\models_v \psi$;\\
    $x\models_v\neg\varphi$ & $\iff$ & $x\not\models_v \varphi$;\\
    $x\models_v \Diamond\varphi$ & $\iff$ & $y\models_v \varphi$ for some $y\in R[x]$;\\
    $x\models_v \exists\varphi$ & $\iff$ & $y\models_v \varphi$ for some $y\in E[x]$.
\end{tabular}
\end{center}

The pair $\mathfrak M=(\mathfrak F, v)$ is called a {\em model} based on $\mathfrak F$. For any $\mathcal{L}_\exists$-formula $\varphi$, we let $\mathfrak{M}\models\varphi$ indicate that $x\models_v\varphi$ for all $x\in X$. We write $\mathfrak{F}\models\varphi$ if $\mathfrak{M}\models\varphi$ for any model $\mathfrak{M}$ based on $\mathfrak{F}$. If $\mathfrak{F}\models\varphi$, we say that $\varphi$ is {\em valid} on $\mathfrak{F}$.

\begin{defn}{\label{M-frame}}
    Let $\textbf M$ be an mm-logic. We call an \textbf{MK}-frame $\mathfrak F = (X,R,E)$ an {\em {\bf M}-frame} provided each theorem of $\textbf M$ is valid on $\mathfrak F$.
\end{defn}
For example, $\mathfrak F$ is an $\textbf{MS4}$-frame provided $R$ is a preorder, $\mathfrak F$ is an $\textbf{MGrz}$-frame provided $R$ is a Noetherian partial order, etc.~(see, e.g., \cite[Sec.~3.8]{CZ}). 

%
%
 By \cite[Thm.~9.10]{mdim}, $\textbf{MK}$ and $\textbf{MS4}$ 
can be realized as expanding relativized products. Therefore, 
by \cite[Thms.~9.10,9.12]{mdim}, they can be embedded into the product logics $\textbf{K}\times\textbf{S5}$ and $\textbf{S4}\times\textbf{S5}$, respectively. By \cite[Thm.~5.27]{mdim}, the latter two logics have the fmp. Consequently, both $\textbf{MK}$ and $\textbf{MS4}$ also
have the fmp 
(for an algebraic proof of the fmp for $\textbf{MS4}$, see \cite[Thm.~5.4]{GLt}). We thus arrive at the following:

\begin{thm} \label{thm: completeness}
\begin{enumerate}
    \item[] 
    \item 
    \textbf{MK} is the logic of finite \textbf{MK}-frames. {\label{MK}}
    \item 
    \textbf{MS4} is the logic of finite \textbf{MS4}-frames. {\label{MS4}}
\end{enumerate}
\end{thm}

Our main result establishes that the above also holds for $\textbf{MGrz}$. 




\section{Monadic Kripke frames and Kripke bundles} {\label{fnb}}

In this section, we demonstrate a natural correspondence between the categories of \textbf{MK}-frames and Kripke bundles. This generalizes a similar correspondence for the semantics of intuitionistic monadic and predicate logics \cite{szk,BV}.

\begin{defn}{\label{Q_rel}}
    Let $\mathfrak{F}=(X,R,E)$ be an \textbf{MK}-frame.
    \begin{enumerate}
        \item Define $Q$ on $X$ to be the composite $Q=E\circ R$; that is, $x \mathrel{Q} y \iff \exists z \in X : x\mathrel{R}z \text{ and } z\mathrel{E}y$.{\label{Qrel}}
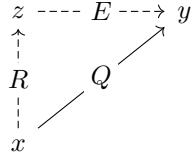
    \[\begin{tikzcd}
	z && y \\
	\\
	x
	\arrow["E"{description}, dashed, from=1-1, to=1-3]
	\arrow["R"{description}, dashed, from=3-1, to=1-1]
	\arrow["Q"{description}, from=3-1, to=1-3]
    \end{tikzcd}\]
        \item The \emph{$E$-skeleton} of $\mathfrak{F}$ is the frame $\mathfrak{F}_0=(X_0,R_0)$, where $X_0$ is the quotient of $X$ by $E$. As usual, we denote the elements of $X_0$ by $[x]$, where $x\in X$, and define $R_0$ by 
    \begin{equation*}
        [x]\mathrel{R_0}[y]\iff x\mathrel{Q}y
    \end{equation*}{\label{Qrel1}}
    for all $x,y\in X$.
        %
        %
    \end{enumerate}
\end{defn}
\begin{rmk}{\label{Q_0}}
 As a consequence of \crefdefpart{mk_stuff}{mk_1}, we obtain that 
\[
[x]\mathrel{R_0}[y]\iff x'\mathrel{R}y' \mbox{ for some } x' \in [x] \mbox{ and } y' \in [y]
\]
for all $x,y \in X$.

\end{rmk}

\begin{prop}{\label{functor1}}
    There is a functor $\mathscr{B}\colon\mathsf{MKF}\to\mathsf{KBn}$. 
\end{prop}
\begin{proof}
    For an \textbf{MK}-frame $\mathfrak F = (X,R,E)$, let $\mathscr{B}(\mathfrak F) =((X,R),\pi_X,(X_0,R_0))$, where $\pi_X\colon X\twoheadrightarrow X_0$ is the quotient map $\pi_X(x) = [x]$. Clearly, $\pi_X$ is onto and $x \mathrel{R} y$ implies $\pi_X(x) \mathrel{R}_0 \pi(y)$ for all $x,y \in X$. To see that it is a p-morphism, suppose 
$[x]\mathrel{R_0}[y]$. {Then}
$x\mathrel{Q}y$, so $x\mathrel{R}z$ and $z\mathrel{E}y$ for some $z\in X$. Thus, $x\mathrel{R}z$ and 
$\pi_X(z)=[y]$. This proves that $\mathscr{B}(X,R,E)$ is a Kripke bundle.

For a morphism of \textbf{MK}-frames $f\colon (X,R,E)\to (X',R',E')$, let $\mathscr{B}(f)$ be the pair $(f,f_0)$, where $f_0\colon X_0\to X_0'$ is the map $f_0([x]) = [f(x)]$. We show that $\mathscr{B}(f)$ is a morphism of Kripke bundles. The map 
$f_0\colon X_0\to X_0'$ is well defined since $[x]=[y]$ implies $x\mathrel{E}y$, so 
$f(x)\mathrel{E}'f(y)$, and hence $[f(x)]=[f(y)]$. It is 
a p-morphism because $f$ is a p-morphism.
Moreover, due to the definition of $f_0$, the following diagram commutes.

\adjustbox{scale=1, center}{
    \begin{tikzcd}
    {(X,R)} && {(X_0,R_0)} \\
    \\
    {(X',R')} && {(X_0',R_0')}
    \arrow["\pi_X", two heads, from=1-1, to=1-3]
    \arrow["f"', from=1-1, to=3-1]
    \arrow["f_0", from=1-3, to=3-3]
    \arrow["{\pi_{X'}}"', two heads, from=3-1, to=3-3]
    \end{tikzcd}}

To see that $f$ is a fiberwise surjection,  it is sufficient to verify  \crefdefpart{morphisms}{fiberwise_onto}. 
Let $\pi_{X'}(f(x))=\pi_{X'}(y)$. Then $f(x)\mathrel{E}'y$. Therefore, there is $z\in X$ such that $x\mathrel{E}z$ and $f(z)=y$. Thus, $\pi_X(x)=\pi_X(z)$ and $f(z)=y$.
Finally, it is routine to check that $\mathscr{B}$ preserves identities and compositions, hence $\mathscr{B}\colon\mathsf{MKF}\to\mathsf{KBn}$ is a functor. 
\end{proof}

\begin{prop}\label{functor2}
    There is a functor  $\mathscr{F}\colon\mathsf{KBn}\to\mathsf{MKF}$.
\end{prop}

\begin{proof}
For a Kripke bundle $\mathfrak B = ((X,R),\pi,(X_0,R_0))$, let $\mathscr{F}(\mathfrak B)=(X,R,E_{\pi})$, where 
\[
x\mathrel{E_{\pi}}y \iff \pi(x)=\pi(y).
\]

We show that $\mathscr{F}(\mathfrak B)$
is an \textbf{MK}-frame.
Since $E_\pi$ is clearly an equivalence on $X$, we only need to verify that commutativity holds for $R$ and $E_\pi$. Suppose $x\mathrel{E_\pi} y$ and $y\mathrel{R}z$. Then
$\pi(x)=\pi(y)$ and $\pi(y)\mathrel{R_0}\pi(z)$, hence $\pi(x)\mathrel{R_0}\pi(z)$. This implies that there is $u\in X$ with $x\mathrel{R}u$ and $\pi(u)=\pi(z)$. Therefore, $x\mathrel{R}u$ and $u\mathrel{E_\pi} z$, as required.

For a Kripke bundle morphism $(f,g)\colon((X,R),\pi,(X_0,R_0))\to ((X',R'),\pi',(X_0',R_0'))$, let ${\mathscr{F}(f,g)=f}$. We show that $\mathscr{F}(f,g)=f$
is a morphism of \textbf{MK}-frames.
Since $f$
is a p-morphism with respect to $R$, we only need to show that $f$ is a p-morphism with respect to $E_\pi$. First, let $x\mathrel{E_\pi} y$.
Then $\pi(x)=\pi(y)$, so $g(\pi(x))=g(\pi(y))$. Since $g\circ\pi=\pi'\circ f$, we obtain $\pi'(f(x))=\pi'(f(y))$, which gives $f(x)\mathrel{E_{\pi'}}f(y)$. 
Next, let
$f(x)\mathrel{E_{\pi'}}y$, so $\pi'(f(x))=\pi'(y)$. Since $(f,g)$ is a morphism of Kripke bundles, by \tricref{kb}{kb2}{kbc} 
there is $z\in X$ with $\pi(x)=\pi(z)$ and $f(z)=y$. Thus,
$x\mathrel{E_\pi} z$ and $f(z)=y$, yielding that $f$ is a morphism of \textbf{MK}-frames.
Finally, it is routine to check that $\mathscr F$ 
preserves identities and compositions, hence $\mathscr{F}\colon\mathsf{KBn}\to\mathsf{MKF}$ is a functor.
\end{proof}

\begin{thm}\label{equivalence}
    The categories $\mathsf{MKF}$ and $\mathsf{KBn}$ are equivalent.
\end{thm}
\begin{proof}
 We show that the functors  $\mathscr{B}\colon\mathsf{MKF}\to\mathsf{KBn}$ and $\mathscr{F}\colon\mathsf{KBn}\to\mathsf{MKF}$ defined in \Cref{functor1,functor2} yield the desired equivalence of categories. For this, 
we first define a natural isomorphism  $\eta\colon\mathbbm{1}_{\mathsf{MKF}}\to \mathscr{F}\circ\mathscr{B}
$ as follows. 
Let $\mathfrak{F}=(X,R,E)$ be an \textbf{MK}-frame. Then $(\mathscr{F}\circ\mathscr{B})(\mathfrak{F})=(X,R,E_{\pi_X})$, and we define 
the component $\eta_{\mathfrak{F}}\colon \mathfrak{F}\to(\mathscr{F}\circ\mathscr{B})(\mathfrak{F})$ to be the identity map. 
We then have $E=E_{\pi_X}$ since
\begin{equation*}
    x\mathrel{E_{\pi_X}}y\iff\pi_X(x)=\pi_X(y)\iff x\mathrel{E}y,
\end{equation*} 
 so it is clear that $\eta$ is a natural isomorphism.

 We next define a natural isomorphism 
$
\varepsilon\colon \mathbbm{1}_{\mathsf{KBn}}\to \mathscr{B}\circ\mathscr{F}$  
as follows. 
Let $\mathfrak{B}=((X,R),\pi,(X_0,R_0))$ be a Kripke bundle. Then $(\mathscr{B}\circ\mathscr{F})(\mathfrak{B})=((X,R),\pi_X,(X/E_\pi, R_\pi))$, where $(X/E_\pi,R_\pi)$ is the $E$-skeleton of $\mathscr{F}(\mathfrak{B})$, and $\pi_X(x) = [x]_{E_\pi}$.
Define the component $\varepsilon_\mathfrak{B}\colon \mathfrak{B}\to (\mathscr{B}\circ\mathscr{F})(\mathfrak{B})$ by setting $\varepsilon_\mathfrak{B}=(\mathbbm{1}_X, p_0)$, where ${p_0\colon X_0\to X/E_\pi}$ is given by $p_0(w) = \pi^{-1}(w)$. 
The map $p_0$ is a well-defined bijection that preserves and reflects the relations in $(X_0,R_0)$ and $(X/E_\pi,R_\pi)$, hence it is an isomorphism of Kripke frames (see \crefdefpart{morphisms}{b-iso}). 
It follows that $\pi_X\circ\mathbbm{1}_X=p_0\circ\pi$. 
Thus, $(\mathbbm{1}_X,p_0)$ is an isomorphism of Kripke bundles. We next verify the naturality of $\varepsilon$. 
If $(f,g)\colon((X,R),\pi,(X_0,R_0))\to((X',R'),\pi',(X_0',R_0'))$ is a morphism of Kripke bundles, then we obtain the diagram shown below.

We show that $(\mathbbm{1}_{X'},p_{0}')\circ(f,g)=(f,f_0)\circ(\mathbbm{1}_X,p_0)$; that is,  $\mathbbm{1}_{X'}\circ f=f\circ\mathbbm{1}_X$ and $p_{0}'\circ g = f_0 \circ p_0$. 
The first equality is obvious and the second 
follows from the definitions of $p_0$, $p_0'$ and $f_0$, and the fact that $f$ is a fiberwise surjection.
Therefore, $\varepsilon$ is a natural isomorphism. Consequently, $\mathsf{MKF}$ and $\mathsf{KBn}$ are equivalent (see, e.g., \cite[Sec.~IV.4]{maclane}).

\adjustbox{scale=0.9, center}{
        \begin{tikzcd}
	& {(X,R)} &&&& {(X,R)} \\
	{(X_0,R_0)} &&&& {(X/E_\pi,R_\pi)} \\
	\\
	\\
	& {(X',R')} &&&& {(X',R')} \\
	{(X_0',R_0')} &&&& {(X'/E_{\pi'},R_{\pi'})}
	\arrow["{{\mathbbm{1}_X}}", from=1-2, to=1-6]
	\arrow["\pi"', two heads, from=1-2, to=2-1]
	\arrow["f", from=1-2, to=5-2]
	\arrow["{{\pi_X}}", two heads, from=1-6, to=2-5]
	\arrow["f", from=1-6, to=5-6]
	\arrow["{{p_0}}", from=2-1, to=2-5]
	\arrow["g"', from=2-1, to=6-1]
	\arrow["{{f_0}}"', from=2-5, to=6-5]
	\arrow["{{\mathbbm{1}_{X'}}}"', from=5-2, to=5-6]
	\arrow["{{\pi'}}", two heads, from=5-2, to=6-1]
	\arrow["{{\pi_{X'}}}", two heads, from=5-6, to=6-5]
	\arrow["{{p_{0}'}}"', from=6-1, to=6-5]
\end{tikzcd}}
\end{proof}

\section{Correspondence between mm-logics and pm-logics}{\label{correspondence}}

There is a close connection between mm-logics and pm-logics. Indeed, in many cases, mm-logics axiomatize the monadic fragment of the corresponding pm-logics; see, for example, \cite{FS,esa,GBm}.
The correspondence we discuss here extends a similar correspondence between monadic and predicate intuitionistic logics \cite{prior,bull,ono,onoszk,szk,BV1}.
To make this more precise,  
we define a translation of formulae of $\mathcal{L}_\exists$ to formulae of
$\mathcal{L}_Q$ as follows. Fix an individual variable $x$. Following \cite{ono},
with each propositional letter $p$, we associate a unique monadic predicate $p^*$, and 
        define the translation $(-)^t\colon {\text{Form}}(\mathcal{L}_\exists)\to{\text{Form}}(\mathcal{L}_Q)$ recursively as follows:
        \begin{enumerate}
            \item $p^t = p^*(x)$;
                \item $(\varphi\vee\psi)^t = \varphi^t \vee \psi^t$;
            \item $(\neg\varphi)^t = \neg \varphi^t$;
            \item $(\Diamond\varphi)^t = \Diamond \varphi^t$;
            \item $(\exists\varphi)^t = \exists x \varphi^t$.
        \end{enumerate}

Note that for any $\mathcal{L}_\exists$-formula $\varphi$, $x$ is the only variable that can occur freely in $\varphi^t$. Therefore, $(\exists\varphi)^t$ is always a sentence. 
Since $\varphi^t$ may only have one free variable, interpreting formulae of the form $(\Diamond\varphi)^t[a/x]$ requires no additional care with inheritors (see \Cref{inheritors}).

As is customary, we identify mm-logics and pm-logics with their deductively closed sets of theorems. For a logic $\textbf{L}$ and a formula $\varphi$, we write  $\textbf{L}\vdash\varphi$ to mean that $\varphi$ is a theorem of $\mathbf{L}$. 

\begin{defn}
    Given an mm-logic $\textbf{M}$ and a pm-logic $\textbf{Q}$, we say that $\textbf{M}$ is the {\em monadic fragment} of $\textbf{Q}$ provided for each $\mathcal{L}_\exists$-formula {$\varphi$,}
    \begin{equation*}
        \textbf{M} \vdash \varphi \iff \textbf{Q} \vdash\varphi^t.
    \end{equation*}
\end{defn}

For an \textbf{MK}-frame $\mathfrak{F}=(X,R,E)$, recall that $\mathscr{B}(\mathfrak{F})=((X,R),\pi_X,(X_0,R_0))$, where $\pi_X\colon X\twoheadrightarrow X_0$ is the quotient map $\pi_X(x)=[x]$ (see the proof of \Cref{functor1}). Throughout this section, we simplify the notation and drop the subscript from $\pi_X$. 

\begin{lem}\label{models}
Let $\mathfrak{F}=(X,R,E)$ be an \textbf{MK}-frame.
\begin{enumerate}
    \item For each valuation $v$ on $\mathfrak F$ there is a valuation $I$ on $\mathscr{B}(\mathfrak{F})$ 
    such that for each $\mathcal{L}_\exists$-formula $\varphi$ and $a\in X$,
    \begin{equation*}
            \pi(a) \models_I\varphi^t[a/x] \iff a\models_v\varphi.
    \end{equation*}{\label{t1}}\vspace{-3.5mm}
    \item For each valuation $I$ on $\mathscr{B}(\mathfrak{F})$ there is a valuation $v$ on $\mathfrak F$ such that for each $\mathcal{L}_\exists$-formula $\varphi$ and $a\in X$,
    \begin{equation*}
            \pi(a) \models_I \varphi^t[a/x] \iff
            a \models_v \varphi.
    \end{equation*}{\label{t2}}
\end{enumerate}
\end{lem}
\vspace{-7mm}
\begin{proof}
(\labelcref{t1}) Define a valuation $I$ on $\mathscr{B}(\mathfrak{F})$ for which
\begin{equation*}
    a\in I_{\pi(a)}(p^*) \iff
    a\models_v p 
\end{equation*}
(equivalently, $I_{\pi(a)}(p^*) = v(p) \cap E[a]$). 

Our proof is by induction on the complexity of the formula $\varphi$. 
The base case follows from the definition of $I$. The induction hypothesis ensures that the claim holds for formulae of the form $\psi\vee\chi$ and $\neg\psi$.

Suppose $\varphi$ is of the form $\exists\psi$. Let 
$\pi(a)\models_I(\exists\psi)^t[a/x]$. Then 
$\pi(a)\models_I\exists x\psi^t$ since $(\exists\psi)^t = \exists x\psi^t$ is a sentence. 
Hence, $\pi(a)\models_I\psi^t[b/x]$ for some $b\in\pi^{-1}(\pi(a))$. Therefore,
$a\mathrel{E}b$ 
and $\pi(b)\models_I\psi^t[b/x]$. By the induction hypothesis, $b\models_v\psi$. Thus, $a\models_v\exists\psi$.
Conversely, let 
$a\models_v\exists\psi$. Then there is $b\in X$ such that $a\mathrel{E}b$ and $b\models_v\psi$. By the induction hypothesis, 
$\pi(b)\models_I\psi^t[b/x]$. Since $\pi(a)=\pi(b)$, we have 
$\pi(a)\models_I\psi^t[b/x]$, which means that $\pi(a)\models_I\exists x\psi^t$. Hence, 
$\pi(a)\models_I(\exists\psi)^t$, and so $\pi(a)\models_I(\exists\psi)^t[a/x]$ since $(\exists\psi)^t$ is a sentence. 

Finally, suppose that $\varphi$ is of the form $\Diamond\psi$. Let  $\pi(a)\models_I(\Diamond\psi)^t[a/x]$. Then $\pi(a)\models_I\Diamond\psi^t[a/x]$. 
Therefore, 
$\pi(c)\models_I\psi^t[b/x]$ for some $b,c\in X$ with $\pi(a)R_0\pi(c)$, $b\in\pi^{-1}(\pi(c))$, and $a\mathrel{R}b$. 
Thus, $\pi(b)\models_I\psi^t[b/x]$. The induction hypothesis then gives $b\models_v\psi$. Consequently, $a\models_v\Diamond\psi$. Conversely, let  $a\models_v\Diamond\psi$. Then there is $b\in X$ such that $a\mathrel{R}b$ and $b\models_v\psi$. The induction hypothesis yields $\pi(b)\models_I\psi^t[b/x]$. Since $a\mathrel{R}b$ implies $\pi(a)R_0\pi(b)$, we obtain
$\pi(a)\models_I\Diamond\psi^t[a/x]$. Thus,
$\pi(a)\models_I(\Diamond\psi)^t[a/x]$. 

(\labelcref{t2}) Define a valuation $v$ on $\mathfrak{F}$ for which 
\begin{equation*}
    a\models_v p \iff 
    a\in I_{\pi(a)}(p^*) 
\end{equation*}
(equivalently, $v(p) = \bigcup_{a\in X} I_{\pi(a)}(p^*)$), 
and proceed as in (\labelcref{t1}). 
\end{proof}
\vspace{-3.5mm}

\begin{thm}\label{t_thm}
    Let $\mathfrak{F}=(X,R,E)$ be an \textbf{MK}-frame. For each $\mathcal{L}_\exists$-formula $\varphi$, 
    \[
    \mathfrak{F}\models\varphi \iff 
    \mathscr{B}(\mathfrak{F})\models \varphi^t.
    \]
\end{thm}
\begin{proof}
First suppose $\mathfrak{F}\models\varphi$. Let $I$ be a 
valuation on $\mathscr{B}(\mathfrak{F})$. 
By \crefdefpart{models}{t2}, 
there is a valuation $v$ on $\mathfrak F$ 
such that for each $a\in X$ we have $\pi(a)\models_I \forall x \varphi^t$ iff $a \models_v \forall\varphi$. The latter always holds since the validity of $\varphi$ implies the validity of $\forall\varphi$ in $\mathfrak F$. Thus, the former also holds, and hence $\mathscr{B}(\mathfrak{F})\models\varphi^t$.

Next suppose $\mathscr{B}(\mathfrak{F})\models\varphi^t$. Let $v$ be a valuation on $\mathfrak F$. By \crefdefpart{models}{t1}, there is a valuation $I$ on $\mathscr{B}(\mathfrak F)$ such that for each $a\in X$ we have $\pi(a)\models_I \forall x \varphi^t$ iff $a \models_v \forall\varphi$. The former always holds since $\varphi^t$ is valid on $\mathscr{B}(\mathfrak F)$. Thus, the latter also holds, which yields $a\models_v \varphi$. Hence, $\mathfrak{F}\models\varphi$.
%
\end{proof}

The following result is a generalization of \cite[Thm.~3.5]{szk}, which in turn is a generalization of \cite[Thm.~3.5]{onoszk}.


\begin{thm}{\label{translation}}
    Let $\textbf{M}$ be an mm-logic and $\textbf{Q}$ a pm-logic satisfying the following conditions:
    \begin{enumerate}
        \item For each $\mathcal{L}_\exists$-formula $\varphi$,
            \begin{equation*}
                \textbf{M} \vdash \varphi \implies \textbf{Q} \vdash \varphi^t.
            \end{equation*} {\label{trans1}} 
            \vspace{-6mm}
        \item $\textbf{M}$ is complete with respect to a class $\mathsf{C}$ of \textbf{MK}-frames. {\label{trans2}}
        \item $\textbf{Q}$ is strongly sound with respect to the class $\{\mathscr{B}(\mathfrak{F})\mid \mathfrak{F}\in\mathsf{C}\}$. {\label{trans3}}
    \end{enumerate}
    Then $\textbf{M}$ is the monadic fragment of $\textbf{Q}$.
\end{thm}
\begin{proof}
Let $\varphi$ be an $\mathcal{L}_\exists$-formula. By (\labelcref{trans1}), it is enough to show that  $\textbf{M}\not\vdash\varphi$ implies $\textbf{Q} \not\vdash\varphi^t$. Suppose $\textbf{M}\not\vdash\varphi$. By (\labelcref{trans2}), there is an \textbf{MK}-frame $\mathfrak{F}\in\mathsf{C}$ such that $\mathfrak{F}\not\models\varphi$.
By \Cref{t_thm}, $\mathscr{B}(\mathfrak{F})\not\models\varphi^t$. By (\labelcref{trans3}),  $\textbf{Q}$ is strongly sound with respect to the Kripke bundle $\mathscr{B}(\mathfrak{F})$. Thus, $\textbf{Q} \not\vdash\varphi^t$.
\end{proof}

\begin{cor}{\label{useful}}
    Let $\textbf{L}$ be a propositional modal logic satisfying the following two conditions:
    \begin{enumerate}
        \item $\textbf{ML}$ is complete with respect to a class $\mathsf{C}$ of \textbf{MK}-frames.
        \item $\textbf{QL}$ is strongly sound with respect to $\{\mathscr{B}(\mathfrak{F})\mid \mathfrak{F}\in\mathsf{C}\}$.
    \end{enumerate}
    Then $\textbf{ML}$ is the monadic fragment of $\textbf{QL}$.
\end{cor}
\begin{proof}
For each $\varphi\in\text{Form}(\mathcal{L}_\exists)$, we show that $\textbf{ML}\vdash\varphi$ implies $\textbf{QL}\vdash\varphi^t$. 
The proof is by induction on the length of the proof of $\varphi$.
If $\varphi$ is a theorem of $\textbf{L}$ or an instance 
 of the connecting axiom $\exists\Diamond q\to \Diamond\exists q$, 
 then it is straightforward to check that $\varphi^t$ is a theorem of $\textbf{QL}$. If $\varphi$ is an \textbf{S5} axiom for $\exists$ in \textbf{ML}, then it is easy to see that $\varphi^t$ is in fact a theorem of the classical predicate calculus; hence, $\varphi^t$ is a theorem of \textbf{QL}.  

If $\varphi$ is obtained by Modus Ponens, i.e., from
${\psi}$ and $\psi\to\varphi$, then $\textbf{QL}\vdash\psi^t, \psi^t\to\varphi^t$ by the induction hypothesis, so $\textbf{QL}\vdash \varphi^t$ by Modus Ponens for $\textbf{QL}$.
The argument for when $\varphi$ is obtained by $\Box$-necessitation or $\forall$-necessitation is similar.
Finally, suppose $\varphi=\psi[\alpha/q]$
is obtained by substitution. Then 
$\varphi^t=\psi^t[\alpha^t/q^*(x)]$, where $\textbf{QL}\vdash\psi^t$ by the induction hypothesis. Thus, $\textbf{QL}\vdash\varphi^t$ by uniform substitution.
Consequently, \Cref{translation} applies, yielding the result.
\end{proof}

To utilize \Cref{useful}, we must show that a predicate logic $\textbf{QL}$ is strongly sound with respect to a class of Kripke bundles, which involves checking for strong validity of formulae. For propositional formulae, 
there is a simple criterion for checking strong validity. 
Consider a Kripke bundle $\mathfrak{B}=((X,R),\pi,(X_0,R_0))$. For each $n\ge 0$, define the  $n$-fold fiberwise product $(X^n,R^n)$ as follows. If $n=0$, let $X^n=X_0$ and $R^n=R_0$; if $n=1$, let $X^n=X$ and $R^n=R$; and if $n>1$, let $X^n=\bigcup\limits_{w\in X_0}[\pi^{-1}(w)]^n$ and define $R^n$ by  
\begin{equation*}
    \Vec{x}\mathrel{R^n}\Vec{y} \iff 
    x_i\mathrel{R}y_i \text{ for each } i \text{ and } \Vec{x}\textbf{ sub }\Vec{y},
\end{equation*}

where $\textbf{sub}$ is the \emph{subordination relation} (see \cite[Def.~5.3.2]{GSS}) defined by 
\[\Vec{x}\textbf{ sub }\Vec{y} \iff 
\mbox{for each } i,j \, (x_i=x_j \implies y_i=y_j).
\]

The 
relation \textbf{sub} is used to ensure coherence for inheritors in the following sense: suppose there are $a,b,c\in X$ such that $a\mathrel{R}b$ and $a\mathrel{R}c$. In the absence of \textbf{sub}, we may have $(a,a)\mathrel{R^2}(b,c)$ in $X^2$, which is 
not allowed in Kripke bundle semantics since we only evaluate formulae of the form $\Diamond\varphi$ by replacing an individual uniformly with one inheritor (see \Cref{inheritors}). 

\begin{defn}{\label{nthlevel}}\cite[Def.~5.3.2]{GSS}
    For a Kripke bundle $\mathfrak{B}=((X,R),\pi,(X_0,R_0))$, the frame $\mathcal{X}_n=(X^n,R^n)$ is called the $n^{\text{th}}$ level of $\mathfrak{B}$, where $X^n$ and $R^n$ are defined as above.
\end{defn}

Recalling \crefdefpart{svalid}{svalid1}, we have:

\begin{prop}\cite[Prop.~5.3.7]{GSS}{\label{sound}}
    For a Kripke bundle $\mathfrak{B}=((X,R),\pi,(X_0,R_0))$ and a propositional formula $\varphi$, we have $\mathfrak{B}\models^+\hspace{-3pt}\varphi$
    iff
    $\mathcal{X}_n\models\varphi$ for each $n<\omega$.
\end{prop}
As a consequence of \Cref{useful} and \Cref{sound}, we obtain: 


\begin{thm}{\label{folk}}
    \begin{enumerate}
        \item[]
        \item $\textbf{MK}$ is the monadic fragment of $\textbf{QK}$. \label{folk1}
        \item $\textbf{MS4}$ is the monadic fragment of $\textbf{QS4}$. \label{folk2}
    \end{enumerate}
\end{thm}

\begin{proof}

(\labelcref{folk1}) By \crefdefpart{thm: completeness}{MK}, $\textbf{MK}$ is complete with respect to the class of all finite $\textbf{MK}$-frames. By \Cref{sv}, for each $\textbf{MK}$-frame $\mathfrak{F}$, we have $\mathscr{B}(\mathfrak{F})\models^+\hspace{-3pt}\textbf{QK}$. Therefore, \Cref{useful} applies.

(\labelcref{folk2}) By \crefdefpart{thm: completeness}{MS4},  $\textbf{MS4}$ is complete with respect to the class of all finite $\textbf{MS4}$-frames. Observe that for a Kripke bundle $\mathfrak{B}=((X,R),\pi,(X_0,R_0))$, we have $\mathfrak{B}\models^+\hspace{-3pt}\textbf{QS4}$ iff $(X,R)$ is a preorder. To see this, by \Cref{sound},  $\mathfrak{B}\models^+\hspace{-3pt}q\to\Diamond q,\Diamond \Diamond q\to \Diamond q$ iff $\mathcal{X}_n$ is a preorder for each $n<\omega$. Hence, if $\mathfrak{B}\models^+\hspace{-3pt}\textbf{QS4}$, then $\mathcal{X}_1=(X,R)$ must be a preorder. Conversely, if $\mathcal{X}_1$ is a preorder, it is straightforward to verify that $\mathcal{X}_n$ is a preorder for each $n<\omega$, which implies that $\mathfrak{B}\models^+\hspace{-3pt}\textbf{QS4}$.

Now, for an $\textbf{MS4}$-frame $\mathfrak{F}=(X,R,E)$, we see that in the Kripke bundle $\mathscr{B}(\mathfrak{F})=((X,R),\pi_X,(X_0,R_0))$, $(X,R)$ is indeed a preorder. Thus, \Cref{useful} applies. 
\end{proof}

\begin{rmk}
The above result is
considered to be 
folklore, but we are not aware of any proof in the literature. 
For example, the easy implication $\textbf{MS4}\vdash \varphi\implies\textbf{QS4}\vdash\varphi^t$ is shown in \cite[Thm.~8]{FS}, 
but the converse is not discussed there.
\end{rmk}

We note that \Cref{useful} can be used to prove that
\textbf{ML} is the monadic fragment of \textbf{QL} for other propositional modal logics \textbf{L}. 
In particular, we will use \Cref{useful} to prove that
$\textbf{MGrz}$ is the monadic fragment of $\textbf{QGrz}$. For this, we require $\textbf{MGrz}$ to be complete with respect to a suitable class $\mathsf{C}$ of $\textbf{MK}$-frames. We will show that $\mathsf{C}$ can be taken to be all finite $\textbf{MGrz}$-frames.


\section{Monadic \textbf{Grz}} {\label{MGrz}}

We recall that $\textbf{MGrz}$ is the least monadic extension of $\textbf{Grz}$.  
It is well known that \textbf{Grz} has the fmp (see, e.g., \cite[Thm.~5.51]{CZ}).  
This can be proved by selecting \emph{maximal} points from a descriptive frame. 
Adopting this construction to \textbf{MGrz}
requires selecting points that we term \emph{strongly maximal}. In this section, we show that we have a sufficient supply of these, and discuss some of their key properties.


We start by recalling the algebraic semantics of modal logic (see, e.g., \cite[Ch.~7]{CZ}).
A {\em modal algebra} is a pair $\mathbbm{B}=(B,\opc)$, where $B$ is a Boolean algebra and $\opc$ is a unary function on $B$ that preserves finite joins. We say that 
\begin{itemize}
    {\item $\mathbbm{B}$} is an {\em $\bf S4$-algebra} if additionally $a\leq\opc a$ and $\opc\opc a\leq \opc a$ for any $a\in B$;
    {\item $\mathbbm{B}$} is an {\em $\bf S5$-algebra} if it is an \textbf{S4}-algebra and $a\leq \neg\opc\neg\opc a$ for any $a\in B$;
    {\item $\mathbbm{B}$} is a {\em $\bf Grz$-algebra} if it is an \textbf{S4}-algebra and $a\leq \opc(a\wedge\neg\opc(\opc a\wedge\neg a))$ for any $a\in B$.
\end{itemize}

A modal formula $\varphi$ is {\em true} 
in $\mathbbm{B}$ if when interpreting propositional variables occurring in $\varphi$ as elements of $B$, the corresponding polynomial 
evaluates to $1$
in $\mathbbm{B}$.



\begin{defn}\ 
    \begin{enumerate}
        \item A {\em monadic modal algebra} or simply an {\em mm-algebra} is a tuple $\mathbbm{B}=(B,\Diamond,\exists)$ such that $(B,\Diamond)$ is a modal algebra, $(B,\exists)$ is an {\textbf{S5}}-algebra, and $\Diamond$ and $\exists$ are connected by $\exists\Diamond a\leq \Diamond\exists a$ for each $a\in B$.

        \item A map $f\colon(B,\Diamond,\exists)\to(B',\Diamond',\exists')$ is a {\em morphism of mm-algebras} if $f$ is a morphism of Boolean algebras such that $f(\Diamond a)=\Diamond'f(a)$ and $f(\exists a)=\exists'f(a)$ for each $a\in B$.
    \end{enumerate}
    
\end{defn}


%


Clearly, mm-algebras provide an algebraic semantics for mm-logics. We denote the 
category of mm-algebras and their morphisms by $\mathsf{MMA}$.

 
J\'onsson-Tarski duality for Boolean algebras with operators yields a convenient representation of mm-algebras. We recall that a subset of a topological space $X$ is {\em clopen} if it is both closed and open, and that $X$ is {\em zero-dimensional} if clopen sets form a basis for $X$.
Given a binary relation $R$ on $X$ and $U\subseteq X$, let  
\[
R[U] = \{ x \in X \mid u\mathrel{R}x \mbox{ for some } u\in U\},
\]
and
\[
R^{-1}[U] = \{ x \in X \mid x\mathrel{R}u \mbox{ for some } u\in U\}.
\]

If $U=\{x\}$, then we write $R[x]$ and $R^{-1}[x]$ instead of $R[U]$ and $R^{-1}[U]$  (see 
\crefdefpart{basics}{basics1}). 

\begin{defn}
\begin{enumerate}
    \item[]
    \item \cite[Sec.~II.4]{johnstone} A topological space $X$ is said to be a {\em Stone space} if it is compact, Hausdorff, and zero-dimensional.
    \item \cite[Sec.~3.1]{esakia2019} A relation $R$ on a Stone space $X$ is said to be {\em continuous} if $R[x]$ is closed for each $x\in X$ and $R^{-1}[U]$ is clopen for each clopen subset $U$ of $X$.
\end{enumerate}
\end{defn}

We next recall descriptive $\bf{MK}$-frames (see \cite[Def.~2.18]{GBm}).

\begin{defn}\ 
\begin{enumerate}
    \item A {\em descriptive $\bf{MK}$-frame} or simply a {\em $\bf{DMK}$-frame} is an $\bf{MK}$-frame $\mathfrak{F}=(X,R,E)$ such that $X$ is a Stone space and $R$, $E$ are continuous relations on $X$. 
    \item A {\em {\bf DMK}-morphism} between two {\bf DMK}-frames $\mathfrak F=(X,R,E)$ and $\mathfrak F'=(X',R',E')$ is a continuous map $f\colon X\to X'$ that is a morphism of the underlying \textbf{MK}-frames.
    \item Let $\mathsf{DMKF}$ be the category of {\bf DMK}-frames and {\bf DMK}-morphisms.
\end{enumerate}
\end{defn}


Specializing J\'onsson-Tarski duality to the category of mm-algebras yields the following result.

\begin{thm} \label{thm: JT for monadic}
    $\mathsf{MMA}$ is dually equivalent to $\mathsf{DMKF}$.
\end{thm}

\begin{rmk}{\label{duality}}
The functors establishing the above dual equivalence are constructed as follows: the functor ${\text{Clop}\colon\mathsf{DMKF}\to\mathsf{MMA}}$ associates with each $\textbf{DMK}$-frame $(X,R,E)$ the mm-algebra $({B}(X), \Diamond_R, \exists_E)$, where $B(X)$ is the Boolean algebra of clopen sets of $X$, $\Diamond_R U = R^{-1}[U]$, and $\exists_E U = E[U]$ for each $U\in B(X)$. Moreover, it associates with each \textbf{DMK}-morphism $f\colon X\to X'$ the mm-morphism $\text{Clop}(f)\colon B(X')\to B(X)$ given by $\text{Clop}(f)(U)=f^{-1}[U]$ for each $U\in B(X')$. 

The functor $\text{Uf}\colon\mathsf{MMA}\to\mathsf{DMKF}$ associates with each mm-algebra $(B,\Diamond,\exists)$ the $\bf{DMK}$-frame $(X_B,R_\Diamond,E_\exists)$ defined as follows: $X_B$ is the Stone space of $B$; that is, the space of ultrafilters of $B$ whose basic open sets are $\sigma(a)=\{u\in X_B \mid a\in u \}$ for each $a\in B$. 
The relation $R_\Diamond$ is given by 
\begin{equation*}
    u\mathrel{R_\Diamond} v\iff v\subseteq \Diamond^{-1}[u], 
\end{equation*}
and the relation 
$E_\exists$ 
by
\begin{equation*}
    u\mathrel{E_\exists} v\iff \exists^{-1}[u]=\exists^{-1}[v].
\end{equation*}
 Moreover, it associates with each mm-morphism $f \colon B \to B'$ the $\bf{DMK}$-morphism $\text{Uf}(f)\colon X_{B'} \to X_B$ given by $\text{Uf}(f)(v) = f^{-1}[v]$ for each $v\in X_{B'}$.
\end{rmk}



\begin{defn}
    Let \textbf{M} be an mm-logic, $\mathbbm{B} = (B,\Diamond,\exists)$ an mm-algebra, and $\mathfrak F = (X,R,E)$ a \textbf{DMK}-frame. We say that:
    \begin{enumerate}
        \item $\mathbbm{B}$ is an {\em {\bf M}-algebra} if the theorems of \textbf{M} are true in $\mathbbm{B}$;
        \item $\mathfrak F$ is a {\em descriptive {\bf M}-frame} if its dual $\text{Clop}(\mathfrak{F})$ is an \textbf{M}-algebra.
    \end{enumerate}
\end{defn}

Let $(B,\Diamond,\exists)$ be an mm-algebra. Then $a\leq\Diamond a$ for each $a\in B$ iff $R_\Diamond$ is reflexive, and $\Diamond\Diamond a\leq \Diamond a$ for each $a\in B$ iff $R_\Diamond$ is transitive \cite[Thm.~3.5]{jonsson}. Therefore, 
$(B,\Diamond,\exists)$ is an \textbf{MS4}-algebra iff $(X_B,R_\Diamond, E_\exists)$ is a preordered \textbf{DMK}-frame. Thus, descriptive \textbf{MS4}-frames are exactly preordered \textbf{DMK}-frames.

Let $\mathsf{MS4A}$ be the full subcategory of $\mathsf{MMA}$ consisting of \textbf{MS4}-algebras, and let $\mathsf{DMS4F}$ be the full subcategory of $\mathsf{DMKF}$ consisting of descriptive \textbf{MS4}-frames.
Then, by the previous paragraph, \Cref{thm: JT for monadic} specializes to the following:

\begin{thm}
    $\mathsf{MS4A}$ is dually equivalent to $\mathsf{DMS4F}$.
\end{thm}


To characterize descriptive \textbf{MGrz}-frames, 
we require the following (see, e.g., \cite[Def.~1.4.9]{esakia2019}): 

\begin{defn}
    Let $\mathfrak{F}=(X,R,E)$ be a descriptive {\bf MS4}-frame and $U\subseteq X$. 
    \begin{enumerate}
        \item A point $x\in U$ is {\em quasi-maximal} in $U$ if $x\mathrel{R}y$ implies $y\mathrel{R}x$ for each $y\in U$. We let ${\textbf{qmax}}_R U$ denote the set of quasi-maximal points of $U$.
        \item A point $x\in U$ is {\em maximal} in $U$ if $x\mathrel{R}y$ implies $x=y$ for each $y\in U$. We let ${\textbf{max}}_R U$ denote the set of maximal points of $U$.
    \end{enumerate}
\end{defn}

%
%

\begin{rmk}
Clearly ${\textbf{max}}_R U\subseteq {\textbf{qmax}}_R U$. However, the reverse inclusion does not hold in general.
For instance, consider the {\bf MS4}-frame $(X,R,E)$, where $X=\{a,b\}$, $R=X^2$, and $E$ is the diagonal (see the diagram below, where the quasi-order $R$ is indicated with black arrows and circles, and $E$-clusters are shown in blue). Since $X$ is finite, giving $X$ the discrete topology turns $(X,R,E)$ into a descriptive {\bf MS4}-frame. Moreover, 
$\textbf{qmax}_R X=X$, while $\textbf{max}_R X=\varnothing$. Thus, ${\textbf{qmax}}_R X \not\subseteq {\textbf{max}}_R X$.
In fact, ${\textbf{qmax}}_R U \subseteq {\textbf{max}}_R U$ iff the $R$-cluster $E_R[x] := \{ y \in X \mid x \mathrel{R} y$ and $y \mathrel{R} x \}$ of each $x \in {\textbf{qmax}}_R U$ is the singleton $\{x\}$.

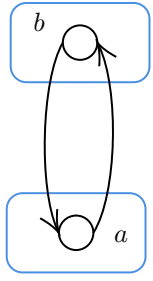
\begin{center}
\tikzset{every picture/.style={line width=0.75pt}}
\begin{tikzpicture}[x=0.75pt,y=0.75pt,yscale=-1,xscale=1, scale=0.6]
\draw  [color={rgb, 255:red, 74; green, 144; blue, 226 }  ,draw opacity=1 ] (14,163) .. controls (14,158.58) and (17.58,155) .. (22,155) -- (76,155) .. controls (80.42,155) and (84,158.58) .. (84,163) -- (84,187) .. controls (84,191.42) and (80.42,195) .. (76,195) -- (22,195) .. controls (17.58,195) and (14,191.42) .. (14,187) -- cycle ;
\draw  [color={rgb, 255:red, 74; green, 144; blue, 226 }  ,draw opacity=1 ] (16,67) .. controls (16,62.58) and (19.58,59) .. (24,59) -- (78,59) .. controls (82.42,59) and (86,62.58) .. (86,67) -- (86,91) .. controls (86,95.42) and (82.42,99) .. (78,99) -- (24,99) .. controls (19.58,99) and (16,95.42) .. (16,91) -- cycle ;
\draw   (40.3,175) .. controls (40.3,170.2) and (44.2,166.3) .. (49,166.3) .. controls (53.8,166.3) and (57.7,170.2) .. (57.7,175) .. controls (57.7,179.8) and (53.8,183.7) .. (49,183.7) .. controls (44.2,183.7) and (40.3,179.8) .. (40.3,175) -- cycle ;
\draw   (42.3,79) .. controls (42.3,74.2) and (46.2,70.3) .. (51,70.3) .. controls (55.8,70.3) and (59.7,74.2) .. (59.7,79) .. controls (59.7,83.8) and (55.8,87.7) .. (51,87.7) .. controls (46.2,87.7) and (42.3,83.8) .. (42.3,79) -- cycle ;
\draw    (57.7,175) .. controls (69.69,156.08) and (70.94,102.38) .. (60.52,80.61) ;
\draw [shift={(59.7,79)}, rotate = 61.25] [color={rgb, 255:red, 0; green, 0; blue, 0 }  ][line width=0.75]    (10.93,-4.9) .. controls (6.95,-2.3) and (3.31,-0.67) .. (0,0) .. controls (3.31,0.67) and (6.95,2.3) .. (10.93,4.9)   ;
\draw    (42.3,79) .. controls (31.28,98.11) and (30.06,152.77) .. (39.55,173.48) ;
\draw [shift={(40.3,175)}, rotate = 242.03] [color={rgb, 255:red, 0; green, 0; blue, 0 }  ][line width=0.75]    (10.93,-4.9) .. controls (6.95,-2.3) and (3.31,-0.67) .. (0,0) .. controls (3.31,0.67) and (6.95,2.3) .. (10.93,4.9)   ;
\draw (66.7,172.4) node [anchor=north west][inner sep=0.75pt]    {$a$};
\draw (26,62.4) node [anchor=north west][inner sep=0.75pt]    {$b$};
\end{tikzpicture}
\end{center}
\end{rmk}

%

%


The following definition is a direct adaptation of \cite[Def.~3.5.4]{esakia2019} to the monadic setting. 

\begin{defn}{\label{passive}}
Let $\mathfrak{F}=(X,R,E)$ be a descriptive \textbf{MS4}-frame. A point $x$ in a clopen subset $U$ of $X$ is an \emph{active point} if there exist $y,z\in X$ such that $x\mathrel{R}y\mathrel{R}z$, $y\notin U$, and $z\in U$. Otherwise, $x$ is a \emph{passive point}. 
\end{defn}

In other words, $x$ is an active point of $U$ if it is possible to exit and re-enter $U$ via the relation $R$; see the diagram below. 

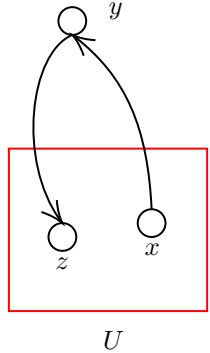
\begin{center}
\tikzset{every picture/.style={line width=0.75pt}} 

\begin{tikzpicture}[x=0.75pt,y=0.75pt,yscale=-1,xscale=1,scale=0.6]
\draw  [color={rgb, 255:red, 255; green, 0; blue, 0 }  ,draw opacity=1 ] (100,101.4) -- (200,101.4) -- (200,183.4) -- (100,183.4) -- cycle ;
\draw   (165,139) .. controls (165,135.13) and (168.13,132) .. (172,132) .. controls (175.87,132) and (179,135.13) .. (179,139) .. controls (179,142.87) and (175.87,146) .. (172,146) .. controls (168.13,146) and (165,142.87) .. (165,139) -- cycle ;
\draw   (120,146) .. controls (120,142.13) and (123.13,139) .. (127,139) .. controls (130.87,139) and (134,142.13) .. (134,146) .. controls (134,149.87) and (130.87,153) .. (127,153) .. controls (123.13,153) and (120,149.87) .. (120,146) -- cycle ;
\draw   (125,37) .. controls (125,33.13) and (128.13,30) .. (132,30) .. controls (135.87,30) and (139,33.13) .. (139,37) .. controls (139,40.87) and (135.87,44) .. (132,44) .. controls (128.13,44) and (125,40.87) .. (125,37) -- cycle ;
\draw    (172,132) .. controls (170.03,91.02) and (159.33,65.18) .. (133.2,44.92) ;
\draw [shift={(132,44)}, rotate = 37.07] [color={rgb, 255:red, 0; green, 0; blue, 0 }  ][line width=0.75]    (10.93,-4.9) .. controls (6.95,-2.3) and (3.31,-0.67) .. (0,0) .. controls (3.31,0.67) and (6.95,2.3) .. (10.93,4.9)   ;
\draw    (132,44) .. controls (110.33,53.26) and (104.18,107.73) .. (125.98,137.65) ;
\draw [shift={(127,139)}, rotate = 232.15] [color={rgb, 255:red, 0; green, 0; blue, 0 }  ][line width=0.75]    (10.93,-4.9) .. controls (6.95,-2.3) and (3.31,-0.67) .. (0,0) .. controls (3.31,0.67) and (6.95,2.3) .. (10.93,4.9)   ;
\draw (167,148) node [anchor=north west][inner sep=0.75pt]    {$x$};
\draw (149,26.4) node [anchor=north west][inner sep=0.75pt]    {$y$};
\draw (122,155) node [anchor=north west][inner sep=0.75pt]    {$z$};
\draw (146,191.8) node [anchor=north west][inner sep=0.75pt]    {$U$};
\end{tikzpicture}
\end{center}

Let $\pi\, U$ be the set of passive points of $U$.
The following was originally proved for descriptive \textbf{Grz}-frames, but the same characterization applies to descriptive \textbf{MGrz}-frames as well.

\begin{lem} {\label{fine-esakia}} 
\cite[p.~70]{esakia2019} Let $\mathfrak F = (X,R,E)$ be a descriptive \textbf{MS4}-frame. The following are equivalent.
\begin{enumerate}
    \item $\mathfrak F$ is a descriptive \textbf{MGrz}-frame.
    \item ${\textbf{max}}_R U = {\textbf{qmax}}_R U$ for each clopen $U\subseteq X$.
    \item $U\subseteq R^{-1}[\pi\,U]$ for each clopen $U\subseteq X$.
\end{enumerate}
\end{lem}

\begin{rmk}{\label{enter-exit}}
The above lemma yields 
that if $U$ is a clopen set in a descriptive \textbf{MGrz}-frame, then each point $y\in U$ can reach a passive point $x\in U$ through $R$. 
Since a maximal point $y$ of $U$ can only reach itself in $U$ through $R$, we obtain that each maximal point in $U$ is passive; that is, ${\textbf{max}}_R U \subseteq \pi\, U$. This, in particular, implies that maximal points in a descriptive \textbf{MGrz}-frame cannot be in an $R$-cluster with any other point. Furthermore, we not only have that $U\subseteq R^{-1}[\pi\,U]$, but also that $U\subseteq R^{-1}[\textbf{max}_R U]$ (see \Cref{fe_principle}). 
We use these observations in several proofs below. 
\end{rmk}

%
%




Note that maximal (and quasi-maximal) points in descriptive frames are defined purely in terms of the relation $R$. We next refine this notion by also involving the equivalence relation $E$.

%
%


For a \textbf{DMK}-frame $\mathfrak{F}=(X,R,E)$, recall that $Q = E\circ R$ (see \crefdefpart{Q_rel}{Qrel}). Since $R$ and $E$ are both continuous relations, so is $Q$. Moreover, $R\subseteq Q$ and if $R$ is a preorder, then so is $Q$.
    
\begin{defn}{\label{strongly_maximal}}
    Let $\mathfrak{F}=(X,R,E)$ be a descriptive \textbf{MGrz}-frame and $U\subseteq X$. A point $x\in U$ is said to be {\em strongly maximal} in $U$ if $x\in{\textbf{max}}_R U$ and the following property holds:
    \begin{equation*}
        x\mathrel{Q}y \text{ and } y\in U \text { imply } x\mathrel{E}y. 
    \end{equation*}
    We let ${\textbf{smax}}_R U$ denote the set of strongly maximal points of $U$. 
\end{defn}

It is obvious that $x\in {\textbf{smax}}_R U$ iff $x\in{\textbf{max}}_R U$ and  
\begin{equation*}
    \text{ for each } y \in X, \, x \mathrel{Q} y \text{ and } x \mathrel{\cancel{E}}y \text{ imply } y\notin U.
\end{equation*} 

\begin{rmk}
        Clearly ${\textbf{smax}}_R U\subseteq {\textbf{max}}_R U$, and this inclusion can be proper. For example, consider the following \textbf{MGrz}-frame:
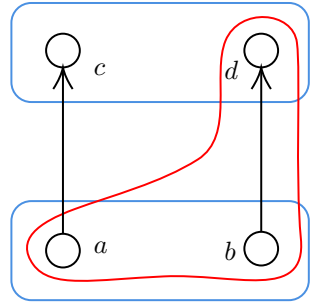
\begin{center}
\tikzset{every picture/.style={line width=0.7pt}} 
\begin{tikzpicture}[x=0.75pt,y=0.75pt,yscale=-1,xscale=1,scale=0.75]
\draw  [color={rgb, 255:red, 74; green, 144; blue, 226 }  ,draw opacity=1 ] (200,61.2) .. controls (200,55.68) and (204.48,51.2) .. (210,51.2) -- (340,51.2) .. controls (345.52,51.2) and (350,55.68) .. (350,61.2) -- (350,91.2) .. controls (350,96.72) and (345.52,101.2) .. (340,101.2) -- (210,101.2) .. controls (204.48,101.2) and (200,96.72) .. (200,91.2) -- cycle ;
\draw  [color={rgb, 255:red, 74; green, 144; blue, 226 }  ,draw opacity=1 ] (200,161.2) .. controls (200,155.68) and (204.48,151.2) .. (210,151.2) -- (340,151.2) .. controls (345.52,151.2) and (350,155.68) .. (350,161.2) -- (350,191.2) .. controls (350,196.72) and (345.52,201.2) .. (340,201.2) -- (210,201.2) .. controls (204.48,201.2) and (200,196.72) .. (200,191.2) -- cycle ;
\draw   (217.5,75.2) .. controls (217.5,70.51) and (221.31,66.7) .. (226,66.7) .. controls (230.69,66.7) and (234.5,70.51) .. (234.5,75.2) .. controls (234.5,79.89) and (230.69,83.7) .. (226,83.7) .. controls (221.31,83.7) and (217.5,79.89) .. (217.5,75.2) -- cycle ;
\draw   (217.5,176.2) .. controls (217.5,171.51) and (221.31,167.7) .. (226,167.7) .. controls (230.69,167.7) and (234.5,171.51) .. (234.5,176.2) .. controls (234.5,180.89) and (230.69,184.7) .. (226,184.7) .. controls (221.31,184.7) and (217.5,180.89) .. (217.5,176.2) -- cycle ;
\draw   (317.5,75.2) .. controls (317.5,70.51) and (321.31,66.7) .. (326,66.7) .. controls (330.69,66.7) and (334.5,70.51) .. (334.5,75.2) .. controls (334.5,79.89) and (330.69,83.7) .. (326,83.7) .. controls (321.31,83.7) and (317.5,79.89) .. (317.5,75.2) -- cycle ;
\draw   (317.5,175.2) .. controls (317.5,170.51) and (321.31,166.7) .. (326,166.7) .. controls (330.69,166.7) and (334.5,170.51) .. (334.5,175.2) .. controls (334.5,179.89) and (330.69,183.7) .. (326,183.7) .. controls (321.31,183.7) and (317.5,179.89) .. (317.5,175.2) -- cycle ;
\draw  [color={rgb, 255:red, 255; green, 0; blue, 0}  ,draw opacity=1 ] (308,70.2) .. controls (315,53.2) and (342,55.2) .. (344,71.2) .. controls (346,87.2) and (343,138.2) .. (346,170.2) .. controls (349,202.2) and (312.09,187.58) .. (277,189.2) .. controls (241.91,190.82) and (219.85,195.59) .. (211,183.2) .. controls (202.15,170.81) and (213,162.2) .. (232,155.2) .. controls (251,148.2) and (275,142.2) .. (294,130.2) .. controls (313,118.2) and (301,87.2) .. (308,70.2) -- cycle ;
\draw    (226,167.7) -- (226,85.7) ;
\draw [shift={(226,83.7)}, rotate = 90] [color={rgb, 255:red, 0; green, 0; blue, 0 }  ][line width=0.75]    (10.93,-4.9) .. controls (6.95,-2.3) and (3.31,-0.67) .. (0,0) .. controls (3.31,0.67) and (6.95,2.3) .. (10.93,4.9)   ;
\draw    (326,166.7) -- (326,85.7) ;
\draw [shift={(326,83.7)}, rotate = 90] [color={rgb, 255:red, 0; green, 0; blue, 0 }  ][line width=0.75]    (10.93,-4.9) .. controls (6.95,-2.3) and (3.31,-0.67) .. (0,0) .. controls (3.31,0.67) and (6.95,2.3) .. (10.93,4.9)   ;
\draw (240,170.4) node [anchor=north west][inner sep=0.75pt]    {$a$};
\draw (306,170.4) node [anchor=north west][inner sep=0.75pt]    {$b$};
\draw (240,80.4) node [anchor=north west][inner sep=0.75pt]    {$c$};
\draw (306,79) node [anchor=north west][inner sep=0.75pt]    {$d$};
\end{tikzpicture}
\end{center}
The black arrows and circles represent the relation $R$, the $E$-clusters are indicated in blue, and $U=\{a,b,d\}$ is represented by 
the red curve.
With the discrete topology, this is a descriptive \textbf{MGrz}-frame. Then observe that $a\in\textbf{max}_R U$, but $a\notin\textbf{smax}_R U$ since $a\mathrel{Q}d$, $a\mathrel{\cancel{E}}d$, and yet $d\in U$.

\end{rmk}

On the other hand, we show that if $U$ is $E$-saturated (that is, $E[U]=U$), 
then $\textbf{smax}_R U=\textbf{max}_R U$. Therefore, while the notions of maximal and strongly maximal points do not always coincide, they agree on $E$-saturated sets.

\begin{prop}{\label{sat-smax}}
    Let $\mathfrak{F}=(X,R,E)$ be a descriptive \textbf{MGrz}-frame and $U \subseteq X$. 
    \begin{enumerate}
    \item If $x\in U\cap\textbf{max}_R E[U]$, then $x\in\textbf{smax}_R E[U]\cap\textbf{smax}_R U$. \label{saturated}
    \item If $U$ is $E$-saturated, then $\textbf{smax}_R U=\textbf{max}_R U$. \label{smax=max}
    \end{enumerate}
\end{prop}

\begin{proof}
(\labelcref{saturated}) We show that $x\in{\textbf{smax}}_R E[U]$. That $x\in{\textbf{smax}}_R U$ can be proved similarly. Suppose $x\mathrel{Q}t$ and $x\mathrel{\cancel{E}}t$ for some $t\in X$. By definition, $x\mathrel{R}y$ and $y\mathrel{E}t$ for some $y\in X$. From $x\mathrel{\cancel{E}}t$ and $y\mathrel{E}t$ it follows that $x\neq y$. Since $x$ is $R$-maximal in $E[U]$, we get $y\notin E[U]$. 
Since {$y\mathrel{E}t$,} we must have $t\notin E[U]$. Thus, $x\in{\textbf{smax}}_R E[U]$. 


%
(\labelcref{smax=max}) By definition, $\textbf{smax}_R U \subseteq \textbf{max}_R U$. For the other inclusion, let $x\in \textbf{max}_R U$. Since $U$ is $E$-saturated, $E[U]=U$, and thus $x\in U \cap \textbf{max}_R E[U]$. Therefore, $x\in\textbf{smax}_R U$ by (\labelcref{saturated}).
\end{proof}

%
%

We recall the well-known Fine-Esakia Principle, applied to descriptive \textbf{MGrz}-frames.

\begin{thm}(Fine-Esakia Principle, {\cite[Cor.~3.5.7]{esakia2019}).}{\label{fe_principle}}
    Let $\mathfrak{F}=(X,R,E)$ be a descriptive \textbf{MGrz}-frame. For each clopen set $U$ and $x\in U$, we have $R[x]\cap\textbf{max}_R U \neq \varnothing$.
\end{thm}

The following result is a modification of the above 
and will play a crucial role in the selective filtration construction of the next section. 

\begin{thm}{\label{smax}}
    Let $\mathfrak{F}=(X,R,E)$ be a descriptive {\textbf{MGrz}}-frame. For each clopen set $U$ and $x\in U$, we have $Q[x]\cap {\textbf{smax}}_R U \neq \varnothing$. 
\end{thm}
\begin{proof}
Let $x\in U$. Since $E$ is a continuous relation, $E[U]$ is  clopen. Hence, by the Fine-Esakia principle, 
there is $t \in R[x]\cap \textbf{max}_R E[U]$. From $t\in E[U]$ it follows that $t\mathrel{E}u$
for some $u\in U$. Applying the Fine-Esakia principle again, 
there is $z \in R[u]\cap \textbf{max}_R U$. We show that $z\in Q[x]\cap\textbf{smax}_R U$. By commutativity, $x\mathrel{Q}z$.
Since $z\in \textbf{max}_R U$, to see that $z\in\textbf{smax}_R U$, suppose $z\mathrel{Q}y$ and $z\mathrel{\cancel{E}}y$.   
It suffices to show that $y\notin U$. Since $t$ is $R$-maximal in $E[U]$, we must have $t\in\textbf{smax}_R E[U]$ by \crefdefpart{sat-smax}{smax=max}. Therefore, $t\mathrel{Q}z$ and $z\in E[U]$ imply that $t\mathrel{E}z$. Thus, $t\mathrel{Q}y$ and $t\mathrel{\cancel{E}}y$. Due to the strong maximality of $t$, we must have $y\notin E[U]$. Since $U\subseteq E[U]$, we conclude that $y\notin U$. 
\end{proof}

%

For a descriptive \textbf{MGrz}-frame $\mathfrak{F}=(X,R,E)$, the relation $Q=E\circ R$ gives rise to the equivalence relation $E_Q$ whose equivalence classes are the
$Q$-clusters $E_Q[x]:=Q[x]\cap Q^{-1}[x]$ for each $x\in X$. We clearly have that $E\subseteq E_Q$. However, the converse is not true in general. 

\begin{lem}{\label{E_Q}}
Let $\mathfrak{F}=(X,R,E)$ be a descriptive \textbf{MGrz}-frame
and $x\in X$. 
    \begin{enumerate}
    \item If $\textbf{max}_R E_Q[x]\neq\varnothing$, then $E_Q[x]=E[x]$. {\label{E_Q0}}
    \item If $x\in \textbf{smax}_R U$ for some clopen $U \subseteq X$, then $E_Q[x]=E[x]$. {\label{E_Q1}}
    \item If $(X,R,E)$ is a finite \textbf{MGrz}-frame, then $E_Q=E$. {\label{E_Q2}}
\end{enumerate}
\end{lem}
\begin{proof}
    (\labelcref{E_Q0}) This follows from \cite[Lem.~3(a)]{BV}. Note that the indicated result assumes that $R$ is a partial order. However, the same proof works in our more general setting.
    
    (\labelcref{E_Q1}) Let $x\in\textbf{smax}_R U$. Then $x\in E[U]$, so \Cref{smax} yields $y\in\textbf{smax}_R E[U]$ such that $x\mathrel{Q}y$. 
    We first show that $x\mathrel{E}y$. Suppose not.
    Since $y\in E[U]$, there is $t\in U$ with $y\mathrel{E}t$. From $x\mathrel{Q}y$ and $y\mathrel{E}t$ it follows that $x\mathrel{Q}t$; and from $x\mathrel{\cancel{E}}y$ it follows that $x\mathrel{\cancel{E}}t$. Therefore, $x\mathrel{Q}t$, $x\mathrel{\cancel{E}}t$, and $t\in U$,  contradicting that $x\in\textbf{smax}_R U$. Thus, $x\mathrel{E}y$. We next show that $y \in \textbf{max}_R E_Q[x]$. Let $y\mathrel{R}z$ and $y\ne z$ for some $z\in E_Q[x]$. Then $z \notin E[U]$ since $y \in \textbf{max}_R E[U]$. On the other hand, since $z\mathrel{E_Q}x$, we have $z\mathrel{Q}y$, so
    $z\mathrel{R}u\mathrel{E}y$ for some $u\in X$. Consequently, 
    $y\mathrel{R}z\mathrel{R}u$ with $y,u\in E[U]$ and $z\notin E[U]$, contradicting that $\mathfrak F$ is a descriptive \textbf{MGrz}-frame (see \Cref{enter-exit}). 
    We conclude that $y \in \textbf{max}_R E_Q[x]$, and hence $E_Q[x]=E[x]$ by (\labelcref{E_Q0}).


    (\labelcref{E_Q2}) If $(X,R,E)$ is finite, then $R$ is a partial order (see, e.g., \cite[Cor.~3.5.10]{esakia2019}). Therefore, $\textbf{max}_R E_Q[x]\neq\varnothing$ for each $x\in X$. Thus, (\labelcref{E_Q0}) applies, and hence $E_Q[x] = E[x]$ for each $x\in X$. Consequently, $E_Q=E$.
\end{proof}


We next show that the assumption in \crefdefpart{E_Q}{E_Q1} that $x\in \textbf{smax}_R U$ for some clopen $U \subseteq X$ is essential.

\begin{exm}
Consider the frame $(X,R,E)$ depicted below, where the black arrows and circles indicate the relation $R$ (see, e.g., \cite[p.~252, Fig.~8.3(b)]{CZ}), and $E$-clusters are shown in blue. The points $x_n,y_n$ are isolated, $(x_n) \longrightarrow x_\infty$, and $(y_n) \longrightarrow y_\infty$. Thus, $X$ is {the} $2$-point compactification of a discrete space, hence is a Stone space. It is straightforward to check that $(X,R)$ is a descriptive \textbf{Grz}-frame (see \cite[Ex.~8.52]{CZ}). Since $E$ is the diagonal, it is then clear that $(X,R,E)$ is a descriptive \textbf{MGrz}-frame. 
We have $E[x_\infty]=\{x_\infty\}$, but $E_Q[x_\infty] = \{ x_\infty, y_\infty \}$. 
\begin{center}
\tikzset{every picture/.style={line width=0.8pt}}
\begin{tikzpicture}[x=0.75pt,y=0.75pt,yscale=-1,xscale=1, scale=0.525]
\draw   (273.9,24.1) .. controls (273.9,20.73) and (276.63,18) .. (280,18) .. controls (283.37,18) and (286.1,20.73) .. (286.1,24.1) .. controls (286.1,27.47) and (283.37,30.2) .. (280,30.2) .. controls (276.63,30.2) and (273.9,27.47) .. (273.9,24.1) -- cycle ;
\draw    (280,32.2) -- (280,60.2) ;
\draw [shift={(280,30.2)}, rotate = 90] [color={rgb, 255:red, 0; green, 0; blue, 0 }  ][line width=0.75]    (10.93,-4.9) .. controls (6.95,-2.3) and (3.31,-0.67) .. (0,0) .. controls (3.31,0.67) and (6.95,2.3) .. (10.93,4.9)   ;
\draw   (273.9,67.1) .. controls (273.9,63.73) and (276.63,61) .. (280,61) .. controls (283.37,61) and (286.1,63.73) .. (286.1,67.1) .. controls (286.1,70.47) and (283.37,73.2) .. (280,73.2) .. controls (276.63,73.2) and (273.9,70.47) .. (273.9,67.1) -- cycle ;
\draw    (280,75.2) -- (280,103.2) ;
\draw [shift={(280,73.2)}, rotate = 90] [color={rgb, 255:red, 0; green, 0; blue, 0 }  ][line width=0.75]    (10.93,-4.9) .. controls (6.95,-2.3) and (3.31,-0.67) .. (0,0) .. controls (3.31,0.67) and (6.95,2.3) .. (10.93,4.9)   ;
\draw   (273.9,109.1) .. controls (273.9,105.73) and (276.63,103) .. (280,103) .. controls (283.37,103) and (286.1,105.73) .. (286.1,109.1) .. controls (286.1,112.47) and (283.37,115.2) .. (280,115.2) .. controls (276.63,115.2) and (273.9,112.47) .. (273.9,109.1) -- cycle ;
\draw    (280,117.2) -- (280,145.2) ;
\draw [shift={(280,115.2)}, rotate = 90] [color={rgb, 255:red, 0; green, 0; blue, 0 }  ][line width=0.75]    (10.93,-4.9) .. controls (6.95,-2.3) and (3.31,-0.67) .. (0,0) .. controls (3.31,0.67) and (6.95,2.3) .. (10.93,4.9)   ;
\draw   (273.9,151.1) .. controls (273.9,147.73) and (276.63,145) .. (280,145) .. controls (283.37,145) and (286.1,147.73) .. (286.1,151.1) .. controls (286.1,154.47) and (283.37,157.2) .. (280,157.2) .. controls (276.63,157.2) and (273.9,154.47) .. (273.9,151.1) -- cycle ;
\draw    (280,159.2) -- (280,187.2) ;
\draw [shift={(280,157.2)}, rotate = 90] [color={rgb, 255:red, 0; green, 0; blue, 0 }  ][line width=0.75]    (10.93,-4.9) .. controls (6.95,-2.3) and (3.31,-0.67) .. (0,0) .. controls (3.31,0.67) and (6.95,2.3) .. (10.93,4.9)   ;
\draw   (314.1,253.9) .. controls (317.47,253.9) and (320.2,256.63) .. (320.2,260) .. controls (320.2,263.37) and (317.47,266.1) .. (314.1,266.1) .. controls (310.73,266.1) and (308,263.37) .. (308,260) .. controls (308,256.63) and (310.73,253.9) .. (314.1,253.9) -- cycle ;
\draw    (306,260.01) -- (255,260.19) ;
\draw [shift={(253,260.2)}, rotate = 359.79] [color={rgb, 255:red, 0; green, 0; blue, 0 }  ][line width=0.75]    (10.93,-4.9) .. controls (6.95,-2.3) and (3.31,-0.67) .. (0,0) .. controls (3.31,0.67) and (6.95,2.3) .. (10.93,4.9)   ;
\draw [shift={(308,260)}, rotate = 179.79] [color={rgb, 255:red, 0; green, 0; blue, 0 }  ][line width=0.75]    (10.93,-4.9) .. controls (6.95,-2.3) and (3.31,-0.67) .. (0,0) .. controls (3.31,0.67) and (6.95,2.3) .. (10.93,4.9)   ;
\draw   (246.9,254.1) .. controls (250.27,254.1) and (253,256.83) .. (253,260.2) .. controls (253,263.57) and (250.27,266.3) .. (246.9,266.3) .. controls (243.53,266.3) and (240.8,263.57) .. (240.8,260.2) .. controls (240.8,256.83) and (243.53,254.1) .. (246.9,254.1) -- cycle ;

\draw  [fill={rgb, 255:red, 0; green, 0; blue, 0 }  ,fill opacity=1 ][line width=0.75]  (279,200.2) .. controls (279,199.65) and (279.45,199.2) .. (280,199.2) .. controls (280.55,199.2) and (281,199.65) .. (281,200.2) .. controls (281,200.75) and (280.55,201.2) .. (280,201.2) .. controls (279.45,201.2) and (279,200.75) .. (279,200.2) -- cycle ;
\draw  [fill={rgb, 255:red, 0; green, 0; blue, 0 }  ,fill opacity=1 ][line width=0.75]  (279,220.2) .. controls (279,219.65) and (279.45,219.2) .. (280,219.2) .. controls (280.55,219.2) and (281,219.65) .. (281,220.2) .. controls (281,220.75) and (280.55,221.2) .. (280,221.2) .. controls (279.45,221.2) and (279,220.75) .. (279,220.2) -- cycle ;
\draw  [fill={rgb, 255:red, 0; green, 0; blue, 0 }  ,fill opacity=1 ][line width=0.75]  (279,240.2) .. controls (279,239.65) and (279.45,239.2) .. (280,239.2) .. controls (280.55,239.2) and (281,239.65) .. (281,240.2) .. controls (281,240.75) and (280.55,241.2) .. (280,241.2) .. controls (279.45,241.2) and (279,240.75) .. (279,240.2) -- cycle ;
\draw  [color={rgb, 255:red, 74; green, 144; blue, 226 }  ,draw opacity=1 ] (268,103.28) .. controls (268,101.14) and (269.74,99.4) .. (271.88,99.4) -- (288.12,99.4) .. controls (290.26,99.4) and (292,101.14) .. (292,103.28) -- (292,114.92) .. controls (292,117.06) and (290.26,118.8) .. (288.12,118.8) -- (271.88,118.8) .. controls (269.74,118.8) and (268,117.06) .. (268,114.92) -- cycle ;
\draw  [color={rgb, 255:red, 74; green, 144; blue, 226 }  ,draw opacity=1 ] (268,18.28) .. controls (268,16.14) and (269.74,14.4) .. (271.88,14.4) -- (288.12,14.4) .. controls (290.26,14.4) and (292,16.14) .. (292,18.28) -- (292,29.92) .. controls (292,32.06) and (290.26,33.8) .. (288.12,33.8) -- (271.88,33.8) .. controls (269.74,33.8) and (268,32.06) .. (268,29.92) -- cycle ;
\draw  [color={rgb, 255:red, 74; green, 144; blue, 226 }  ,draw opacity=1 ] (268,61.28) .. controls (268,59.14) and (269.74,57.4) .. (271.88,57.4) -- (288.12,57.4) .. controls (290.26,57.4) and (292,59.14) .. (292,61.28) -- (292,72.92) .. controls (292,75.06) and (290.26,76.8) .. (288.12,76.8) -- (271.88,76.8) .. controls (269.74,76.8) and (268,75.06) .. (268,72.92) -- cycle ;
\draw  [color={rgb, 255:red, 74; green, 144; blue, 226 }  ,draw opacity=1 ] (268,145.28) .. controls (268,143.14) and (269.74,141.4) .. (271.88,141.4) -- (288.12,141.4) .. controls (290.26,141.4) and (292,143.14) .. (292,145.28) -- (292,156.92) .. controls (292,159.06) and (290.26,160.8) .. (288.12,160.8) -- (271.88,160.8) .. controls (269.74,160.8) and (268,159.06) .. (268,156.92) -- cycle ;
\draw  [color={rgb, 255:red, 74; green, 144; blue, 226 }  ,draw opacity=1 ] (234.9,254.38) .. controls (234.9,252.24) and (236.64,250.5) .. (238.78,250.5) -- (255.02,250.5) .. controls (257.16,250.5) and (258.9,252.24) .. (258.9,254.38) -- (258.9,266.02) .. controls (258.9,268.16) and (257.16,269.9) .. (255.02,269.9) -- (238.78,269.9) .. controls (236.64,269.9) and (234.9,268.16) .. (234.9,266.02) -- cycle ;
\draw  [color={rgb, 255:red, 74; green, 144; blue, 226 }  ,draw opacity=1 ] (302.1,254.18) .. controls (302.1,252.04) and (303.84,250.3) .. (305.98,250.3) -- (322.22,250.3) .. controls (324.36,250.3) and (326.1,252.04) .. (326.1,254.18) -- (326.1,265.82) .. controls (326.1,267.96) and (324.36,269.7) .. (322.22,269.7) -- (305.98,269.7) .. controls (303.84,269.7) and (302.1,267.96) .. (302.1,265.82) -- cycle ;
\draw (248.9,272) node [anchor=north west][inner sep=0.75pt]    {$x_{\infty }$};
\draw (316.1,272) node [anchor=north west][inner sep=0.75pt]    {$y_{\infty }$};
\draw (294,21.68) node [anchor=north west][inner sep=0.75pt]    {$x_{0}$};
\draw (294,64.68) node [anchor=north west][inner sep=0.75pt]    {$y_{0}$};
\draw (294,106.68) node [anchor=north west][inner sep=0.75pt]    {$x_{1}$};
\draw (294,148.68) node [anchor=north west][inner sep=0.75pt]    {$y_{1}$};
\end{tikzpicture}
\end{center}
\end{exm}
Observe that $x_\infty$ is not an $R$-maximal point of any clopen subset of $X$. We conclude this section by giving an example of a descriptive \textbf{MGrz}-frame $(X,R,E)$, a clopen subset $U$ of $X$, and $x\in\textbf{max}_R U$ such that $E_Q[x]$ properly contains $E[x]$.

\begin{exm}\label{example}
Consider the frame $\mathfrak{F}=(X,R,E)$ depicted below, where the black arrows and circles indicate the relation $R$, and the $E$-clusters are shown in blue. Clearly $R$ is a quasi-order and $E$ is an equivalence relation. A direct inspection shows that commutativity holds, and hence $\mathfrak F$ is an \textbf{MS4}-frame. The points $x_n,y_n,z_n$ are isolated, $(x_n) \longrightarrow x_\infty$, $(y_n) \longrightarrow y_\infty$, and $(z_n) \longrightarrow z_\infty$. Thus, $X$ is the 3-point compactification of a discrete space, hence is a Stone space. We first show that $R$ is a continuous relation. Let 
\[
A=\{x_n\}_{n<\omega}\cup\{x_\infty\}, \ B=\{y_n\}_{n<\omega}\cup\{y_\infty\}, \mbox{ and } C=\{z_n\}_{n<\omega}\cup\{z_\infty\}.
\]
These sets partition $X$ into three clopen sets. 
Let $U\subseteq X$ be clopen.
If $U$ is disjoint from 
$A$ and $B$, then $R^{-1}[U]=U$. Otherwise,
$R^{-1}[U]$ is a cofinite subset of $X$ containing all three limit points. Therefore, $R^{-1}[U]$ is clopen. 
Next, we show that $R[t]$ is closed for each
$t\in X$. If $t$
is an isolated 
point, then $R[t]$ is a finite set of isolated points, 
hence is closed. 
If $t$ is a limit point, then $R[t]$ is closed 
because 
$R[x_\infty]=R[y_\infty]=A\cup B$ and 
$R[z_\infty]=A\cup B\cup\{z_\infty\}$.
Thus, 
$R$ is a continuous relation.

We next show that $E$ is a continuous relation. Let $U \subseteq X$ be clopen. 
If $U$ is contained in $B$, then $E[U]=U$. Let $U\subseteq A \cup C$. If $U$ is finite, then $E[U]$ is a finite set of isolated points, hence is clopen. If $U$ is infinite, then $E[U]$ contains all but finitely many equivalence classes of isolated points in $A\cup C$, and thus 
is clopen again. The clopen $U$ may fall into one of these categories or be a finite union of such clopens. Therefore, $E[U]$ is clopen. Lastly, it is straightforward to see that each $E$-equivalence class is finite, and hence closed. Consequently, $E$ is a continuous relation.



\begin{center}
\centering
\tikzset{every picture/.style={line width=0.75pt}} 
\begin{tikzpicture}[x=0.75pt,y=0.75pt,yscale=-1,xscale=1,scale=0.8]
\draw   (297.16,20.49) .. controls (295.13,22.49) and (291.88,22.48) .. (289.91,20.47) .. controls (287.93,18.47) and (287.97,15.22) .. (290,13.22) .. controls (292.03,11.22) and (295.28,11.23) .. (297.25,13.24) .. controls (299.23,15.25) and (299.19,18.49) .. (297.16,20.49) -- cycle ;
\draw    (293.58,24.01) -- (293.58,47.37) ;
\draw [shift={(293.58,22.01)}, rotate = 90] [color={rgb, 255:red, 0; green, 0; blue, 0 }  ][line width=0.75]    (10.93,-4.9) .. controls (6.95,-2.3) and (3.31,-0.67) .. (0,0) .. controls (3.31,0.67) and (6.95,2.3) .. (10.93,4.9)   ;
\draw   (288.48,53.2) .. controls (288.48,50.35) and (290.76,48.04) .. (293.58,48.04) .. controls (296.4,48.04) and (298.68,50.35) .. (298.68,53.2) .. controls (298.68,56.05) and (296.4,58.35) .. (293.58,58.35) .. controls (290.76,58.35) and (288.48,56.05) .. (288.48,53.2) -- cycle ;
\draw    (293.58,60.35) -- (293.58,83.71) ;
\draw [shift={(293.58,58.35)}, rotate = 90] [color={rgb, 255:red, 0; green, 0; blue, 0 }  ][line width=0.75]    (10.93,-4.9) .. controls (6.95,-2.3) and (3.31,-0.67) .. (0,0) .. controls (3.31,0.67) and (6.95,2.3) .. (10.93,4.9)   ;
\draw   (297.16,92.33) .. controls (295.13,94.33) and (291.89,94.32) .. (289.91,92.32) .. controls (287.93,90.31) and (287.97,87.06) .. (290,85.06) .. controls (292.03,83.07) and (295.27,83.07) .. (297.25,85.08) .. controls (299.23,87.08) and (299.19,90.33) .. (297.16,92.33) -- cycle ;
\draw    (293.58,95.85) -- (293.58,119.21) ;
\draw [shift={(293.58,93.85)}, rotate = 90] [color={rgb, 255:red, 0; green, 0; blue, 0 }  ][line width=0.75]    (10.93,-4.9) .. controls (6.95,-2.3) and (3.31,-0.67) .. (0,0) .. controls (3.31,0.67) and (6.95,2.3) .. (10.93,4.9)   ;
\draw   (288.48,124.19) .. controls (288.48,121.35) and (290.76,119.04) .. (293.58,119.04) .. controls (296.4,119.04) and (298.68,121.35) .. (298.68,124.19) .. controls (298.68,127.04) and (296.4,129.35) .. (293.58,129.35) .. controls (290.76,129.35) and (288.48,127.04) .. (288.48,124.19) -- cycle ;
\draw    (293.58,131.35) -- (293.58,154.71) ;
\draw [shift={(293.58,129.35)}, rotate = 90] [color={rgb, 255:red, 0; green, 0; blue, 0 }  ][line width=0.75]    (10.93,-4.9) .. controls (6.95,-2.3) and (3.31,-0.67) .. (0,0) .. controls (3.31,0.67) and (6.95,2.3) .. (10.93,4.9)   ;
\draw  [fill={rgb, 255:red, 0; green, 0; blue, 0 }  ,fill opacity=1 ][line width=0.75]  (292.75,165.7) .. controls (292.74,165.23) and (293.12,164.85) .. (293.58,164.85) .. controls (294.04,164.85) and (294.41,165.23) .. (294.41,165.69) .. controls (294.41,166.16) and (294.04,166.54) .. (293.58,166.54) .. controls (293.12,166.54) and (292.75,166.16) .. (292.75,165.7) -- cycle ;
\draw  [fill={rgb, 255:red, 0; green, 0; blue, 0 }  ,fill opacity=1 ][line width=0.75]  (292.75,182.6) .. controls (292.74,182.13) and (293.12,181.76) .. (293.58,181.76) .. controls (294.04,181.75) and (294.41,182.13) .. (294.41,182.6) .. controls (294.41,183.06) and (294.04,183.44) .. (293.58,183.44) .. controls (293.12,183.44) and (292.75,183.07) .. (292.75,182.6) -- cycle ;
\draw  [fill={rgb, 255:red, 0; green, 0; blue, 0 }  ,fill opacity=1 ][line width=0.75]  (292.75,199.5) .. controls (292.74,199.04) and (293.12,198.66) .. (293.58,198.66) .. controls (294.04,198.66) and (294.41,199.03) .. (294.41,199.5) .. controls (294.41,199.97) and (294.04,200.34) .. (293.58,200.34) .. controls (293.12,200.35) and (292.75,199.97) .. (292.75,199.5) -- cycle ;
\draw   (322.09,211.08) .. controls (324.91,211.08) and (327.2,213.39) .. (327.2,216.24) .. controls (327.2,219.08) and (324.91,221.39) .. (322.09,221.39) .. controls (319.28,221.39) and (316.99,219.08) .. (316.99,216.24) .. controls (316.99,213.39) and (319.28,211.08) .. (322.09,211.08) -- cycle ;
\draw    (314.99,216.24) -- (273,216.4) ;
\draw [shift={(271,216.4)}, rotate = 359.79] [color={rgb, 255:red, 0; green, 0; blue, 0 }  ][line width=0.75]    (10.93,-4.9) .. controls (6.95,-2.3) and (3.31,-0.67) .. (0,0) .. controls (3.31,0.67) and (6.95,2.3) .. (10.93,4.9)   ;
\draw [shift={(316.99,216.24)}, rotate = 179.79] [color={rgb, 255:red, 0; green, 0; blue, 0 }  ][line width=0.75]    (10.93,-4.9) .. controls (6.95,-2.3) and (3.31,-0.67) .. (0,0) .. controls (3.31,0.67) and (6.95,2.3) .. (10.93,4.9)   ;
\draw   (265.9,211.25) .. controls (268.72,211.25) and (271,213.56) .. (271,216.4) .. controls (271,219.25) and (268.72,221.56) .. (265.9,221.56) .. controls (263.08,221.56) and (260.8,219.25) .. (260.8,216.4) .. controls (260.8,213.56) and (263.08,211.25) .. (265.9,211.25) -- cycle ;
\draw    (265.9,223.56) -- (265.9,246.92) ;
\draw [shift={(265.9,221.56)}, rotate = 90] [color={rgb, 255:red, 0; green, 0; blue, 0 }  ][line width=0.75]    (10.93,-4.9) .. controls (6.95,-2.3) and (3.31,-0.67) .. (0,0) .. controls (3.31,0.67) and (6.95,2.3) .. (10.93,4.9)   ;
\draw   (265.9,246.92) .. controls (268.72,246.92) and (271,249.22) .. (271,252.07) .. controls (271,254.92) and (268.72,257.23) .. (265.9,257.23) .. controls (263.08,257.23) and (260.8,254.92) .. (260.8,252.07) .. controls (260.8,249.22) and (263.08,246.92) .. (265.9,246.92) -- cycle ;
\draw   (64.14,249.23) .. controls (66.06,247.13) and (69.3,246.96) .. (71.39,248.86) .. controls (73.47,250.75) and (73.6,254) .. (71.68,256.1) .. controls (69.76,258.21) and (66.52,258.37) .. (64.44,256.48) .. controls (62.36,254.58) and (62.22,251.33) .. (64.14,249.23) -- cycle ;
\draw   (135.96,249.25) .. controls (137.87,247.14) and (141.11,246.95) .. (143.2,248.84) .. controls (145.3,250.72) and (145.45,253.96) .. (143.54,256.08) .. controls (141.64,258.2) and (138.4,258.38) .. (136.3,256.5) .. controls (134.21,254.61) and (134.06,251.37) .. (135.96,249.25) -- cycle ;
\draw  [fill={rgb, 255:red, 0; green, 0; blue, 0 }  ,fill opacity=1 ][line width=0.75]  (186.75,253.5) .. controls (186.29,253.5) and (185.91,253.13) .. (185.91,252.67) .. controls (185.91,252.21) and (186.28,251.83) .. (186.75,251.83) .. controls (187.22,251.83) and (187.59,252.21) .. (187.59,252.67) .. controls (187.6,253.13) and (187.22,253.5) .. (186.75,253.5) -- cycle ;
\draw  [fill={rgb, 255:red, 0; green, 0; blue, 0 }  ,fill opacity=1 ][line width=0.75]  (203.66,253.5) .. controls (203.19,253.5) and (202.81,253.13) .. (202.81,252.67) .. controls (202.81,252.21) and (203.19,251.83) .. (203.65,251.83) .. controls (204.12,251.83) and (204.5,252.21) .. (204.5,252.67) .. controls (204.5,253.13) and (204.12,253.5) .. (203.66,253.5) -- cycle ;
\draw  [fill={rgb, 255:red, 0; green, 0; blue, 0 }  ,fill opacity=1 ][line width=0.75]  (220.56,253.5) .. controls (220.1,253.5) and (219.72,253.13) .. (219.72,252.67) .. controls (219.72,252.21) and (220.09,251.83) .. (220.56,251.83) .. controls (221.02,251.83) and (221.4,252.21) .. (221.4,252.67) .. controls (221.4,253.13) and (221.03,253.5) .. (220.56,253.5) -- cycle ;
\draw  [color={rgb, 255:red, 74; green, 144; blue, 226 }  ,draw opacity=1 ] (57.18,254.99) .. controls (55.48,253.37) and (55.41,250.69) .. (57.03,248.99) -- (288.88,5.69) .. controls (290.5,3.99) and (293.19,3.93) .. (294.89,5.54) -- (304.12,14.34) .. controls (305.81,15.95) and (305.88,18.64) .. (304.26,20.34) -- (72.41,263.64) .. controls (70.79,265.34) and (68.1,265.4) .. (66.4,263.79) -- cycle ;
\draw  [color={rgb, 255:red, 74; green, 144; blue, 226 }  ,draw opacity=1 ] (127.25,255.07) .. controls (125.55,253.46) and (125.49,250.77) .. (127.1,249.07) -- (290.58,77.53) .. controls (292.2,75.83) and (294.88,75.77) .. (296.58,77.38) -- (305.79,86.16) .. controls (307.48,87.78) and (307.55,90.46) .. (305.93,92.16) -- (142.46,263.71) .. controls (140.84,265.4) and (138.16,265.47) .. (136.46,263.85) -- cycle ;
\draw  [color={rgb, 255:red, 74; green, 144; blue, 226 }  ,draw opacity=1 ] (256,209.6) .. controls (256,207.39) and (257.79,205.6) .. (260,205.6) -- (272,205.6) .. controls (274.21,205.6) and (276,207.39) .. (276,209.6) -- (276,258.6) .. controls (276,260.81) and (274.21,262.6) .. (272,262.6) -- (260,262.6) .. controls (257.79,262.6) and (256,260.81) .. (256,258.6) -- cycle ;
\draw    (71.39,248.86) -- (288.52,21.92) ;
\draw [shift={(289.91,20.47)}, rotate = 133.74] [color={rgb, 255:red, 0; green, 0; blue, 0 }  ][line width=0.75]    (10.93,-3.29) .. controls (6.95,-1.4) and (3.31,-0.3) .. (0,0) .. controls (3.31,0.3) and (6.95,1.4) .. (10.93,3.29)   ;
\draw    (143.2,248.84) -- (288.54,93.78) ;
\draw [shift={(289.91,92.32)}, rotate = 133.15] [color={rgb, 255:red, 0; green, 0; blue, 0 }  ][line width=0.75]    (10.93,-3.29) .. controls (6.95,-1.4) and (3.31,-0.3) .. (0,0) .. controls (3.31,0.3) and (6.95,1.4) .. (10.93,3.29)   ;
\draw  [color={rgb, 255:red, 74; green, 144; blue, 226 }  ,draw opacity=1 ] (286.58,48.2) .. controls (286.58,46.65) and (287.83,45.4) .. (289.38,45.4) -- (297.78,45.4) .. controls (299.33,45.4) and (300.58,46.65) .. (300.58,48.2) -- (300.58,58.2) .. controls (300.58,59.75) and (299.33,61) .. (297.78,61) -- (289.38,61) .. controls (287.83,61) and (286.58,59.75) .. (286.58,58.2) -- cycle ;
\draw  [color={rgb, 255:red, 74; green, 144; blue, 226 }  ,draw opacity=1 ] (286.58,119.19) .. controls (286.58,117.65) and (287.83,116.39) .. (289.38,116.39) -- (297.78,116.39) .. controls (299.33,116.39) and (300.58,117.65) .. (300.58,119.19) -- (300.58,129.2) .. controls (300.58,130.74) and (299.33,132) .. (297.78,132) -- (289.38,132) .. controls (287.83,132) and (286.58,130.74) .. (286.58,129.2) -- cycle ;
\draw  [color={rgb, 255:red, 74; green, 144; blue, 226 }  ,draw opacity=1 ] (315.09,211.24) .. controls (315.09,209.69) and (316.35,208.44) .. (317.89,208.44) -- (326.29,208.44) .. controls (327.84,208.44) and (329.09,209.69) .. (329.09,211.24) -- (329.09,221.24) .. controls (329.09,222.78) and (327.84,224.04) .. (326.29,224.04) -- (317.89,224.04) .. controls (316.35,224.04) and (315.09,222.78) .. (315.09,221.24) -- cycle ;

\draw (227,208.1) node [anchor=north west][inner sep=0.75pt]    {$x_{\infty }$};
\draw (325,224) node [anchor=north west][inner sep=0.75pt]    {$y_{\infty }$};
\draw (305.68,20.26) node [anchor=north west][inner sep=0.75pt]    {$x_{0}$};
\draw (305.68,56.6) node [anchor=north west][inner sep=0.75pt]    {$y_{0}$};
\draw (305.68,92.1) node [anchor=north west][inner sep=0.75pt]    {$x_{1}$};
\draw (305.68,127.59) node [anchor=north west][inner sep=0.75pt]    {$y_{1}$};
\draw (278,263) node [anchor=north west][inner sep=0.75pt]    {$z_{\infty }$};
\draw (53.18,263) node [anchor=north west][inner sep=0.75pt]    {$z_{0}$};
\draw (123,263) node [anchor=north west][inner sep=0.75pt]    {$z_{1}$};
\end{tikzpicture}
\end{center}
We thus proved that $\mathfrak F$
is a descriptive \textbf{MS4}-frame. Moreover, it is a descriptive \textbf{MGrz}-frame by \Cref{fine-esakia} {since} $\{x_\infty,y_\infty\}$ is the only non-trivial $R$-cluster and neither $x_\infty$ nor $y_\infty$ is a quasi-maximal point in any clopen set.
Observe that $z_\infty\in\textbf{max}_R C$. However, $E_Q[z_\infty]=\{x_\infty,y_\infty,z_\infty\}$ and $E[z_\infty]=\{x_\infty,z_\infty\}$.
\end{exm}

\section{Selective filtration for \textbf{MGrz}}{\label{construction}}

In this section, we prove that \textbf{MGrz} has the fmp, thus establishing our main result. 
To this end, we need to show that each non-theorem of \textbf{MGrz} is refuted in a finite \textbf{MGrz}-frame. If $\textbf{MGrz}\not\vdash\varphi$, then $\varphi$ 
is refuted in some descriptive \textbf{MGrz}-frame, from which we will select a finite \textbf{MGrz}-frame that still refutes $\varphi$. We do this by selecting 
strongly maximal points, a sufficient supply of which is provided by \Cref{smax}. To perform selective filtration, we utilize the following:
%

\begin{lem}{\label{max_smax}}
    Let $U$ be a clopen set in a descriptive {\textbf{MGrz}}-frame $\mathfrak{F}=(X,R,E)$ and $x,y,t\in X$.
    \begin{enumerate}
        \item If $x\in{\textbf{smax}}_R U$ and $y\in{\textbf{max}}_R U$ with $x\mathrel{E}y$, then $y\in{\textbf{smax}}_R U$. \label{max_smax1}

        \item If $x\in{\textbf{smax}}_R U$ and $y\in Q[t] \cap U$ with $x\mathrel{E}t$, then $x\mathrel{E}y$. \label{max_smax2}

        \item If $x\in{\textbf{smax}}_R U$ and $y\in R[t] \cap E[U]$ with $x\mathrel{E}t$, then $x\mathrel{E}y$. \label{max_smax3}
    \end{enumerate}
\end{lem}

\begin{proof}
(\labelcref{max_smax1}) Suppose $y\mathrel{Q}t$ and $y\mathrel{\cancel{E}}t$ for some $t\in X$. From $x\mathrel{E}y$ it follows that $x\mathrel{Q}t$ and $x\mathrel{\cancel{E}}t$. Since $x\in\textbf{smax}_R U$, we have $t\notin U$. Consequently,  $y\in{\textbf{smax}}_R U$.

(\labelcref{max_smax2}) First, observe that $t\mathrel{Q}y$ implies $x\mathrel{Q}y$. Since $y\in U$, the strong maximality of $x$ in $U$ forces $x\mathrel{E}y$.

(\labelcref{max_smax3}) Since $y\in E[U]$, there is $u\in U$ with $y\mathrel{E}u$. Therefore, $t\mathrel{R}y$ and $y\mathrel{E}u$, so $t\mathrel{Q}u$. This implies $x\mathrel{Q}u$, which along with the strong maximality of $x$ in $U$ gives $x\mathrel{E}u$. Thus, $x\mathrel{E}y$.
\end{proof}

%

Now let $\textbf{MGrz}\not\vdash\varphi$. Then there is an $\textbf{MGrz}$-algebra $\mathbbm{B}=(B,\Diamond,\exists)$ refuting $\varphi$, and hence $\varphi$ is refuted in the descriptive \textbf{MGrz}-frame $\mathfrak{F}=(X,R,E)$ dual to $\mathbbm{B}$. 
Therefore, there is a valuation $v$ on $\mathfrak F$  
such that $\mathfrak{F}\not\models_v\varphi$. We will select a finite subframe $\widehat{\mathfrak{F}}=(\widehat{X},\widehat{R},\widehat{E})$ of $\mathfrak{F}$ such that $\widehat{\mathfrak{F}}$ is an \textbf{MGrz}-frame and refutes $\varphi$.

{The} particular strategy we adopt combines the approaches of \cite{Grefe} and \cite{GBm}. In the former, maximal points are selected as witnesses, while in the latter, some $Q$-arrows are turned into $\widehat{R}$-arrows. 
 We will build $\widehat{\mathfrak{F}}$ by selecting strongly maximal points from $\mathfrak{F}$ that are necessary for the refutation of $\varphi$.
For each selected point, a new copy will be created. If a point $t$ is selected, we denote the copy by $\widehat{t}$. Further, if a copy $\widehat{t}$ of $t$ has already been added, we will not introduce it again, i.e., there will be no duplicates. This will allow us to control the size of the selected subset of points. The construction proceeds in stages, and at each stage we produce a partially ordered \textbf{MS4}-frame. 

To begin with, let $S:=\text{Sub}(\varphi)$ be the set of subformulae of $\varphi$. Define an equivalence relation $\sim_S$ on $X$ by
\begin{equation*}
    x \mathrel{\sim_S} y \iff \text{ for any } \psi\in S, \  x\models_v \psi \mbox{ iff } y\models_v \psi.
\end{equation*}

For each $x\in X$, let
\begin{equation*}
    W_x^{\exists}=\{\exists\psi\in S \mid x\models_v\exists\psi \text{ but } x\not\models_v \psi\},
\end{equation*}
\begin{equation*}
    W_x^{\Diamond}=\{\Diamond\psi\in S \mid x\models_v\Diamond\psi \text{ but } x\not\models_v \psi\}.
\end{equation*}
Here, the letter $W$ stands for ``witnesses'', which we refer to as $\exists$-witnesses and $\Diamond$-witnesses.

Since $\mathfrak{F}\not\models_v\varphi$, there is $u\in X$ such that $u\in v(\neg\varphi)$. By \Cref{smax}, there is 
$x\in Q[u]\cap\textbf{smax}_R v(\neg\varphi)$. 
Let
\begin{equation*}
    X_0=\{\widehat{x}\},\enspace R_0=X_0^2,\enspace E_0=X_0^2,
\end{equation*}
and set $\mathfrak{F}_0=(X_0,R_0,E_0)$. Clearly $\mathfrak{F}_0$ is a partially ordered \textbf{MS4}-frame. 

Suppose a partially ordered \textbf{MS4}-frame $\mathfrak{F}_{k-1}=(X_{k-1},R_{k-1},E_{k-1})$ has already been constructed. 
We construct the frame $\mathfrak{F}_k=(X_k,R_k,E_k)$ in stages by first introducing all the necessary $\exists$-witnesses, then the $\Diamond$-witnesses within given $E$-clusters, 
after that the remaining $\Diamond$-witnesses, and finally ensuring that commutativity holds.  For the reader's convenience, the stages of the construction are illustrated in the diagram before \Cref{rmk}. 

\underline{\textbf{$\exists$-Step}}: 
Set $X_k^\exists=X_{k-1}$, $R_k^\exists=R_{k-1}$, $E_k^\exists=E_{k-1}$.

Let $\widehat{t}\in X_{k-1}$ and $\exists\psi\in W_t^\exists$. If $\widehat{t} \mathrel{E_{k}^\exists} \widehat{u}$ for some $\widehat{u}\in X_k^\exists$ with $u\models_v\psi$ in $(\mathfrak{F},v)$, 
there is nothing to do. Otherwise, we make use of the following lemma to select $\exists$-witnesses. 

\begin{lem}{\label{horz}}
    Let $z\in{\textbf{smax}}_R U$ for some clopen $U$ of $X$. If $\exists\psi\in W_z^{\exists}$, then there is $y \in {\textbf{smax}}_R (E[U]\cap v(\psi))$ such that $y \mathrel{E} z$. 
\end{lem}
\begin{proof}
    Since $\exists\psi\in W_z^{\exists}$, there is $u\in v(\psi)$ with $z\mathrel{E}u$. Because $u\in E[U] \cap v(\psi)$, by the Fine-Esakia principle there is $y\in R[u]\cap \textbf{max}_R(E[U]\cap v(\psi))$. We show that $y\mathrel{E}z$. Since $y\in E[U]$, there is $s\in U$ with $s\mathrel{E}y$. From $u\mathrel{R}y$ and $y\mathrel{E}s$ it follows that $u\mathrel{Q}s$. Because $z\in\textbf{smax}_R U$ and $z\mathrel{E}u$, we must have $z\mathrel{E}s$ by \crefdefpart{max_smax}{max_smax2}, so $y\mathrel{E}z$. It remains to show that $y\in\textbf{smax}_R (E[U]\cap v(\psi))$. Let $y\mathrel{Q}t$ and $t\in E[U]\cap v(\psi)$. Then $t\mathrel{E}w$ for some $w\in U$. Since $z\in{\textbf{smax}}_R U$ and $z\mathrel{Q}w$, we must have $z\mathrel{E}w$, so $y\mathrel{E}w$, and hence $y\mathrel{E}t$. Thus, $y\in\textbf{smax}_R (E[U]\cap v(\psi))$.
\end{proof}

If there is $\widehat{u}$ in $X_k^\exists$ such that $t\mathrel{E}u$, $u\models_v\psi$, and $u$ is strongly maximal in some clopen set in $X$, then add $(\widehat{t},\widehat{u})$ to $E_k^\exists$ and generate the least equivalence relation. Otherwise, since $\widehat{t}\in X_{k-1}$, we have $t\in\textbf{smax}_R U$ for some clopen $U$ (because we only select strongly maximal points at each step). Hence, by \Cref{horz}, there is $w\in \textbf{smax}_R (E[U]\cap v(\psi))$ such that $t\mathrel{E}w$ and $w\models_v\psi$. We add $\widehat{w}$ to $X_k^\exists$, $(\widehat{w},\widehat{w})$ to $R_k^\exists$, $(\widehat{t},\widehat{w})$ to $E_k^\exists$, and generate the least equivalence relation. Repeat this process until all the formulae in $W_t^\exists$ have $\exists$-witnesses for all $\widehat{t}\in X_{k-1}$ and proceed to the next step. 

\underline{\textbf{$\Diamond$-Step}}: Set $X_k^\Diamond=X_k^\exists$, $R_k^\Diamond=R_k^\exists$, $E_k^\Diamond=E_k^\exists$.

Suppose $\widehat{y}\in X_k^\exists$ and $\Diamond\psi\in W_y^\Diamond$. If $\widehat{y} \mathrel{R_{k}^\Diamond}  \widehat{z}$ for some $\widehat{z}\in X_k^\Diamond$ and $z\models_v\psi$ in $(\mathfrak{F},v)$, there is nothing to do. Otherwise, we use the following lemma to select $\Diamond$-witnesses.

\begin{lem}{\label{vert}}
    For $\Diamond\psi\in W_y^{\Diamond}$ let 
    \begin{equation*}
            A = v(\Diamond\psi)\cap\bigcap\bigl\{ v(\neg\Diamond\alpha)\mid \Diamond\alpha\in S,\,  y\not\models_v \Diamond\alpha \bigr\}.
        \end{equation*}
    Then there is $z\in X$ such that the following conditions are satisfied:
    \begin{enumerate}
        \item $y\mathrel{Q}z$, $y\neq z$, and $z\in\textbf{smax}_R A\cap\textbf{max}_R v(\psi)$; \label{vert3}
        
        \item $y\not\models_v\Diamond\alpha$ implies $z\not\models_v\alpha$ for each $\Diamond\alpha\in S$; \label{vert2}
        
        \item If $y\mathrel{E}z$, then there is $u\in {\textbf{smax}}_R A \cap {\textbf{max}}_R v(\psi)$ such that $y\mathrel{R}u$ and $y\mathrel{E}u$; \label{vert4}

        \item If $z\mathrel{Q}t$ and $z \mathrel{\cancel{E}} t$, then $t\mathrel{\cancel{\sim}_S}y$. \label{vert6}
    \end{enumerate}
\end{lem}

\begin{proof}
(\labelcref{vert3}) Since $S$ is finite, $A$ is a clopen set. Clearly, $y\in A$. Therefore, by \Cref{smax}, there exists $z\in Q[y]\cap\textbf{smax}_R A$. 
We show that $z\in\textbf{max}_R v(\psi)$. Since $z\in A$, we have $z\models_v\Diamond\psi$, so $t\models_v\psi$ for some $t\in R[z]$. If $t\neq z$, then $z\mathrel{R}t$ and $z\in\textbf{smax}_R A\subseteq\textbf{max}_R A$ imply that $t\notin A$. However, ${t\in \bigcap\bigl\{ v(\neg\Diamond\alpha)\mid \Diamond\alpha\in S,\, y\not\models_v \Diamond\alpha\bigr\}}$ because $z\mathrel{R}t$ and $z\in v(\neg\Diamond\alpha)$ for all $\Diamond\alpha\in S$. Therefore, $t\notin v(\Diamond\psi)$, and hence $t\not\models_v\psi$, a contradiction. Thus, $t=z$, and so $z\in\textbf{max}_R v(\psi)$. Consequently, $z\in\textbf{smax}_R A\cap\textbf{max}_R v(\psi)$. Finally, since $z\models_v\psi$ and $y\not\models_v\psi$ (because $\Diamond\psi\in W_y^{\Diamond}$), we see that $z\neq y$.

(\labelcref{vert2}) Let $\Diamond\alpha\in S$ and $y\not\models_v\Diamond\alpha$. Then $A\subseteq v(\neg\Diamond\alpha)$. Since $z\in A$, we see that $z\in v(\neg\Diamond\alpha)$, so $z\not\models_v\Diamond\alpha$. Thus, $z\not\models_v\alpha$.

(\labelcref{vert4}) Suppose $y\mathrel{E}z$. Since 
$y\models_v\Diamond\psi$, we have $t\models_v\psi$ for some $t\in R[y]$. Choose $u\in R[t]\cap \textbf{max}_Rv(\psi)$. Then $y\mathrel{R}u$. We show that $u\in\textbf{max}_R v(\Diamond\psi)$. Clearly $u\in v(\Diamond\psi)$. Suppose $u\mathrel{R}w$, $u\neq w$, and $w\models_v\Diamond\psi$. Then $w\mathrel{R}u'$ for some $u'\in v(\psi)$. Therefore, $u\mathrel{R}u'$, and so $u=u'$ because $u\in\textbf{max}_R v(\psi)$. We have $u\mathrel{R}w\mathrel{R}u$, with $u\in\textbf{max}_R v(\psi)$ and $w\notin v(\psi)$ (because  $u\in\textbf{max}_R v(\psi)$, $u\mathrel{R}w$, and $u\neq w$). This cannot happen since $(X,R,E)$ is a descriptive \textbf{MGrz}-frame (see \Cref{enter-exit}). Thus, $u\in \textbf{max}_R v(\Diamond\psi)$. We use this to prove that $u\in\textbf{smax}_R A$. First note that $u\in A$ since $y\mathrel{R}u$ implies that $u\in\bigcap\bigl\{ v(\neg\Diamond\alpha)\mid \Diamond\alpha\in S,  y\not\models_v \Diamond\alpha\bigr\}$. Moreover, $u\in\textbf{max}_R v(\Diamond\psi)$ implies that $u\in \textbf{max}_R A$ since $A\subseteq v(\Diamond\psi)$ and $u\in A$. We must have $y\mathrel{E}u$ as otherwise we get $z\mathrel{Q}u$, $z \mathrel{\cancel{E}} u$, and $u\in A$, contradicting that $z\in \textbf{smax}_R A$ (see  (\labelcref{vert3})). Thus, $u\in \textbf{smax}_R A$ by \crefdefpart{max_smax}{max_smax1}, and hence $u$ satisfies all the conditions in~(\labelcref{vert4}).

(\labelcref{vert6}) Suppose $z\mathrel{Q}t$, $z \mathrel{\cancel{E}} t$, and $t\mathrel{\sim_S} y$. Then $t\in A$, contradicting 
that $z\in \textbf{smax}_R A$ (see (\labelcref{vert3})). 
\end{proof}


\underline{\textbf{Horizontal Step}}: Adding $\Diamond$-witnesses within $E$-clusters.

If there is $\widehat{u} \in X_k^\Diamond$ such that $y\mathrel{R}u$, $y\mathrel{E}u$, $u\in\textbf{smax}_R A$, and $u\models_v\psi$ in $(\mathfrak{F},v)$, then add $(\widehat{y},\widehat{u})$ to $R_k^\Diamond$ and take the reflexive-transitive closure; also, add $(\widehat{y},\widehat{u})$ to $E_k^\Diamond$ and generate the least equivalence relation.

Otherwise, by \crefdefpart{vert}
{vert3}, 
there is $z\in \textbf{smax}_R A$ such that $y\mathrel{Q}z$ and $z\models_v\psi$. If $y\mathrel{E}z$, then by \crefdefpart{vert}{vert4}, there is $u\in \textbf{smax}_R A$ with $y\mathrel{R}u$, $y\mathrel{E}u$, and $u\models_v\psi$. Add $\widehat{u}$ to $X_k^\Diamond$, $(\widehat{y},\widehat{u})$ to $R_k^\Diamond$, and take the reflexive-transitive closure; also, add $(\widehat{y},\widehat{u})$ to $E_k^\exists$ and generate the least equivalence relation.  
Repeat this process until all $\Diamond$-witnesses from $E[y]$ for $\widehat{y}$ have been added. 
Repeat this for all $\widehat{y}\in X_k^\exists$ 
and proceed to the next step.

\underline{\textbf{Vertical Step}}: Adding $\Diamond$-witnesses outside $E$-clusters.

If there is already a point $\widehat{z}\in X_k^\Diamond$ such that $y\mathrel{Q}z$, $y \mathrel{\cancel{E}} z$, $z\in\textbf{smax}_R A$, and $z\models_v \psi$ in $(\mathfrak{F},v)$, then add $(\widehat{y},\widehat{z})$ to $R_k^\Diamond$ and take its reflexive-transitive closure.

If no such point exists, by \crefdefpart{vert}
{vert3} 
there is $z\in \textbf{smax}_R A$ such that $y\mathrel{Q}z$ and $z\models_v\psi$. Moreover, we must have $z \mathrel{\cancel{E}} y$ since the $z\mathrel{E}y$ case is handled in the horizontal step. Add $\widehat{z}$ to $X_k^\Diamond$, add $(\widehat{y},\widehat{z})$ to $R_k^\Diamond$, and take the reflexive-transitive closure. Add 
$(\widehat{z},\widehat{z})$ to $E_k^\Diamond$.

\textbf{Note}: This is the only step in the construction where a $Q$-arrow in the original frame is turned into an $R_k^\Diamond$-arrow in the selected subframe. 

If $z\mathrel{E}t$ in $X$ and $\widehat{t}$ has been added previously, add $(\widehat{z},\widehat{t})$ to $E_k^\Diamond$ and generate the least equivalence relation. Add all such $\Diamond$-witnesses for $\widehat{y}$. Repeat this for all $\widehat{y}\in X_k^\exists$.  

\begin{lem}\label{EQ-lemma}
    For $\widehat{u},\widehat{w}\in X_k^\Diamond$, 
    \begin{enumerate}
        \item $\widehat{u}\mathrel{E_k^\Diamond}\widehat{w} \iff u\mathrel{E}w$. {\label{E-lemma}}

        \item $\widehat{u}\mathrel{R_k^\Diamond}\widehat{w}\implies u\mathrel{Q}w$. {\label{Q-lemma}}
    \end{enumerate}
\end{lem}
\begin{proof}
    (\labelcref{E-lemma}) For the left-to-right implication, note that if $\widehat{u}\mathrel{E_k^\Diamond}\widehat{w}$, then we must have $u\mathrel{E}w$ because we only introduce points into ${E}_k^\Diamond$-clusters if their original copies are in $E$-clusters. Thus, $\widehat{u}\mathrel{E_k^\Diamond}\widehat{w}$ implies $u\mathrel{E}w$. For the right-to-left implication, each time a new $E_k^\Diamond$-cluster is introduced in the vertical step, we add the necessary $E_k^\Diamond$ relations to ensure $u\mathrel{E}w\implies\widehat{u}\mathrel{E_k^\Diamond}\widehat{w}$.

    (\labelcref{Q-lemma}) We introduce new $R_k^\Diamond$-relations only if they correspond to $R$- or $Q$-relations in the original frame, and $Q$ is transitive.
\end{proof}

\underline{\textbf{Commutativity Step}}: 

    We now augment the sets $X_k^\Diamond, R_k^\Diamond, E_k^\Diamond$ so that the resulting structure satisfies commutativity. We point out that the newly added points may again require closure under commutativity, so this part of the construction must be repeated. 

    First, set $X_k^0=X_k^\Diamond, R_k^0=R_k^\Diamond, E_k^0=E_k^\Diamond$. 
    Suppose $X_k^{j-1}, R_k^{j-1}, E_k^{j-1}$ have already been defined. Let ${X_k^j=X_k^{j-1}}$,  $R_k^j=R_k^{j-1}$, and $E_k^j=E_k^{j-1}$. We extend $X_k^j$, $R_k^j$, and $E_k^j$ so that commutativity is satisfied for points in $X_{j-1}$. At each stage indexed by $j$, we will ensure that an analog of \Cref{EQ-lemma} is applicable (see \Cref{lemma_lc}). 
    
    Suppose $\widehat{t}\mathrel{E_k^{j-1}}\widehat{u}$ and $\widehat{u}\mathrel{R_k^{j-1}}\widehat{w}$. We need to find some $\widehat{s}$ such that $\widehat{t}\mathrel{R_k^j}\widehat{s}$ and $\widehat{s}\mathrel{E_k^j}\widehat{w}$. In case $\widehat{u}\mathrel{E_k^{j-1}}\widehat{w}$, the choice of $\widehat{s}$ is trivial since we can simply take $\widehat{s}=\widehat{t}$. Hence, suppose $\widehat{u}\mathrel{\cancel{E_k^{j-1}}}\widehat{w}$. Then we can find a witness $\widehat{s}$ using the following lemma.

\begin{lem}\label{l-comm}
    Let $u\mathrel{Q}t$, where $t\in\textbf{smax}_R U$ for some clopen $U$. Then there is $s\in\textbf{smax}_RE[U]$ such that $u\mathrel{R}s$ and $s\mathrel{E}t$.
\end{lem}
\begin{proof}
    Since $u\mathrel{Q}t$, we have $u\mathrel{R}c$ and $c\mathrel{E}t$ for some $c\in X$. Note that $c\in E[U]$ because $t\in U$. By \Cref{fe_principle}, there is $s\in R[c]\cap\textbf{max}_RE[U]$, and so 
    $s\in\textbf{smax}_RE[U]$ by \crefdefpart{sat-smax}{smax=max}. Clearly $u\mathrel{R}s$, also $t\in\textbf{smax}_R U$ and $s\in R[c]\cap E[U]$ with $t\mathrel{E}c$ imply $s\mathrel{E}t$ by \crefdefpart{max_smax}{max_smax3}.
\end{proof}

    Since $\widehat{t}\mathrel{E_k^{j-1}}\widehat{u}$ and $\widehat{u}\mathrel{R_k^{j-1}}\widehat{w}$, \Cref{lemma_lc} for $j-1$ (or \Cref{EQ-lemma} if $j=1$) gives that $t\mathrel{E}u$ and $u\mathrel{Q}w$. Thus, $t\mathrel{Q}w$, where $w\in\textbf{smax}_RU$ for some clopen $U$ (since every selected point is strongly maximal in some clopen). Therefore, \Cref{l-comm} applies, by which there is $s\in\textbf{smax}\hspace{2pt}E[U]$ such that $t\mathrel{R}s$ and $s\mathrel{E}w$.

    If $\widehat{s}$ is not already present, add it to $X_k^\Diamond$. Also add $(\widehat{t},\widehat{s})$ to $R_k^\Diamond$ and take the reflexive-transitive closure. In addition, add $(\widehat{s},\widehat{w})$ to $E^\Diamond_k$ and generate the least equivalence relation. Repeat this process for every instance of $\widehat{t}\mathrel{E_k^\Diamond}\widehat{u}$ and $\widehat{u}\mathrel{R_k^\Diamond}\widehat{w}$ until commutativity is satisfied. While this appears as a potentially infinite process, we will see in \Cref{partial} that it terminates. 

    \begin{lem}{\label{lemma_lc}}
    For $\widehat{u},\widehat{w}\in X_k^j$,
    \begin{enumerate}[label={(\arabic*)}]
        \item $\widehat{u}\mathrel{E_k^j}\widehat{w}$ iff $u\mathrel{E}w$. \label{E_lc}

        \item $\widehat{u}\mathrel{R_k^j}\widehat{w}$ implies $u\mathrel{Q}w$. \label{Q_lc}
    \end{enumerate}
    \end{lem}
    \begin{proof}
        \labelcref{E_lc} The implication $\widehat{u}\mathrel{E_k^j}\widehat{w}\implies u\mathrel{E}w$ is clear since we introduce $E_k^j$-relations between points only if their original copies are $E$-related. Conversely, suppose $u\mathrel{E}w$. Since $\widehat{u},\widehat{w}\in X_k^j$, 
        there are $\widehat{t}, \widehat{s}\in X_k^{j-1}$ such that $\widehat{w}\mathrel{E_k^j}\widehat{t}$ and $\widehat{u}\mathrel{E_k^j}\widehat{s}$ (because each point that may be introduced by left commutativity must involve creating an $E_k^j$-relation). Thus, by the forward implication, 
        $w\mathrel{E}t$ and $u\mathrel{E}s$, and hence {$t\mathrel{E}s$}. By \Cref{lemma_lc} for $j-1$ (or \Cref{EQ-lemma} when $j=1$), we deduce that {$\widehat{t}\mathrel{E_k^{j-1}}\widehat{s}$}, and so $\widehat{t}\mathrel{E_k^j}\widehat{s}$. 
        Thus, $\widehat{u}\mathrel{E_k^j}\widehat{w}$.

        \labelcref{Q_lc} 
        This is immediate since a $R_k^j$-relation is introduced between points only if their original copies are $Q$-related (in fact, the $R_k^j$-relations introduced in the commutativity step correspond to $R$-relations in the original frame).
    \end{proof}

    Set $X_k=\bigcup_{j<\omega}X_k^j, R_k=\bigcup_{j<\omega}R_k^j, E_k=\bigcup_{j<\omega}E_k^j$. We have thus constructed 
    $\mathfrak{F}_k=(X_k,R_k,E_k)$.

\textbf{Note:} The $R_k$-arrows introduced in each step of the construction correspond to the following arrows in the original frame:
\begin{center}
    \begin{tabular}{|l|l|}
    \hline 
    \textbf{Horizontal Step}     &  $R$ \\
    \hline
    \textbf{Vertical Step}     &  $Q$ \\
    \hline
    \textbf{Commutativity Step} & $R$ \\
    \hline
    \end{tabular}
\end{center}

The construction can be summarized pictorially as follows, where circles and arrows indicate the relation $R_k$, and $E_k$-clusters are shown in blue. The newly added relations in each step are highlighted in red, while those retained from previous stages are in black and gray. 

\begin{center}
\tikzset{every picture/.style={line width=0.75pt}} 

\begin{tikzpicture}[x=0.75pt,y=0.75pt,yscale=-1,xscale=1]

\draw  [color={rgb, 255:red, 74; green, 144; blue, 226 }  ,draw opacity=1 ] (40.22,121.81) .. controls (43.02,121.81) and (45.28,124.08) .. (45.28,126.87) -- (45.28,154.61) .. controls (45.28,157.4) and (43.02,159.67) .. (40.22,159.67) -- (25.06,159.67) .. controls (22.26,159.67) and (20,157.4) .. (20,154.61) -- (20,126.87) .. controls (20,124.08) and (22.26,121.81) .. (25.06,121.81) -- cycle ;
\draw [color={rgb, 255:red, 155; green, 155; blue, 155 }  ,draw opacity=1 ]   (35.55,140.74) -- (74.16,164.25) ;
\draw [shift={(75.87,165.29)}, rotate = 211.34] [color={rgb, 255:red, 155; green, 155; blue, 155 }  ,draw opacity=1 ][line width=0.75]    (4.37,-1.96) .. controls (2.78,-0.92) and (1.32,-0.27) .. (0,0) .. controls (1.32,0.27) and (2.78,0.92) .. (4.37,1.96)   ;
\draw  [color={rgb, 255:red, 74; green, 144; blue, 226 }  ,draw opacity=1 ] (86.11,150.15) .. controls (88.9,150.15) and (91.16,152.41) .. (91.16,155.21) -- (91.16,200.14) .. controls (91.16,202.94) and (88.9,205.2) .. (86.11,205.2) -- (70.94,205.2) .. controls (68.15,205.2) and (65.88,202.94) .. (65.88,200.14) -- (65.88,155.21) .. controls (65.88,152.41) and (68.15,150.15) .. (70.94,150.15) -- cycle ;
\draw [color={rgb, 255:red, 155; green, 155; blue, 155 }  ,draw opacity=1 ]   (35.55,140.74) -- (75.59,107.62) ;
\draw [shift={(77.13,106.35)}, rotate = 140.41] [color={rgb, 255:red, 155; green, 155; blue, 155 }  ,draw opacity=1 ][line width=0.75]    (4.37,-1.96) .. controls (2.78,-0.92) and (1.32,-0.27) .. (0,0) .. controls (1.32,0.27) and (2.78,0.92) .. (4.37,1.96)   ;
\draw  [color={rgb, 255:red, 74; green, 144; blue, 226 }  ,draw opacity=1 ] (86.87,71.2) .. controls (89.66,71.2) and (91.92,73.46) .. (91.92,76.26) -- (91.92,115.35) .. controls (91.92,118.15) and (89.66,120.41) .. (86.87,120.41) -- (71.7,120.41) .. controls (68.91,120.41) and (66.64,118.15) .. (66.64,115.35) -- (66.64,76.26) .. controls (66.64,73.46) and (68.91,71.2) .. (71.7,71.2) -- cycle ;
\draw   (29.73,140.74) .. controls (29.73,139.37) and (31.03,138.25) .. (32.64,138.25) .. controls (34.25,138.25) and (35.55,139.37) .. (35.55,140.74) .. controls (35.55,142.12) and (34.25,143.23) .. (32.64,143.23) .. controls (31.03,143.23) and (29.73,142.12) .. (29.73,140.74) -- cycle ;
\draw   (75.87,165.29) .. controls (75.87,163.92) and (77.17,162.8) .. (78.78,162.8) .. controls (80.38,162.8) and (81.68,163.92) .. (81.68,165.29) .. controls (81.68,166.67) and (80.38,167.78) .. (78.78,167.78) .. controls (77.17,167.78) and (75.87,166.67) .. (75.87,165.29) -- cycle ;
\draw   (77.13,106.35) .. controls (77.13,104.98) and (78.44,103.86) .. (80.04,103.86) .. controls (81.65,103.86) and (82.95,104.98) .. (82.95,106.35) .. controls (82.95,107.72) and (81.65,108.84) .. (80.04,108.84) .. controls (78.44,108.84) and (77.13,107.72) .. (77.13,106.35) -- cycle ;
\draw  [color={rgb, 255:red, 74; green, 144; blue, 226 }  ,draw opacity=1 ] (140.08,121.81) .. controls (142.88,121.81) and (145.14,124.08) .. (145.14,126.87) -- (145.14,154.61) .. controls (145.14,157.4) and (142.88,159.67) .. (140.08,159.67) -- (124.91,159.67) .. controls (122.12,159.67) and (119.86,157.4) .. (119.86,154.61) -- (119.86,126.87) .. controls (119.86,124.08) and (122.12,121.81) .. (124.91,121.81) -- cycle ;
\draw [color={rgb, 255:red, 155; green, 155; blue, 155 }  ,draw opacity=1 ]   (135.41,140.74) -- (174.02,164.25) ;
\draw [shift={(175.73,165.29)}, rotate = 211.34] [color={rgb, 255:red, 155; green, 155; blue, 155 }  ,draw opacity=1 ][line width=0.75]    (4.37,-1.96) .. controls (2.78,-0.92) and (1.32,-0.27) .. (0,0) .. controls (1.32,0.27) and (2.78,0.92) .. (4.37,1.96)   ;
\draw  [color={rgb, 255:red, 74; green, 144; blue, 226 }  ,draw opacity=1 ] (185.97,150.15) .. controls (188.76,150.15) and (191.02,152.41) .. (191.02,155.21) -- (191.02,200.14) .. controls (191.02,202.94) and (188.76,205.2) .. (185.97,205.2) -- (170.8,205.2) .. controls (168.01,205.2) and (165.74,202.94) .. (165.74,200.14) -- (165.74,155.21) .. controls (165.74,152.41) and (168.01,150.15) .. (170.8,150.15) -- cycle ;
\draw [color={rgb, 255:red, 155; green, 155; blue, 155 }  ,draw opacity=1 ]   (135.41,140.74) -- (175.45,107.62) ;
\draw [shift={(176.99,106.35)}, rotate = 140.41] [color={rgb, 255:red, 155; green, 155; blue, 155 }  ,draw opacity=1 ][line width=0.75]    (4.37,-1.96) .. controls (2.78,-0.92) and (1.32,-0.27) .. (0,0) .. controls (1.32,0.27) and (2.78,0.92) .. (4.37,1.96)   ;
\draw  [color={rgb, 255:red, 74; green, 144; blue, 226 }  ,draw opacity=1 ] (186.73,71.2) .. controls (189.52,71.2) and (191.78,73.46) .. (191.78,76.26) -- (191.78,115.35) .. controls (191.78,118.15) and (189.52,120.41) .. (186.73,120.41) -- (171.56,120.41) .. controls (168.76,120.41) and (166.5,118.15) .. (166.5,115.35) -- (166.5,76.26) .. controls (166.5,73.46) and (168.76,71.2) .. (171.56,71.2) -- cycle ;
\draw   (129.59,140.74) .. controls (129.59,139.37) and (130.89,138.25) .. (132.5,138.25) .. controls (134.1,138.25) and (135.41,139.37) .. (135.41,140.74) .. controls (135.41,142.12) and (134.1,143.23) .. (132.5,143.23) .. controls (130.89,143.23) and (129.59,142.12) .. (129.59,140.74) -- cycle ;
\draw   (175.73,165.29) .. controls (175.73,163.92) and (177.03,162.8) .. (178.64,162.8) .. controls (180.24,162.8) and (181.54,163.92) .. (181.54,165.29) .. controls (181.54,166.67) and (180.24,167.78) .. (178.64,167.78) .. controls (177.03,167.78) and (175.73,166.67) .. (175.73,165.29) -- cycle ;
\draw   (176.99,106.35) .. controls (176.99,104.98) and (178.29,103.86) .. (179.9,103.86) .. controls (181.51,103.86) and (182.81,104.98) .. (182.81,106.35) .. controls (182.81,107.72) and (181.51,108.84) .. (179.9,108.84) .. controls (178.29,108.84) and (176.99,107.72) .. (176.99,106.35) -- cycle ;
\draw  [color={rgb, 255:red, 208; green, 2; blue, 27 }  ,draw opacity=1 ] (173.33,95.8) .. controls (173.33,94.43) and (174.63,93.32) .. (176.23,93.32) .. controls (177.84,93.32) and (179.14,94.43) .. (179.14,95.8) .. controls (179.14,97.18) and (177.84,98.29) .. (176.23,98.29) .. controls (174.63,98.29) and (173.33,97.18) .. (173.33,95.8) -- cycle ;
\draw  [color={rgb, 255:red, 208; green, 2; blue, 27 }  ,draw opacity=1 ] (170.04,182.6) .. controls (170.04,181.22) and (171.34,180.11) .. (172.95,180.11) .. controls (174.55,180.11) and (175.85,181.22) .. (175.85,182.6) .. controls (175.85,183.97) and (174.55,185.08) .. (172.95,185.08) .. controls (171.34,185.08) and (170.04,183.97) .. (170.04,182.6) -- cycle ;
\draw  [color={rgb, 255:red, 208; green, 2; blue, 27 }  ,draw opacity=1 ] (180.78,155.56) .. controls (180.78,154.18) and (182.09,153.07) .. (183.69,153.07) .. controls (185.3,153.07) and (186.6,154.18) .. (186.6,155.56) .. controls (186.6,156.93) and (185.3,158.05) .. (183.69,158.05) .. controls (182.09,158.05) and (180.78,156.93) .. (180.78,155.56) -- cycle ;
\draw  [color={rgb, 255:red, 74; green, 144; blue, 226 }  ,draw opacity=1 ] (243.1,121.81) .. controls (245.89,121.81) and (248.16,124.08) .. (248.16,126.87) -- (248.16,154.61) .. controls (248.16,157.4) and (245.89,159.67) .. (243.1,159.67) -- (227.93,159.67) .. controls (225.14,159.67) and (222.88,157.4) .. (222.88,154.61) -- (222.88,126.87) .. controls (222.88,124.08) and (225.14,121.81) .. (227.93,121.81) -- cycle ;
\draw [color={rgb, 255:red, 155; green, 155; blue, 155 }  ,draw opacity=1 ]   (238.42,140.74) -- (277.04,164.25) ;
\draw [shift={(278.75,165.29)}, rotate = 211.34] [color={rgb, 255:red, 155; green, 155; blue, 155 }  ,draw opacity=1 ][line width=0.75]    (4.37,-1.96) .. controls (2.78,-0.92) and (1.32,-0.27) .. (0,0) .. controls (1.32,0.27) and (2.78,0.92) .. (4.37,1.96)   ;
\draw  [color={rgb, 255:red, 74; green, 144; blue, 226 }  ,draw opacity=1 ] (288.99,150.15) .. controls (291.78,150.15) and (294.04,152.41) .. (294.04,155.21) -- (294.04,200.14) .. controls (294.04,202.94) and (291.78,205.2) .. (288.99,205.2) -- (273.82,205.2) .. controls (271.02,205.2) and (268.76,202.94) .. (268.76,200.14) -- (268.76,155.21) .. controls (268.76,152.41) and (271.02,150.15) .. (273.82,150.15) -- cycle ;
\draw [color={rgb, 255:red, 155; green, 155; blue, 155 }  ,draw opacity=1 ]   (238.42,140.74) -- (278.47,107.62) ;
\draw [shift={(280.01,106.35)}, rotate = 140.41] [color={rgb, 255:red, 155; green, 155; blue, 155 }  ,draw opacity=1 ][line width=0.75]    (4.37,-1.96) .. controls (2.78,-0.92) and (1.32,-0.27) .. (0,0) .. controls (1.32,0.27) and (2.78,0.92) .. (4.37,1.96)   ;
\draw  [color={rgb, 255:red, 74; green, 144; blue, 226 }  ,draw opacity=1 ] (289.74,71.2) .. controls (292.54,71.2) and (294.8,73.46) .. (294.8,76.26) -- (294.8,115.35) .. controls (294.8,118.15) and (292.54,120.41) .. (289.74,120.41) -- (274.58,120.41) .. controls (271.78,120.41) and (269.52,118.15) .. (269.52,115.35) -- (269.52,76.26) .. controls (269.52,73.46) and (271.78,71.2) .. (274.58,71.2) -- cycle ;
\draw   (232.61,140.74) .. controls (232.61,139.37) and (233.91,138.25) .. (235.52,138.25) .. controls (237.12,138.25) and (238.42,139.37) .. (238.42,140.74) .. controls (238.42,142.12) and (237.12,143.23) .. (235.52,143.23) .. controls (233.91,143.23) and (232.61,142.12) .. (232.61,140.74) -- cycle ;
\draw   (278.75,165.29) .. controls (278.75,163.92) and (280.05,162.8) .. (281.65,162.8) .. controls (283.26,162.8) and (284.56,163.92) .. (284.56,165.29) .. controls (284.56,166.67) and (283.26,167.78) .. (281.65,167.78) .. controls (280.05,167.78) and (278.75,166.67) .. (278.75,165.29) -- cycle ;
\draw   (280.01,106.35) .. controls (280.01,104.98) and (281.31,103.86) .. (282.92,103.86) .. controls (284.52,103.86) and (285.83,104.98) .. (285.83,106.35) .. controls (285.83,107.72) and (284.52,108.84) .. (282.92,108.84) .. controls (281.31,108.84) and (280.01,107.72) .. (280.01,106.35) -- cycle ;
\draw   (276.35,95.8) .. controls (276.35,94.43) and (277.65,93.32) .. (279.25,93.32) .. controls (280.86,93.32) and (282.16,94.43) .. (282.16,95.8) .. controls (282.16,97.18) and (280.86,98.29) .. (279.25,98.29) .. controls (277.65,98.29) and (276.35,97.18) .. (276.35,95.8) -- cycle ;
\draw  [color={rgb, 255:red, 0; green, 0; blue, 0 }  ,draw opacity=1 ] (273.06,182.6) .. controls (273.06,181.22) and (274.36,180.11) .. (275.97,180.11) .. controls (277.57,180.11) and (278.87,181.22) .. (278.87,182.6) .. controls (278.87,183.97) and (277.57,185.08) .. (275.97,185.08) .. controls (274.36,185.08) and (273.06,183.97) .. (273.06,182.6) -- cycle ;
\draw   (283.8,155.56) .. controls (283.8,154.18) and (285.1,153.07) .. (286.71,153.07) .. controls (288.32,153.07) and (289.62,154.18) .. (289.62,155.56) .. controls (289.62,156.93) and (288.32,158.05) .. (286.71,158.05) .. controls (285.1,158.05) and (283.8,156.93) .. (283.8,155.56) -- cycle ;
\draw [color={rgb, 255:red, 208; green, 2; blue, 27 }  ,draw opacity=1 ]   (279.25,93.32) -- (287.55,82.06) ;
\draw [shift={(288.73,80.45)}, rotate = 126.38] [color={rgb, 255:red, 208; green, 2; blue, 27 }  ,draw opacity=1 ][line width=0.75]    (4.37,-1.32) .. controls (2.78,-0.56) and (1.32,-0.12) .. (0,0) .. controls (1.32,0.12) and (2.78,0.56) .. (4.37,1.32)   ;
\draw  [color={rgb, 255:red, 208; green, 2; blue, 27 }  ,draw opacity=1 ] (285.83,77.96) .. controls (285.83,76.59) and (287.13,75.47) .. (288.73,75.47) .. controls (290.34,75.47) and (291.64,76.59) .. (291.64,77.96) .. controls (291.64,79.33) and (290.34,80.45) .. (288.73,80.45) .. controls (287.13,80.45) and (285.83,79.33) .. (285.83,77.96) -- cycle ;
\draw [color={rgb, 255:red, 208; green, 2; blue, 27 }  ,draw opacity=1 ]   (281.65,167.78) -- (286.42,185.18) ;
\draw [shift={(286.95,187.11)}, rotate = 254.68] [color={rgb, 255:red, 208; green, 2; blue, 27 }  ,draw opacity=1 ][line width=0.75]    (4.37,-1.32) .. controls (2.78,-0.56) and (1.32,-0.12) .. (0,0) .. controls (1.32,0.12) and (2.78,0.56) .. (4.37,1.32)   ;
\draw  [color={rgb, 255:red, 74; green, 144; blue, 226 }  ,draw opacity=1 ] (346.1,121.81) .. controls (348.89,121.81) and (351.16,124.08) .. (351.16,126.87) -- (351.16,154.61) .. controls (351.16,157.4) and (348.89,159.67) .. (346.1,159.67) -- (330.93,159.67) .. controls (328.14,159.67) and (325.88,157.4) .. (325.88,154.61) -- (325.88,126.87) .. controls (325.88,124.08) and (328.14,121.81) .. (330.93,121.81) -- cycle ;
\draw [color={rgb, 255:red, 155; green, 155; blue, 155 }  ,draw opacity=1 ]   (341.42,140.74) -- (380.04,164.25) ;
\draw [shift={(381.75,165.29)}, rotate = 211.34] [color={rgb, 255:red, 155; green, 155; blue, 155 }  ,draw opacity=1 ][line width=0.75]    (4.37,-1.96) .. controls (2.78,-0.92) and (1.32,-0.27) .. (0,0) .. controls (1.32,0.27) and (2.78,0.92) .. (4.37,1.96)   ;
\draw  [color={rgb, 255:red, 74; green, 144; blue, 226 }  ,draw opacity=1 ] (391.99,150.15) .. controls (394.78,150.15) and (397.04,152.41) .. (397.04,155.21) -- (397.04,200.14) .. controls (397.04,202.94) and (394.78,205.2) .. (391.99,205.2) -- (376.82,205.2) .. controls (374.02,205.2) and (371.76,202.94) .. (371.76,200.14) -- (371.76,155.21) .. controls (371.76,152.41) and (374.02,150.15) .. (376.82,150.15) -- cycle ;
\draw [color={rgb, 255:red, 155; green, 155; blue, 155 }  ,draw opacity=1 ]   (341.42,140.74) -- (381.47,107.62) ;
\draw [shift={(383.01,106.35)}, rotate = 140.41] [color={rgb, 255:red, 155; green, 155; blue, 155 }  ,draw opacity=1 ][line width=0.75]    (4.37,-1.96) .. controls (2.78,-0.92) and (1.32,-0.27) .. (0,0) .. controls (1.32,0.27) and (2.78,0.92) .. (4.37,1.96)   ;
\draw  [color={rgb, 255:red, 74; green, 144; blue, 226 }  ,draw opacity=1 ] (392.74,71.2) .. controls (395.54,71.2) and (397.8,73.46) .. (397.8,76.26) -- (397.8,115.35) .. controls (397.8,118.15) and (395.54,120.41) .. (392.74,120.41) -- (377.58,120.41) .. controls (374.78,120.41) and (372.52,118.15) .. (372.52,115.35) -- (372.52,76.26) .. controls (372.52,73.46) and (374.78,71.2) .. (377.58,71.2) -- cycle ;
\draw   (335.61,140.74) .. controls (335.61,139.37) and (336.91,138.25) .. (338.52,138.25) .. controls (340.12,138.25) and (341.42,139.37) .. (341.42,140.74) .. controls (341.42,142.12) and (340.12,143.23) .. (338.52,143.23) .. controls (336.91,143.23) and (335.61,142.12) .. (335.61,140.74) -- cycle ;
\draw   (383.01,106.35) .. controls (383.01,104.98) and (384.31,103.86) .. (385.92,103.86) .. controls (387.52,103.86) and (388.83,104.98) .. (388.83,106.35) .. controls (388.83,107.72) and (387.52,108.84) .. (385.92,108.84) .. controls (384.31,108.84) and (383.01,107.72) .. (383.01,106.35) -- cycle ;
\draw   (379.35,95.8) .. controls (379.35,94.43) and (380.65,93.32) .. (382.25,93.32) .. controls (383.86,93.32) and (385.16,94.43) .. (385.16,95.8) .. controls (385.16,97.18) and (383.86,98.29) .. (382.25,98.29) .. controls (380.65,98.29) and (379.35,97.18) .. (379.35,95.8) -- cycle ;
\draw   (376.06,182.6) .. controls (376.06,181.22) and (377.36,180.11) .. (378.97,180.11) .. controls (380.57,180.11) and (381.87,181.22) .. (381.87,182.6) .. controls (381.87,183.97) and (380.57,185.08) .. (378.97,185.08) .. controls (377.36,185.08) and (376.06,183.97) .. (376.06,182.6) -- cycle ;
\draw   (386.8,155.56) .. controls (386.8,154.18) and (388.1,153.07) .. (389.71,153.07) .. controls (391.32,153.07) and (392.62,154.18) .. (392.62,155.56) .. controls (392.62,156.93) and (391.32,158.05) .. (389.71,158.05) .. controls (388.1,158.05) and (386.8,156.93) .. (386.8,155.56) -- cycle ;
\draw [color={rgb, 255:red, 155; green, 155; blue, 155 }  ,draw opacity=1 ]   (382.25,93.32) -- (390.55,82.06) ;
\draw [shift={(391.73,80.45)}, rotate = 126.38] [color={rgb, 255:red, 155; green, 155; blue, 155 }  ,draw opacity=1 ][line width=0.75]    (4.37,-1.32) .. controls (2.78,-0.56) and (1.32,-0.12) .. (0,0) .. controls (1.32,0.12) and (2.78,0.56) .. (4.37,1.32)   ;
\draw   (388.83,77.96) .. controls (388.83,76.59) and (390.13,75.47) .. (391.73,75.47) .. controls (393.34,75.47) and (394.64,76.59) .. (394.64,77.96) .. controls (394.64,79.33) and (393.34,80.45) .. (391.73,80.45) .. controls (390.13,80.45) and (388.83,79.33) .. (388.83,77.96) -- cycle ;
\draw  [color={rgb, 255:red, 74; green, 144; blue, 226 }  ,draw opacity=1 ] (437.87,70.2) .. controls (440.66,70.2) and (442.92,72.46) .. (442.92,75.26) -- (442.92,114.35) .. controls (442.92,117.15) and (440.66,119.41) .. (437.87,119.41) -- (422.7,119.41) .. controls (419.91,119.41) and (417.64,117.15) .. (417.64,114.35) -- (417.64,75.26) .. controls (417.64,72.46) and (419.91,70.2) .. (422.7,70.2) -- cycle ;
\draw  [color={rgb, 255:red, 208; green, 2; blue, 27 }  ,draw opacity=1 ] (428.13,105.35) .. controls (428.13,103.98) and (429.44,102.86) .. (431.04,102.86) .. controls (432.65,102.86) and (433.95,103.98) .. (433.95,105.35) .. controls (433.95,106.72) and (432.65,107.84) .. (431.04,107.84) .. controls (429.44,107.84) and (428.13,106.72) .. (428.13,105.35) -- cycle ;
\draw [color={rgb, 255:red, 208; green, 2; blue, 27 }  ,draw opacity=1 ]   (385.16,95.8) -- (426.18,104.92) ;
\draw [shift={(428.13,105.35)}, rotate = 192.52] [color={rgb, 255:red, 208; green, 2; blue, 27 }  ,draw opacity=1 ][line width=0.75]    (4.37,-1.32) .. controls (2.78,-0.56) and (1.32,-0.12) .. (0,0) .. controls (1.32,0.12) and (2.78,0.56) .. (4.37,1.32)   ;
\draw [color={rgb, 255:red, 155; green, 155; blue, 155 }  ,draw opacity=1 ] [dash pattern={on 0.75pt off 0.75pt}]  (74,62.8) .. controls (75.67,61.13) and (77.33,61.13) .. (79,62.8) .. controls (80.67,64.47) and (82.33,64.47) .. (84,62.8) .. controls (85.67,61.13) and (87.33,61.13) .. (89,62.8) .. controls (90.67,64.47) and (92.33,64.47) .. (94,62.8) .. controls (95.67,61.13) and (97.33,61.13) .. (99,62.8) .. controls (100.67,64.47) and (102.33,64.47) .. (104,62.8) .. controls (105.67,61.13) and (107.33,61.13) .. (109,62.8) .. controls (110.67,64.47) and (112.33,64.47) .. (114,62.8) .. controls (115.67,61.13) and (117.33,61.13) .. (119,62.8) .. controls (120.67,64.47) and (122.33,64.47) .. (124,62.8) .. controls (125.67,61.13) and (127.33,61.13) .. (129,62.8) .. controls (130.67,64.47) and (132.33,64.47) .. (134,62.8) .. controls (135.67,61.13) and (137.33,61.13) .. (139,62.8) .. controls (140.67,64.47) and (142.33,64.47) .. (144,62.8) .. controls (145.67,61.13) and (147.33,61.13) .. (149,62.8) .. controls (150.67,64.47) and (152.33,64.47) .. (154,62.8) .. controls (155.67,61.13) and (157.33,61.13) .. (159,62.8) .. controls (160.67,64.47) and (162.33,64.47) .. (164,62.8) .. controls (165.67,61.13) and (167.33,61.13) .. (169,62.8) -- (177,62.8) ;
\draw [shift={(179,62.8)}, rotate = 180] [color={rgb, 255:red, 155; green, 155; blue, 155 }  ,draw opacity=1 ][line width=0.75]    (10.93,-3.29) .. controls (6.95,-1.4) and (3.31,-0.3) .. (0,0) .. controls (3.31,0.3) and (6.95,1.4) .. (10.93,3.29)   ;
\draw [color={rgb, 255:red, 155; green, 155; blue, 155 }  ,draw opacity=1 ] [dash pattern={on 0.75pt off 0.75pt}]  (179,62.8) .. controls (180.65,61.12) and (182.32,61.11) .. (184,62.76) .. controls (185.68,64.41) and (187.35,64.4) .. (189,62.72) .. controls (190.65,61.04) and (192.32,61.03) .. (194,62.68) .. controls (195.67,64.33) and (197.34,64.32) .. (199,62.65) .. controls (200.65,60.97) and (202.32,60.96) .. (204,62.61) .. controls (205.68,64.26) and (207.35,64.25) .. (209,62.57) .. controls (210.65,60.89) and (212.32,60.88) .. (214,62.53) .. controls (215.68,64.18) and (217.35,64.17) .. (219,62.49) .. controls (220.65,60.81) and (222.32,60.8) .. (224,62.45) .. controls (225.67,64.1) and (227.34,64.09) .. (229,62.42) .. controls (230.65,60.74) and (232.32,60.73) .. (234,62.38) .. controls (235.68,64.03) and (237.35,64.02) .. (239,62.34) .. controls (240.65,60.66) and (242.32,60.65) .. (244,62.3) .. controls (245.68,63.95) and (247.35,63.94) .. (249,62.26) .. controls (250.65,60.58) and (252.32,60.57) .. (254,62.22) .. controls (255.68,63.87) and (257.35,63.86) .. (259,62.18) .. controls (260.66,60.51) and (262.33,60.5) .. (264,62.15) .. controls (265.68,63.8) and (267.35,63.79) .. (269,62.11) -- (273,62.08) -- (281,62.02) ;
\draw [shift={(283,62)}, rotate = 179.56] [color={rgb, 255:red, 155; green, 155; blue, 155 }  ,draw opacity=1 ][line width=0.75]    (10.93,-3.29) .. controls (6.95,-1.4) and (3.31,-0.3) .. (0,0) .. controls (3.31,0.3) and (6.95,1.4) .. (10.93,3.29)   ;
\draw [color={rgb, 255:red, 155; green, 155; blue, 155 }  ,draw opacity=1 ] [dash pattern={on 0.75pt off 0.75pt}]  (283,62) .. controls (284.67,60.35) and (286.34,60.36) .. (288,62.03) .. controls (289.66,63.7) and (291.33,63.71) .. (293,62.06) .. controls (294.67,60.4) and (296.34,60.41) .. (298,62.08) .. controls (299.66,63.75) and (301.33,63.76) .. (303,62.11) .. controls (304.67,60.46) and (306.34,60.47) .. (308,62.14) .. controls (309.66,63.81) and (311.33,63.82) .. (313,62.17) .. controls (314.67,60.52) and (316.34,60.53) .. (318,62.2) .. controls (319.66,63.87) and (321.33,63.88) .. (323,62.23) .. controls (324.67,60.57) and (326.34,60.58) .. (328,62.25) .. controls (329.66,63.92) and (331.33,63.93) .. (333,62.28) .. controls (334.67,60.63) and (336.34,60.64) .. (338,62.31) .. controls (339.66,63.98) and (341.33,63.99) .. (343,62.34) .. controls (344.67,60.69) and (346.34,60.7) .. (348,62.37) .. controls (349.66,64.04) and (351.33,64.05) .. (353,62.4) .. controls (354.67,60.74) and (356.34,60.75) .. (358,62.42) .. controls (359.66,64.09) and (361.33,64.1) .. (363,62.45) .. controls (364.67,60.8) and (366.34,60.81) .. (368,62.48) .. controls (369.66,64.15) and (371.33,64.16) .. (373,62.51) .. controls (374.67,60.86) and (376.34,60.87) .. (378,62.54) -- (379,62.54) -- (387,62.59) ;
\draw [shift={(389,62.6)}, rotate = 180.32] [color={rgb, 255:red, 155; green, 155; blue, 155 }  ,draw opacity=1 ][line width=0.75]    (10.93,-3.29) .. controls (6.95,-1.4) and (3.31,-0.3) .. (0,0) .. controls (3.31,0.3) and (6.95,1.4) .. (10.93,3.29)   ;
\draw  [color={rgb, 255:red, 155; green, 155; blue, 155 }  ,draw opacity=1 ] (178,212) .. controls (177.99,216.67) and (180.32,219) .. (184.99,219.01) -- (270.49,219.18) .. controls (277.16,219.19) and (280.48,221.53) .. (280.47,226.2) .. controls (280.48,221.53) and (283.82,219.21) .. (290.49,219.22)(287.49,219.21) -- (375.99,219.38) .. controls (380.66,219.39) and (382.99,217.06) .. (383,212.39) ;
\draw  [color={rgb, 255:red, 208; green, 2; blue, 27 }  ,draw opacity=1 ] (284.04,189.6) .. controls (284.04,188.22) and (285.34,187.11) .. (286.95,187.11) .. controls (288.55,187.11) and (289.85,188.22) .. (289.85,189.6) .. controls (289.85,190.97) and (288.55,192.08) .. (286.95,192.08) .. controls (285.34,192.08) and (284.04,190.97) .. (284.04,189.6) -- cycle ;
\draw   (381.75,165.29) .. controls (381.75,163.92) and (383.05,162.8) .. (384.65,162.8) .. controls (386.26,162.8) and (387.56,163.92) .. (387.56,165.29) .. controls (387.56,166.67) and (386.26,167.78) .. (384.65,167.78) .. controls (383.05,167.78) and (381.75,166.67) .. (381.75,165.29) -- cycle ;
\draw [color={rgb, 255:red, 155; green, 155; blue, 155 }  ,draw opacity=1 ]   (384.65,167.78) -- (389.42,185.18) ;
\draw [shift={(389.95,187.11)}, rotate = 254.68] [color={rgb, 255:red, 155; green, 155; blue, 155 }  ,draw opacity=1 ][line width=0.75]    (4.37,-1.32) .. controls (2.78,-0.56) and (1.32,-0.12) .. (0,0) .. controls (1.32,0.12) and (2.78,0.56) .. (4.37,1.32)   ;
\draw  [color={rgb, 255:red, 0; green, 0; blue, 0 }  ,draw opacity=1 ] (387.04,189.6) .. controls (387.04,188.22) and (388.34,187.11) .. (389.95,187.11) .. controls (391.55,187.11) and (392.85,188.22) .. (392.85,189.6) .. controls (392.85,190.97) and (391.55,192.08) .. (389.95,192.08) .. controls (388.34,192.08) and (387.04,190.97) .. (387.04,189.6) -- cycle ;
\draw  [color={rgb, 255:red, 74; green, 144; blue, 226 }  ,draw opacity=1 ] (491.1,122.81) .. controls (493.89,122.81) and (496.16,125.08) .. (496.16,127.87) -- (496.16,155.61) .. controls (496.16,158.4) and (493.89,160.67) .. (491.1,160.67) -- (475.93,160.67) .. controls (473.14,160.67) and (470.88,158.4) .. (470.88,155.61) -- (470.88,127.87) .. controls (470.88,125.08) and (473.14,122.81) .. (475.93,122.81) -- cycle ;
\draw [color={rgb, 255:red, 155; green, 155; blue, 155 }  ,draw opacity=1 ]   (486.42,141.74) -- (525.04,165.25) ;
\draw [shift={(526.75,166.29)}, rotate = 211.34] [color={rgb, 255:red, 155; green, 155; blue, 155 }  ,draw opacity=1 ][line width=0.75]    (4.37,-1.96) .. controls (2.78,-0.92) and (1.32,-0.27) .. (0,0) .. controls (1.32,0.27) and (2.78,0.92) .. (4.37,1.96)   ;
\draw  [color={rgb, 255:red, 74; green, 144; blue, 226 }  ,draw opacity=1 ] (536.99,151.15) .. controls (539.78,151.15) and (542.04,153.41) .. (542.04,156.21) -- (542.04,201.14) .. controls (542.04,203.94) and (539.78,206.2) .. (536.99,206.2) -- (521.82,206.2) .. controls (519.02,206.2) and (516.76,203.94) .. (516.76,201.14) -- (516.76,156.21) .. controls (516.76,153.41) and (519.02,151.15) .. (521.82,151.15) -- cycle ;
\draw [color={rgb, 255:red, 155; green, 155; blue, 155 }  ,draw opacity=1 ]   (486.42,141.74) -- (526.47,108.62) ;
\draw [shift={(528.01,107.35)}, rotate = 140.41] [color={rgb, 255:red, 155; green, 155; blue, 155 }  ,draw opacity=1 ][line width=0.75]    (4.37,-1.96) .. controls (2.78,-0.92) and (1.32,-0.27) .. (0,0) .. controls (1.32,0.27) and (2.78,0.92) .. (4.37,1.96)   ;
\draw  [color={rgb, 255:red, 74; green, 144; blue, 226 }  ,draw opacity=1 ] (537.74,72.2) .. controls (540.54,72.2) and (542.8,74.46) .. (542.8,77.26) -- (542.8,116.35) .. controls (542.8,119.15) and (540.54,121.41) .. (537.74,121.41) -- (522.58,121.41) .. controls (519.78,121.41) and (517.52,119.15) .. (517.52,116.35) -- (517.52,77.26) .. controls (517.52,74.46) and (519.78,72.2) .. (522.58,72.2) -- cycle ;
\draw   (480.61,141.74) .. controls (480.61,140.37) and (481.91,139.25) .. (483.52,139.25) .. controls (485.12,139.25) and (486.42,140.37) .. (486.42,141.74) .. controls (486.42,143.12) and (485.12,144.23) .. (483.52,144.23) .. controls (481.91,144.23) and (480.61,143.12) .. (480.61,141.74) -- cycle ;
\draw   (526.75,166.29) .. controls (526.75,164.92) and (528.05,163.8) .. (529.65,163.8) .. controls (531.26,163.8) and (532.56,164.92) .. (532.56,166.29) .. controls (532.56,167.67) and (531.26,168.78) .. (529.65,168.78) .. controls (528.05,168.78) and (526.75,167.67) .. (526.75,166.29) -- cycle ;
\draw   (528.01,107.35) .. controls (528.01,105.98) and (529.31,104.86) .. (530.92,104.86) .. controls (532.52,104.86) and (533.83,105.98) .. (533.83,107.35) .. controls (533.83,108.72) and (532.52,109.84) .. (530.92,109.84) .. controls (529.31,109.84) and (528.01,108.72) .. (528.01,107.35) -- cycle ;
\draw   (524.35,96.8) .. controls (524.35,95.43) and (525.65,94.32) .. (527.25,94.32) .. controls (528.86,94.32) and (530.16,95.43) .. (530.16,96.8) .. controls (530.16,98.18) and (528.86,99.29) .. (527.25,99.29) .. controls (525.65,99.29) and (524.35,98.18) .. (524.35,96.8) -- cycle ;
\draw   (521.06,183.6) .. controls (521.06,182.22) and (522.36,181.11) .. (523.97,181.11) .. controls (525.57,181.11) and (526.87,182.22) .. (526.87,183.6) .. controls (526.87,184.97) and (525.57,186.08) .. (523.97,186.08) .. controls (522.36,186.08) and (521.06,184.97) .. (521.06,183.6) -- cycle ;
\draw   (531.8,156.56) .. controls (531.8,155.18) and (533.1,154.07) .. (534.71,154.07) .. controls (536.32,154.07) and (537.62,155.18) .. (537.62,156.56) .. controls (537.62,157.93) and (536.32,159.05) .. (534.71,159.05) .. controls (533.1,159.05) and (531.8,157.93) .. (531.8,156.56) -- cycle ;
\draw [color={rgb, 255:red, 155; green, 155; blue, 155 }  ,draw opacity=1 ]   (527.25,94.32) -- (535.55,83.06) ;
\draw [shift={(536.73,81.45)}, rotate = 126.38] [color={rgb, 255:red, 155; green, 155; blue, 155 }  ,draw opacity=1 ][line width=0.75]    (4.37,-1.32) .. controls (2.78,-0.56) and (1.32,-0.12) .. (0,0) .. controls (1.32,0.12) and (2.78,0.56) .. (4.37,1.32)   ;
\draw   (533.83,78.96) .. controls (533.83,77.59) and (535.13,76.47) .. (536.73,76.47) .. controls (538.34,76.47) and (539.64,77.59) .. (539.64,78.96) .. controls (539.64,80.33) and (538.34,81.45) .. (536.73,81.45) .. controls (535.13,81.45) and (533.83,80.33) .. (533.83,78.96) -- cycle ;
\draw  [color={rgb, 255:red, 74; green, 144; blue, 226 }  ,draw opacity=1 ] (582.87,71.2) .. controls (585.66,71.2) and (587.92,73.46) .. (587.92,76.26) -- (587.92,115.35) .. controls (587.92,118.15) and (585.66,120.41) .. (582.87,120.41) -- (567.7,120.41) .. controls (564.91,120.41) and (562.64,118.15) .. (562.64,115.35) -- (562.64,76.26) .. controls (562.64,73.46) and (564.91,71.2) .. (567.7,71.2) -- cycle ;
\draw   (573.13,106.35) .. controls (573.13,104.98) and (574.44,103.86) .. (576.04,103.86) .. controls (577.65,103.86) and (578.95,104.98) .. (578.95,106.35) .. controls (578.95,107.72) and (577.65,108.84) .. (576.04,108.84) .. controls (574.44,108.84) and (573.13,107.72) .. (573.13,106.35) -- cycle ;
\draw [color={rgb, 255:red, 155; green, 155; blue, 155 }  ,draw opacity=1 ]   (530.16,96.8) -- (571.18,105.92) ;
\draw [shift={(573.13,106.35)}, rotate = 192.52] [color={rgb, 255:red, 155; green, 155; blue, 155 }  ,draw opacity=1 ][line width=0.75]    (4.37,-1.32) .. controls (2.78,-0.56) and (1.32,-0.12) .. (0,0) .. controls (1.32,0.12) and (2.78,0.56) .. (4.37,1.32)   ;
\draw  [color={rgb, 255:red, 208; green, 2; blue, 27 }  ,draw opacity=1 ] (573.13,87.35) .. controls (573.13,85.98) and (574.44,84.86) .. (576.04,84.86) .. controls (577.65,84.86) and (578.95,85.98) .. (578.95,87.35) .. controls (578.95,88.72) and (577.65,89.84) .. (576.04,89.84) .. controls (574.44,89.84) and (573.13,88.72) .. (573.13,87.35) -- cycle ;
\draw [color={rgb, 255:red, 208; green, 2; blue, 27 }  ,draw opacity=1 ]   (539.64,78.96) -- (571.19,86.86) ;
\draw [shift={(573.13,87.35)}, rotate = 194.06] [color={rgb, 255:red, 208; green, 2; blue, 27 }  ,draw opacity=1 ][line width=0.75]    (4.37,-1.32) .. controls (2.78,-0.56) and (1.32,-0.12) .. (0,0) .. controls (1.32,0.12) and (2.78,0.56) .. (4.37,1.32)   ;
\draw [color={rgb, 255:red, 155; green, 155; blue, 155 }  ,draw opacity=1 ]   (533.83,107.35) -- (574.04,108.77) ;
\draw [shift={(576.04,108.84)}, rotate = 182.02] [color={rgb, 255:red, 155; green, 155; blue, 155 }  ,draw opacity=1 ][line width=0.75]    (4.37,-1.32) .. controls (2.78,-0.56) and (1.32,-0.12) .. (0,0) .. controls (1.32,0.12) and (2.78,0.56) .. (4.37,1.32)   ;
\draw [color={rgb, 255:red, 155; green, 155; blue, 155 }  ,draw opacity=1 ]   (529.65,168.78) -- (534.42,186.18) ;
\draw [shift={(534.95,188.11)}, rotate = 254.68] [color={rgb, 255:red, 155; green, 155; blue, 155 }  ,draw opacity=1 ][line width=0.75]    (4.37,-1.32) .. controls (2.78,-0.56) and (1.32,-0.12) .. (0,0) .. controls (1.32,0.12) and (2.78,0.56) .. (4.37,1.32)   ;
\draw  [color={rgb, 255:red, 0; green, 0; blue, 0 }  ,draw opacity=1 ] (532.04,190.6) .. controls (532.04,189.22) and (533.34,188.11) .. (534.95,188.11) .. controls (536.55,188.11) and (537.85,189.22) .. (537.85,190.6) .. controls (537.85,191.97) and (536.55,193.08) .. (534.95,193.08) .. controls (533.34,193.08) and (532.04,191.97) .. (532.04,190.6) -- cycle ;
\draw [color={rgb, 255:red, 155; green, 155; blue, 155 }  ,draw opacity=1 ] [dash pattern={on 0.75pt off 0.75pt}]  (389,62.6) .. controls (390.68,60.95) and (392.35,60.96) .. (394,62.64) .. controls (395.65,64.32) and (397.32,64.34) .. (399,62.69) .. controls (400.68,61.04) and (402.35,61.05) .. (404,62.73) .. controls (405.65,64.41) and (407.32,64.42) .. (409,62.77) .. controls (410.68,61.12) and (412.35,61.13) .. (414,62.81) .. controls (415.65,64.49) and (417.32,64.51) .. (419,62.86) .. controls (420.68,61.21) and (422.35,61.22) .. (424,62.9) .. controls (425.65,64.58) and (427.32,64.59) .. (429,62.94) .. controls (430.68,61.29) and (432.35,61.3) .. (434,62.98) .. controls (435.65,64.66) and (437.32,64.68) .. (439,63.03) .. controls (440.68,61.38) and (442.35,61.39) .. (444,63.07) .. controls (445.65,64.75) and (447.32,64.76) .. (449,63.11) .. controls (450.68,61.46) and (452.35,61.47) .. (454,63.15) .. controls (455.65,64.83) and (457.32,64.85) .. (459,63.2) .. controls (460.68,61.55) and (462.35,61.56) .. (464,63.24) .. controls (465.65,64.92) and (467.32,64.93) .. (469,63.28) .. controls (470.68,61.63) and (472.35,61.64) .. (474,63.32) .. controls (475.65,65) and (477.32,65.02) .. (479,63.37) .. controls (480.68,61.72) and (482.35,61.73) .. (484,63.41) .. controls (485.65,65.09) and (487.32,65.1) .. (489,63.45) .. controls (490.68,61.8) and (492.35,61.81) .. (494,63.49) .. controls (495.65,65.17) and (497.32,65.19) .. (499,63.54) .. controls (500.68,61.89) and (502.35,61.9) .. (504,63.58) .. controls (505.65,65.26) and (507.32,65.27) .. (509,63.62) .. controls (510.68,61.97) and (512.35,61.98) .. (514,63.66) .. controls (515.65,65.34) and (517.32,65.36) .. (519,63.71) -- (520,63.71) -- (528,63.78) ;
\draw [shift={(530,63.8)}, rotate = 180.49] [color={rgb, 255:red, 155; green, 155; blue, 155 }  ,draw opacity=1 ][line width=0.75]    (10.93,-3.29) .. controls (6.95,-1.4) and (3.31,-0.3) .. (0,0) .. controls (3.31,0.3) and (6.95,1.4) .. (10.93,3.29)   ;

\draw (108,45.4) node [anchor=north west][inner sep=0.75pt] [font=\footnotesize,color={rgb, 255:red, 74; green, 74; blue, 74 }  ,opacity=1 ]    {\textbf{\fontfamily{ppl}\selectfont $\exists$ step}};
\draw (178,45.4) node [anchor=north west][inner sep=0.75pt]  [font=\footnotesize,color={rgb, 255:red, 74; green, 74; blue, 74 }  ,opacity=1 ] [align=left] {\textbf{\fontfamily{ppl}\selectfont Horizontal step}};
\draw (294,45.4) node [anchor=north west][inner sep=0.75pt] [font=\footnotesize,color={rgb, 255:red, 74; green, 74; blue, 74 }  ,opacity=1 ]  [align=left] {\textbf{\fontfamily{ppl}\selectfont Vertical step}};
\draw (256,227.4) node [anchor=north west][inner sep=0.75pt]  [font=\footnotesize,color={rgb, 255:red, 74; green, 74; blue, 74 }  ,opacity=1 ]  {\textbf{\fontfamily{ppl}\selectfont $\Diamond$ step}};
\draw (394,45.4) node [anchor=north west][inner sep=0.75pt] [font=\footnotesize,color={rgb, 255:red, 74; green, 74; blue, 74 }  ,opacity=1 ]  [align=left] {\textbf{\fontfamily{ppl}\selectfont Commutativity step}};
\end{tikzpicture}
\end{center}

\begin{lem}\label{rmk}
Suppose $\widehat{u},\widehat{w}\in X_k$.
\begin{enumerate}
    \item $\widehat{u}\mathrel{E_k}\widehat{w}$ iff $u\mathrel{E}w$. \label{rmk1}

    \item $\widehat{u}\mathrel{R_k}\widehat{w}$ implies $u\mathrel{Q}w$. \label{rmk3}
    
    \item $\widehat{u}\mathrel{R_k}\widehat{w}$ implies ($u\mathrel{R}w$ and $u\mathrel{E}w$) or ($u\mathrel{Q}w$ and $u \mathrel{\cancel{E}} w$). \label{rmk2}
    
    \item $\widehat{u}\mathrel{Q_k}\widehat{w}$ implies $u\mathrel{Q}w$. \label{rmk4}
    
    \item $u=w$ iff $\widehat{u}=\widehat{w}$ (no duplicates). \label{rmk5}

\end{enumerate}
\end{lem}

\begin{proof}
(\labelcref{rmk1}) and (\labelcref{rmk3}) 
follow from \Cref{lemma_lc}.

(\labelcref{rmk2})  
Let $\widehat{u}\mathrel{R_k}\widehat{w}$. If $\widehat{u}\mathrel{\cancel{E_k}}\widehat{w}$, applying (\labelcref{rmk1}) and (\labelcref{rmk3}) yields $u\mathrel{Q}w$ and $u\mathrel{\cancel{E}}w$. 
Suppose that $\widehat{u}\mathrel{E_k}\widehat{w}$. Then $u\mathrel{E}w$ by (\labelcref{rmk1}). Because we take the reflexive-transitive closure at each step, we have $\widehat{u}\mathrel{R_k}\widehat{x}_1\mathrel{R_k}\dots\widehat{x}_{n-1}\mathrel{R_k}\widehat{w}$, where each $R_k$-arrow is introduced in one of the steps of the construction. By (\labelcref{rmk3}), $u\mathrel{Q}x_1\mathrel{Q}\dots x_{n-1}\mathrel{Q}w$. We 
show that $u\mathrel{E}x_i$ for each $1\leq i\leq n-1$. First, 
note that $u\mathrel{Q}x_i\mathrel{Q}w\mathrel{Q}u$, where $w\mathrel{Q}u$ follows from $w\mathrel{E}u$. Consequently, $u\mathrel{Q}x_i\mathrel{Q}u$, and since $u$ is strongly maximal, we have $u\mathrel{E}x_i$ by \crefdefpart{E_Q}{E_Q1}. Thus, $\widehat{u}\mathrel{R_k}\widehat{x}_1\mathrel{R_k}\dots\widehat{x}_{n-1}\mathrel{R_k}\widehat{w}$ and  $\widehat{u}\mathrel{E_k}\widehat{x}_1\mathrel{E_k}\dots\widehat{x}_{n-1}\mathrel{E_k}\widehat{w}$, where the latter follows from (\labelcref{rmk1}). This means that each $R_k$-relation is introduced in the horizontal step, and by construction, such relations correspond to $R$-relations in the original frame. Thus, $u\mathrel{R}x_1\mathrel{R}\dots x_{n-1}\mathrel{R}w$, and hence $u\mathrel{R}w$.

(\labelcref{rmk4}) This follows from (\labelcref{rmk1}) and (\labelcref{rmk3}).

(\labelcref{rmk5}) Suppose $\widehat{u}\neq\widehat{w}$. If $u=w$, then $\widehat{u}$ was added to $X_k$ twice, which contradicts the fact that we introduce new points only if they have not already been added. Hence, $u=w\implies\widehat{u}=\widehat{w}$. Suppose $u\neq w$. If $\widehat{u}=\widehat{w}$, then distinct points in $X$ have the same copies in $X_k$, which is impossible. Hence, $\widehat{u}=\widehat{w}\implies u=w$.
\end{proof}

For the next lemma, we recall that a subset $C$ of a partially ordered set $(P,\leq)$ is a \emph{chain} if $x\le y$ or $y\le x$ for each $x,y\in C$. 
The \emph{length} of a chain $C$ is the cardinality of $C$. 
The \emph{depth} of $P$ is $n<\omega$ if there is a chain in $P$ of length $n$ and no other chain in $P$ has higher length. We say that $P$ is of \emph{finite depth} if $P$ is of depth $n$ for some $n<\omega$.

\begin{lem}{\label{partial}}
    $\mathfrak{F}_k$ is a finite partially ordered {\textbf{MS4}}-frame for each $k<\omega$.
\end{lem}

\begin{proof}
First, it is immediate from the construction that $R_k$ is a preorder and $E_k$ is an equivalence relation. Commutativity holds as well since it is built into the construction of $\mathfrak{F}_k$. Hence, $\mathfrak{F}_{k}$ is an \textbf{MS4}-frame. 

We show that $R_k$ is antisymmetric. Suppose $\widehat{u} \mathrel{R_k} \widehat{w}$ in $\mathfrak{F}_k$ with $\widehat{u}\neq\widehat{w}$. By \crefdefpart{rmk}{rmk5}, $u\neq w$. By \crefdefpart{rmk}{rmk2}, ($u\mathrel{R}w$ and $u\mathrel{E}w$) or ($u\mathrel{Q}w$ and $u\mathrel{\cancel{E}}w$). First suppose the former condition holds. If $\widehat{w}\mathrel{R_k}\widehat{u}$, then since $w\mathrel{E}u$, we must have $w\mathrel{R}u$ by \crefdefpart{rmk}{rmk2}. However, $u$ and $w$ are both (strongly) maximal points, which means that they cannot form an $R$-cluster (see \Cref{enter-exit}), a contradiction. Now suppose the latter condition holds. If $\widehat{w}\mathrel{R_k}\widehat{u}$, then since $w\mathrel{\cancel{E}}u$, we must have $w\mathrel{Q}u$ by \crefdefpart{rmk}{rmk2}. However, $u$ and $w$ are strongly maximal points, which means that they cannot form a $Q$-cluster that is not already an $E$-cluster (see \crefdefpart{E_Q}{E_Q1}), a contradiction. Thus, we cannot have $\widehat{w}\mathrel{R_k}\widehat{u}$, and hence the relation $R_k$ is antisymmetric. Therefore,  $\mathfrak{F}_k=(X_k,R_k,E_k)$ is a partially ordered $\textbf{MS4}$-frame. 

Finally, we show that the set $X_k$ is finite by induction on $k$. Clearly $X_0$ is finite. Let $k>1$. First note that the $E_k$-skeleton $(\mathfrak{F}_k)_0$ of $\mathfrak{F}_k$ is entirely determined by the points in the set $X_k^\Diamond$ because no new $E_k$-clusters are introduced in the commutativity step. Since $X_k^\Diamond$ is finite (because it is obtained by adding finitely many points to $X_{k-1}$, which is finite by the induction hypothesis), we conclude that $(\mathfrak{F}_k)_0$ is finite. By \crefpart{rmk}{rmk1}{rmk4}, we see that $\mathfrak{F}_k$ does not have any $Q_k$-clusters that are not already $E_k$-clusters (see \crefdefpart{E_Q}{E_Q1}), so $(\mathfrak{F}_k)_0$ is a poset, which is clearly rooted by construction. Let the depth of $(\mathfrak{F}_k)_0$ be $n<\omega$.
Observe that the newly added points in the $j^{th}$ commutativity step are given by the set $X_k^{j}-X_k^{j-1}$. It is clear that the points in $X_k^1-X_k^0$ cannot occur in the root cluster, the points in $X_k^2-X_k^1$ cannot occur in the root cluster or the $E_k$-clusters that correspond to the immediate successors of the root cluster in $(\mathfrak{F}_k)_0$, and so on. In general, points in $X_k^{j+1}-X_k^j$ are not in the $E_k$-clusters that correspond to points of depth at most $n-j$ in $(\mathfrak{F}_k)_0$. Therefore, at some $j$, newly added points occur in the final $E_k$-clusters, after which no more commutativity is demanded. Hence, with this $j$, we have $X_k^j=X_k^{j+l}$ for all $l<\omega$. Since $X_k^j$ is finite, $X_k$ is finite. 
\end{proof}

Let $\widehat{\mathfrak{F}} = (\widehat{X},\widehat{R},\widehat{E})$, where
\begin{equation*}
    \widehat{X}=\bigcup_{k<\omega}X_k, \quad \widehat{R}=\bigcup_{k<\omega}R_k, \quad \widehat{E}=\bigcup_{k<\omega}E_k.
\end{equation*}

As a direct consequence of \Cref{partial}, we obtain: 

\begin{lem}{\label{partial2}}
    $\widehat{\mathfrak{F}}$ is a partially ordered \textbf{MS4}-frame.
\end{lem}

Define a valuation $\widehat{v}$ by taking $\widehat{v}(p)=\{\widehat{y}\in \widehat{X}\mid y\in v(p)\}$
for a propositional letter $p$ in $S$, and $\widehat{v}(p)=\varnothing$ otherwise.

\begin{lem}[Truth Lemma]{\label{TL}}
    For any $\widehat{y}\in \widehat{X}$ and $\psi\in S$,
    \begin{equation*}
        \widehat{y}\models_{\widehat{v}}\psi \iff y\models_{v}\psi.
    \end{equation*}
\end{lem}

\begin{proof}
We prove this by induction on the complexity of the formula $\psi$. The base case follows from the definition of $\widehat{v}$, and the $\wedge$ and $\lnot$ cases are obvious. It is left to consider the $\Diamond$ and $\exists$ cases.

Suppose $y\models_v\exists\delta$. If $y\models_v\delta$, then 
$\widehat{y}\models_{\widehat{v}}\delta$ by the induction hypothesis, and hence $\widehat{y}\models_{\widehat{v}}\exists\delta$ since $\widehat{y}\mathrel{\widehat{E}}\widehat{y}$. Suppose $y\not\models_v\delta$. Then $\exists\delta\in W_y^\exists$. Let $k<\omega$ be the least number for which $\widehat{y}\in X_k$. Since $\exists\delta\in W_y^\exists$, by construction, there 
is $\widehat{z}\in X_{k+1}$ such that $\widehat{y}\mathrel{E_{k+1}}\widehat{z}$ and $z\models_v\delta$. 
By the induction hypothesis, $\widehat{z}\models_{\widehat{v}}\delta$, and $\widehat{y}\mathrel{E_{k+1}}\widehat{z}$ implies that $\widehat{y}\mathrel{\widehat{E}}\widehat{z}$. Thus,  $\widehat{y}\models_{\widehat{v}}\exists\delta$.

Suppose $\widehat{y}\models_{\widehat{v}}\exists\delta$. Then there is $\widehat{z}\in\widehat{X}$ such that $\widehat{y}\mathrel{\widehat{E}}\widehat{z}$ and $\widehat{z}\models_{\widehat{v}}\delta$. By the induction hypothesis, $z\models_v\delta$, and $\widehat{y}\mathrel{\widehat{E}}\widehat{z}$ implies $\widehat{y}\mathrel{E_k}\widehat{z}$ for some $k<\omega$. Therefore, $y\mathrel{E}z$ by \crefdefpart{rmk}{rmk1}, and thus $y\models_v\exists\delta$.

Now let $y\models_v\Diamond\delta$. If $y\models_v\delta$, 
then the induction hypothesis yields $\widehat{y}\models_{\widehat{v}}\delta$, and 
$\widehat{y}\mathrel{\widehat{R}}\widehat{y}$ implies  that $\widehat{y}\models_{\widehat{v}}\Diamond\delta$. Suppose $y\not\models_v\delta$. Then $\Diamond\delta\in W_y^\Diamond$, so there is a least $k<\omega$ such that $\widehat{y}\in X_k$. By construction, there is $\widehat{z}\in X_{k+1}$ such that $\widehat{y}\mathrel{R_{k+1}}\widehat{z}$ and $z\models_v\delta$. 
By the induction hypothesis, $\widehat{z}\models_v\delta$, and $\widehat{y}\mathrel{\widehat{R}}\widehat{z}$ since 
$R_{k+1}\subseteq\widehat{R}$. Thus,
$\widehat{y}\models_{\widehat{v}}\Diamond\delta$.

Finally, let $\widehat{y}\models_{\widehat{v}}\Diamond\delta$. Then there is $\widehat{z}\in\widehat{X}$ such that  $\widehat{y}\mathrel{\widehat{R}}\widehat{z}$ and $\widehat{z}\models_{\widehat{v}}\delta$. Choose the least $k,l<\omega$ for which $\widehat{y}\mathrel{R_k^l}\widehat{z}$. Because we take the reflexive-transitive closure every time an $R_k^l$-relation is introduced, we have a chain $$\widehat{y}=\widehat{x}_1\mathrel{R_k^l}\widehat{x}_2\mathrel{R_k^l}\dots
\widehat{x}_n\mathrel{R_k^l}\widehat{x}_{n+1}=\widehat{z}.$$ 

To see that $y\models\Diamond\delta$, we first show that for any $j\geq2$, we have  $x_j\models\Diamond\delta\implies x_{j-1}\models \Diamond\delta$. So, assume that $x_j\models\Diamond\delta$. If $x_{j-1}\mathrel{R}x_j$, then $x_{j-1}\models\Diamond\Diamond\delta$, which yields $x_{j-1}\models\Diamond\delta$. Suppose $x_{j-1}\mathrel{\cancel{R}}x_j$. By \crefdefpart{rmk}{rmk3}, $x_{j-1}\mathrel{Q}x_j$. Since $x_{j-1}\mathrel{Q}x_j$, $x_{j-1}\mathrel{\cancel{R}}x_j$, and $\widehat{x}_{j-1}\mathrel{R_k^l}\widehat{x}_j$, a proper $Q$-arrow was turned into a $R_k^l$-arrow. This can happen only in the vertical step of the construction, and hence $\widehat{x}_{j}$ is a vertical $\Diamond$-witness for $\widehat{x}_{j-1}$ with respect to some formula $\Diamond\psi\in S$. Since vertical $\Diamond$-witnesses are chosen using \Cref{vert}, we must have $x_j\in\textbf{smax}_R A$, where
\begin{equation*}
    A = v(\Diamond\psi)\cap\bigcap\bigl\{ v(\neg\Diamond\alpha)\mid \Diamond\alpha\in S,\, x_{j-1}\not\models \Diamond\alpha \bigr\}.
\end{equation*}
Therefore, if $x_{j-1}\not\models\Diamond\delta$, we must also have $x_j\not\models\Diamond\delta$, which contradicts our assumption that $x_j\models\Diamond\delta$. 
Thus, $x_{j-1}\models\Diamond\delta$.

Applying \crefdefpart{rmk}{rmk5} to $\widehat{x}_{n+1}=\widehat{z}$ yields $x_{n+1}=z$, and by the induction hypothesis we have $z\models\delta$. Consequently,   
$x_{n+1}\models\delta$, and by repeated application of $x_j\models\Diamond\delta\implies x_{j-1}\models\Diamond\delta$, we obtain $x_1\models\Diamond\delta$. From $\widehat{x}_1=\widehat{y}$ it follows that $x_1=y$ by \crefdefpart{rmk}{rmk5}, and hence $y \models\Diamond\delta$.
\end{proof}


The next two lemmas show that the length of each $\widehat{R}$-chain in $\widehat{\mathfrak{F}}=(\widehat{X},\widehat{R},\widehat{E})$ is bounded.

\begin{lem}{\label{chain1}}
    If $\widehat{C}$ is an $\widehat{R}$-chain \[\widehat{x}_1\mathrel{\widehat{R}}\widehat{x}_2\dots \widehat{x}_{n-1}\mathrel{\widehat{R}}\widehat{x}_n\dots\] in $\widehat{X}$ such that $\widehat{x}_i\mathrel{\widehat{E}}\widehat{x}_j$ for all $i$, $j$, then the length of $\widehat{C}$ is bounded by $2^{|S|}$.
\end{lem}

\begin{proof}
 First, observe that it is enough to establish the above claim when $\widehat{C}$ is a chain where every relation $\widehat{x}_i\mathrel{\widehat{R}}\widehat{x}_{i+1}$ is introduced at some stage of the construction (otherwise, we can decompose $\widehat{x}_i\mathrel{\widehat{R}}\widehat{x}_{i+1}$ further into a chain $\widehat{x}_i\mathrel{\widehat{R}}\widehat{y}_1\dots\widehat{y}_n\mathrel{\widehat{R}}\widehat{x}_{i+1}$, where the latter relations are indeed introduced through the course of the construction).  Now, note that $\widehat{x}_i\mathrel{\widehat{R}}\widehat{x}_j$ and $x_i\mathrel{\widehat{E}}x_j$ imply that $\widehat{x}_i\mathrel{R_k} \widehat{x}_j$ and $\widehat{x}_i\mathrel{E_k} \widehat{x}_j$ for some $k<\omega$. Hence, $x_i\mathrel{R}x_j$ and $x_i\mathrel{E}x_j$ by \crefpart{rmk}{rmk1}{rmk2}.
We let $C$ be the $R$-chain  
\[
x_1 \mathrel{R} x_2 \dots x_{n-1} \mathrel{R} x_{n} \dots
\] 
in $X$. 
We claim that $x_i \mathrel{\cancel{\sim}_S} \, x_j$ for $i \neq j$. Without loss of generality, suppose $i<j$. Let
\begin{equation*}
    U= \bigcap\{v(\alpha)\mid \alpha\in S,\,  x_i\models_v\alpha\} \cap \bigcap \{v(\neg\beta) \mid \beta\in S,\,x_i\not\models_v\beta\}.
\end{equation*}

We have:
\begin{itemize}
    \item $U$ is a clopen set in $X$ such that
    $x_i\in U$.
    \item $x_i\mathrel{\sim_S}x_k$ iff $x_k\in U$ for each $k$. 
    \item $\widehat{x}_i\mathrel{\widehat{R}} \widehat{x}_{i+1}$ and $\widehat{x}_i\mathrel{E_k}\widehat{x}_{i+1}$ for some $k<\omega$. Because ${R}_k$-arrows are drawn in $E_k$-clusters only to introduce $\Diamond$-witnesses,  it follows that $\widehat{x}_{i+1}$ is a witness for $\widehat{x}_i$ with respect to a formula $\Diamond\delta\in W_{{x}_i}^\Diamond$; i.e., $x_i\models_v\Diamond\delta$ and $x_i\not\models_v\delta$. Since $x_{i+1}\models_v\delta$, we have $x_{i+1}\notin v(\neg\delta)$. Further, $x_i\not\models_v\delta$ implies $U\subseteq v(\neg\delta)$. Therefore, because $x_{i+1}\notin v(\neg\delta)$, we see that $x_{i+1}\notin U$.  
\end{itemize}

Now, $\widehat{x}_i$ is added to $\widehat{X}$ in one of three ways: in the $\exists$-\textbf{Step}, in the $\Diamond$-\textbf{Step}, or in the \textbf{Commutativity Step}. We address each case separately. 

\begin{enumerate}
    \item Suppose $\widehat{x}_i$ is added in the $\exists$-\textbf{step}. Then, by \Cref{horz}, we must have $x_i\in \textbf{smax}_R(E[V]\cap v(\psi))$ for some clopen $V$ and $\exists\psi\in S$. Since $U\subseteq v(\psi)$, we have $E[V]\cap U\subseteq E[V]\cap v(\psi)$. Thus, $x_i\in\textbf{max}_R(E[V]\cap U)$. For the sake of contradiction, suppose $x_i\mathrel{\sim_S}x_j$, so $x_j\in U$ (and hence $j>i+1$ since $j>i$ and $x_{i+1}\notin U$). From  $x_i\mathrel{E}x_j$ and $x_i\in E[V]$ it follows that $x_j\in E[V]$. Therefore, $x_j\in E[V]\cap U$. We thus have
    $x_i\mathrel{R}x_{i+1}\mathrel{R}x_j$,  
    $x_i\in\textbf{max}_R(E[V]\cap U)$, $x_{i+1}\notin E[V]\cap U$, and $x_j\in E[V]\cap U$. This contradicts
    that $\mathfrak{F}=(X,R,E)$ is a descriptive \textbf{MGrz}-frame (see \Cref{enter-exit}).

    \item Suppose $\widehat{x}_i$ is introduced in the $\Diamond$-step. By \crefdefpart{vert}{vert3}, we must have ${x_i\in\textbf{smax}_R A \cap \textbf{max}_R v(\psi)}$ for some $\Diamond\psi\in S$. From $U\subseteq v(\psi)$ it follows that  $x_i\in\textbf{max}_R U$. If $x_i\mathrel{\sim_S}x_j$, then we have $x_i\mathrel{R}x_{i+1}\mathrel{R}x_j$,  $x_i\in\textbf{max}_R U$, $x_{i+1}\notin U$, and $x_j \in U$. This again contradicts that $(X,R,E)$ is a descriptive \textbf{MGrz}-frame.

    \item Suppose $\widehat{x}_i$ is introduced to ensure commutativity. Since such a point is chosen using \Cref{l-comm}, we must have $x_i\in\textbf{max}_R E[A]$. Thus, $x_i\in\textbf{max}_R E[A]\cap U$, and putting $A$ in place of $V$, the rest of the proof is identical to the first case.
\end{enumerate}

Thus, we conclude that $x_i\mathrel{\cancel{\sim}_S} x_j$ for $i \neq j$. Consequently, the length of $C$, and hence that of $\widehat{C}$, can be at most $2^{|S|}$.
\end{proof}


\begin{lem}{\label{chain2}}
    If $\widehat{C}$ is an $\widehat{R}_0$-chain \[[\widehat{y}_1]\mathrel{\widehat{R}_0}[\widehat{y}_2]\dots [\widehat{y}_{n-1}]\mathrel{\widehat{R}_0}[\widehat{y}_n]\dots\]  in $(\widehat{X}_0, \widehat{R}_0)$, then the length of $\widehat{C}$ is bounded by $2\cdot2^{|S|}$.
\end{lem}

\begin{proof}
From  $[\widehat{y}_i]\mathrel{\widehat{R}_0}[\widehat{y}_{i+1}]$ it follows that  $\widehat{y}_i\mathrel{\widehat{Q}}\widehat{y}_{i+1}$ and $\widehat{y}_i\mathrel{\cancel{\widehat{E}}}\widehat{y}_{i+1}$ (otherwise $[\widehat{y}_i] = [\widehat{y}_{i+1}]$). 
Let $k<\omega$ be the least number for which $\widehat{y}_i\mathrel{Q_k}\widehat{y}_{i+1}$. From $\widehat{y}_i\mathrel{\cancel{\widehat{E}}}\widehat{y}_{i+1}$ it follows that $\widehat{y}_i\mathrel{\cancel{E}_k}\widehat{y}_{i+1}$.  
We may assume that $\widehat{y}_{i+1}$ is an immediate $Q_k$-successor of $\widehat{y}_i$ since otherwise we can decompose $\widehat{y}_i\mathrel{Q_k}\widehat{y}_{i+1}$ further into a chain $\widehat{y}_i\mathrel{Q_k}\widehat{t}_1\dots\widehat{t}_{n}\mathrel{Q_k}\widehat{y}_{i+1}$, where each point along the chain is indeed an immediate $Q_k$-successor, 
and extend $\widehat{C}$ to a longer chain by including the points $[\widehat{t}_i]$. This means that $[\widehat{y}_{i+1}]$ is an immediate successor of $[\widehat{y}_i]$ in $(\mathfrak{F}_k)_0$. Since 
immediate successors in $(\mathfrak{F}_k)_0$ 
are determined by $R_k^\Diamond$-arrows drawn to introduce vertical $\Diamond$-witnesses,  there must be $\widehat{y}\in E_k[\widehat{y}_i]$, $\widehat{z}\in E_k[\widehat{y}_{i+1}]$, and 
$\Diamond\psi\in W_y^\Diamond$ such that $\widehat{z}$ is a vertical $\Diamond$-witness for $\widehat{y}$ with respect to $\Diamond\psi$. Because $\widehat{y}\in E_k[\widehat{y}_i]$, $\widehat{z}\in E_k[\widehat{y}_{i+1}]$, we have $[\widehat{y}]=[\widehat{y}_i]$ and $[\widehat{y}_{i+1}]=[\widehat{z}]$.
Therefore, $[\widehat{y}]\mathrel{\widehat{R}_0}[\widehat{y}_{i+1}]$ and $[\widehat{y}_{i+1}]=[\widehat{z}]$. This means that we can choose a representative $\widehat{y}$ for each $[\widehat{y}_i]$ such that the cluster $\widehat{E}[\widehat{y}_{i+1}]$ contains a vertical $\Diamond$-witness $\widehat{z}$ for $\widehat{y}$. Thus, 
we may assume that for each $i$ there is $\widehat{z}_{i+1}\in \widehat{X}$ such that $\widehat{y}_{i+1}\mathrel{\widehat{E}}\widehat{z}_{i+1}$, $\widehat{y}_i\mathrel{\widehat{R}}\widehat{z}_{i+1}$, and $\widehat{z}_{i+1}$ is added as a vertical $\Diamond$-witness for $\widehat{y}_i$ with respect to some 
$\Diamond\psi\in S$. 

Now, $[\widehat{y}_i]\widehat{R}_0[\widehat{y}_j]\implies\widehat{y}_i\mathrel{\widehat{Q}}\widehat{y}_j\implies y_i\mathrel{Q}y_j$, where the first implication follows from the definition of $\widehat{R}_0$ and the second 
from \crefdefpart{rmk}{rmk4}.  Let $C$ be the $Q$-chain \[y_1 \mathrel{Q} y_2\dots y_{n-1} \mathrel{Q} y_n \dots\] in $X$, where $y_i \mathrel{\cancel{E}} y_j$ for $i \neq j$. Note that $y_i\mathrel{\cancel{E}}y_j$ since $\widehat{y}_i\mathrel{\cancel{\widehat{E}}}\widehat{y}_{i+1}$ (see \crefdefpart{rmk}{rmk1}).  
Also since $\widehat{y}_{i+1}\mathrel{\widehat{E}}\widehat{z}_{i+1}$, 
we have that $y_{i+1}\mathrel{E}z_{i+1}$. 
Since $\widehat{z}_{i+1}$ is a $\Diamond$-witness for $\widehat{y}_i$, 
it must satisfy \Cref{vert}. 

Define a function $f\colon C\to \wp(S)$ by
\begin{equation*}
    f(y_i)=\{\psi\in S \mid y_i\models_v\psi \}.
\end{equation*}
Let $U\in\wp(S)$. We claim that $|f^{-1}(U)|\leq 2$. Suppose $|f^{-1}(U)|\geq 3$. This implies the existence of distinct $y_i$, $y_j$, $y_k$ for which $f(y_i)=f(y_j)=f(y_k)=U$. By the definition of $f$, we must have $y_i\mathrel{\sim_S}y_j\mathrel{\sim_S}y_k$. Without loss of generality, let $i<j<k$. Since $y_{i+1}\mathrel{E}z_{i+1}$, we have $z_{i+1}\mathrel{Q}y_j$ and $z_{i+1}\mathrel{Q}y_k$. 
Also, since $y_i\mathrel{\sim_S}y_j$, we have $z_{i+1}\mathrel{E}y_j$ by \crefdefpart{vert}{vert6}. Similarly, $z_{i+1}\mathrel{E}y_k$. Therefore, $y_j\mathrel{E}y_k$, which is a contradiction to 
$y_j \mathrel{\cancel{E}} y_k$ since $j<k$. Consequently, $C$ must have at most $2\cdot 2^{|S|}$ elements. Thus, the length of $\widehat{C}$ is at most $2\cdot 2^{|S|}$.
\end{proof}

\begin{rmk}
    It is worth pointing out that \Cref{chain1,chain2} are enough to obtain the fmp of $\textbf{MGrz}$. Together, they imply that $\widehat{\mathfrak{F}}$ has finite depth with respect to $\widehat{R}$. Indeed, the map $\pi\colon\widehat{\mathfrak{F}}\twoheadrightarrow\widehat{\mathfrak{F}}_0$ associates a chain $C$ in $\widehat{\mathfrak{F}}$ to a chain $C_0$ in $\widehat{\mathfrak{F}}_0$. By \Cref{chain2}, the length of $C_0$ is bounded by $2^{|S|+1}$, and for any $[\widehat{y}]\in C_0$, the chain $\pi^{-1}[\widehat{y}]\cap C$ has at most $2^{|S|}$ points by \Cref{chain1}. Thus, $C$ may have at most $2^{2|S|+1}$ points. Therefore, $\widehat{\mathfrak{F}}$ is an $\textbf{MGrz}$-frame of finite depth, and $\widehat{\mathfrak{F}}\not\models\varphi$ by \Cref{TL}. 

    Let $\textbf{MGrz}[n]$ denote the extension of $\textbf{MGrz}$ by the finite depth formula $bd_n$ (see, e.g., \cite[Prop.~3.44]{CZ}). Since $\widehat{\mathfrak{F}}$ is an $\textbf{MGrz}[n]$-frame with $n=2^{2|S|+1}$, we have $\varphi\notin\textbf{MGrz}[n]$. By \cite[Sec.~4.10]{locally_finite}, 
    $\textbf{MGrz}[n]$ is locally tabular, 
    so $\varphi$ is refuted in a finite $\textbf{MGrz}[n]$-frame, thereby establishing the fmp of $\textbf{MGrz}$. However, it is worth noting that the constructed frame $\widehat{\mathfrak{F}}$ is indeed finite, and we furnish the details below for the reader's convenience. 
\end{rmk}

To prove that $\widehat{\mathfrak{F}}=(\widehat{X},\widehat{R},\widehat{E})$ is finite, it suffices to show that $\widehat{\mathfrak{F}}_0=(\widehat{X}_0,\widehat{R}_0)$ is a finite poset and that $\widehat{E}[\widehat{y}]$ is finite for each $[\widehat{y}]\in\widehat{X}_0$. To this end, we first have the following.



\begin{lem}\label{lem: F0 is poset}
    The frame $\widehat{\mathfrak{F}}_0=(\widehat{X}_0, \widehat{R}_0)$ is a rooted poset.
\end{lem}
\begin{proof}
    Since $\widehat{R}$ is a preorder, 
    $\widehat{Q}$ is a preorder, and hence $\widehat{R}_0$ is a preorder. To show 
    anti-symmetry, let $[\widehat{x}]\mathrel{\widehat{R}_0}[\widehat{y}]$ and $[\widehat{y}]\mathrel{\widehat{R}_0}[\widehat{x}]$. Therefore,
    $\widehat{x}\mathrel{\widehat{Q}}\widehat{y}$ and $\widehat{y}\mathrel{\widehat{Q}}\widehat{x}$, so 
    there is $k<\omega$ such that
    $\widehat{x}\mathrel{Q_k}\widehat{y}$ and $\widehat{y}\mathrel{Q_k}\widehat{x}$. By \crefdefpart{rmk}{rmk4}, $x\mathrel{Q}y$ and $y\mathrel{Q}x$, yielding that $y\in E_Q[x]$. However, $x$ and $y$ are strongly maximal points, so
    \crefdefpart{E_Q}{E_Q1} implies that $E_Q[x]=E[x]$. Thus, $y\in E[x]$, and so $\widehat{y}\mathrel{{E}_k}\widehat{x}$ by \crefdefpart{rmk}{rmk1}. Therefore, $\widehat{y}\mathrel{\widehat{E}}\widehat{x}$, and so $[\widehat{x}]=[\widehat{y}]$. That $\widehat{\mathfrak{F}}_0$ is rooted is immediate from the construction. Indeed, we started the construction with $\widehat{x}$ where $x\in\textbf{smax}_R v(\neg\varphi)$; and for each $[\widehat{y}]\in\widehat{X}_0$, 
    we have $[\widehat{x}]\mathrel{\widehat{R}_0}[\widehat{y}]$ because $\widehat{x}\mathrel{\widehat{Q}}\widehat{y}$. Thus, $[\widehat{x}]$ is the root {of~$\widehat{\mathfrak{F}}_0$}. 
\end{proof}

For a poset $(P,\leq)$, recall that the {\em branching} at a point $x\in P$ is the cardinality of the set of its immediate successors.

\begin{lem}{\label{rooted}}
    A rooted poset $(P,\leq)$ is finite iff it has finite depth and finite branching at each point.
\end{lem}
\begin{proof}
    The forward implication is obvious. For the converse, suppose $P$ has depth $n$ and finite branching at each $x\in P$. Let $r$ be the root of $(P,\leq)$. We proceed by induction to define a finite subset $C_m\subseteq P$ for each $m\leq n$ such that $(C_m,\leq)$ has depth $m$. 
    Set $C_1=\{r\}$ and observe that $C_1$ is finite and $(C_1,\leq)$ has depth 1.
    Now suppose $C_k\subseteq P$ has been defined for $k<n$. By the induction hypothesis, 
    we may assume that $C_k$ is finite and that $(C_k,\leq)$ has depth $k$. We define $C_{k+1}$ to be the set obtained by adjoining the immediate successors of elements in $C_k$. Since the branching is finite at each point, we only introduce finitely many new points. Therefore, $C_{k+1}$ is finite. Moreover, adjoining immediate successors increases the length of chains by 1. Hence, $(C_{k+1},\leq)$ has depth $k+1$. 
    Thus, $C_m$ is finite and $(C_m,\leq)$ has depth $m$ for all $m\leq n$. Since $C_{n}=P$, we conclude that $P$ is finite. 
\end{proof}

In \Cref{lem: F0 is poset}, we saw that $\widehat{\mathfrak{F}}_0$ is a rooted poset whose root is $[\widehat{x}]$. We next show that $\widehat{E}[\widehat{x}]$ is finite. To this end, note that a rooted poset of depth $n$ and branching at most $k$ at each point can have at most $1+k+k^2+\cdots+k^{n-1}$ points.

\begin{lem}{\label{base-case}}
    $\widehat{E}[\widehat{x}]$ is finite.
\end{lem}
\begin{proof}
    We start the construction by taking $\widehat{x}$ along with at most $|S|$-many $\exists$-witnesses. Hence, the cluster $\widehat{E}[\widehat{x}]$ initially has at most $|S|+1$ points. Then for each such point, we add at most $|S|$-many $\Diamond$-witnesses to $\widehat{E}[\widehat{x}]$, and we continue adding points to $\widehat{E}[\widehat{x}]$ in this manner. But this process terminates 
    since, by \Cref{chain1}, 
    $\widehat{R}$-chains within $\widehat{E}$-clusters have 
    length 
    at most $2^{|S|}$. Therefore, $\widehat{E}[\widehat{x}]$ has at most as many points as the disjoint union of $(|S|+1)$-many rooted posets of depth at most $2^{|S|}$ and the branching size at each point at most $|S|$. Thus, the size of $\widehat{E}[\widehat{x}]$ is bounded by $(|S|+1)(1+|S|+|S|^2+\cdots+|S|^{2^{|S|}-1})$. 
\end{proof}

\begin{lem}{\label{clusters-branching}}
    For each $[\widehat{y}]\in\widehat{X}_0$, we have that $\widehat{E}[\widehat{y}]$ is finite and $[\widehat{y}]$ has finite $\widehat{R}_0$-branching. 
\end{lem}
\begin{proof}
    We show that if the asserted property holds for immediate predecessors of $[\widehat{y}]$, it also 
    holds for $[\widehat{y}]$. 
    By \Cref{base-case}, $\widehat{E}[\widehat{x}]$ is finite. 
     Each $\widehat{y}\in\widehat{E}[\widehat{x}]$ can contribute at most 
    $|S|$-many immediate $R_0$-successors 
    of $[\widehat{x}]$ . Therefore, $\widehat{E}[\widehat{x}]$ can have at most $|S|\cdot|\widehat{E}[\widehat{x}]|$-many immediate $\widehat{R}_0$-successors. Thus, the $\widehat{R}_0$-branching of $[\widehat{x}]$ is at most $|S|\cdot|\widehat{E}[\widehat{x}]|$, and so the claim holds for the root $[\widehat{x}]$.

    Now 
    assume that for each $[\widehat{u}]\in\widehat{R}_0^{-1}[\widehat{y}]\setminus\{[\widehat{y}]\}$ we have that $\widehat{E}[\widehat{u}]$ is finite and $[\widehat{u}]$ has finite $\widehat{R}_0$-branching. 
    Then $(\widehat{R}_0^{-1}[\widehat{y}]\setminus\{[\widehat{y}]\},\widehat{R}_0)$ is a finite poset by \Cref{rooted}.
    Let $F$ be the set of immediate $\widehat{R}_0$-predecessors of $[\widehat{y}]$. Clearly, $F$ is finite. Define
    \begin{equation*}
        N=\max\{|\widehat{E}[\widehat{u}]| \mid [\widehat{u}]\in F \}.
    \end{equation*}
    
     Let $[\widehat{u}]\in F$ and $\widehat{z}\in \widehat{E}[\widehat{u}]$. We draw at most $|S|$-many $\widehat{R}$-arrows to introduce $\Diamond$-witnesses for $\widehat{z}$. Due to commutativity, we would need to draw at most $|S|\cdot N$-many $\widehat{R}$-arrows from points in $\widehat{E}[\widehat{u}]$ to introduce $\Diamond$-witnesses for $\widehat{z}$. Hence, to add all the $\Diamond$-witnesses for points in $\widehat{E}[\widehat{u}]$, we would need to draw at most $|S|\cdot N^2$-many $\widehat{R}$-arrows. Thus, we would have to draw $|S|\cdot N^2\cdot |F|$-many $\widehat{R}$-arrows to add all the necessary $\Diamond$-witnesses for any point in $\widehat{E}[\widehat{u}]$ and $[\widehat{u}]\in F$. Therefore, the cluster $\widehat{E}[\widehat{y}]$ may start with at most $|S|\cdot N^2\cdot |F|$-many points. Then an argument similar to that in \Cref{base-case} yields that $\widehat{E}[\widehat{y}]$ has at most $(|S|+|S|\cdot N^2\cdot |F|)(1+|S|+|S|^2+\cdots+|S|^{2^{|S|}-1})$-many points; and 
     the $\widehat{R}_0$-branching of $[\widehat{y}]$ is at most $|S|\cdot|\widehat{E}[\widehat{y}]|$. 
\end{proof}

\begin{lem}{\label{finite-skeleton}}
    The frame $\widehat{\mathfrak{F}}_0=(\widehat{X}_0,\widehat{R}_0)$ is finite.
\end{lem}
\begin{proof}
    By \Cref{lem: F0 is poset}, $\widehat{\mathfrak{F}}_0$ is a rooted poset. 
    By \Cref{chain2}, $\widehat{\mathfrak{F}}_0$ has finite depth; and by \Cref{clusters-branching}, $\widehat{\mathfrak{F}}_0$ has finite branching at each point. Thus,  $\widehat{\mathfrak{F}}_0$ is finite by \Cref{rooted}.
\end{proof}

\begin{thm}{\label{finite}}
    $\widehat{\mathfrak{F}}$ is a finite partially ordered {\textbf{MS4}}-frame.
\end{thm}

\begin{proof}
By \Cref{partial2}, $\widehat{\mathfrak{F}}$ is a partially ordered \textbf{MS4}-frame. By \Cref{finite-skeleton}, $\widehat{\mathfrak{F}}_0=(\widehat{X}_0, \widehat{R}_0)$ is finite. By \Cref{clusters-branching}, $\widehat{E}[\widehat{y}]$ is finite for each $[\widehat{y}]\in\widehat{X}_0$. Thus,  $\widehat{\mathfrak{F}}$ is finite.
\end{proof}

\begin{thm}{\label{fmp}}
    {\textbf{MGrz}} has the fmp. 
\end{thm}

\begin{proof}
Suppose $\textbf{MGrz}\not\vdash\varphi$. Then there is a descriptive \textbf{MGrz}-frame $\mathfrak{F}$ and a valuation $v$ on $\mathfrak{F}$ such that $\mathfrak{F}\not\models_v\varphi$. From the above construction, we can extract a finite partially ordered \textbf{MS4}-frame $\widehat{\mathfrak{F}}$. 
By \Cref{TL}, we can define a model $(\widehat{\mathfrak{F}},\widehat{v})$ such that $\widehat{\mathfrak{F}}\not\models_{\widehat{v}}\varphi$.
\end{proof}

The following are some immediate consequences of \Cref{fmp}. 

\begin{cor}{\label{consequences}}
    \begin{enumerate}
        \item[] 
        \item \textbf{MGrz} is complete with respect to \textbf{MGrz}-frames. {\label{completeness}}
        \item \textbf{MGrz} is decidable. {\label{decidable}}
        \item \textbf{MGrz} is the monadic fragment of \textbf{QGrz}. {\label{fragment}}
    \end{enumerate}
\end{cor}
\begin{proof}
(\labelcref{completeness}) This is an immediate consequence of the fmp for \textbf{MGrz}.

(\labelcref{decidable}) Since \textbf{MGrz} is finitely axiomatizable and has the fmp, \textbf{MGrz} is decidable by Harrop's theorem (see, e.g., \cite[Thm.~16.3]{CZ}).

(\labelcref{fragment}) By (\labelcref{completeness}), \textbf{MGrz} is complete with respect to the class $\mathsf{C}$ of all \textbf{MGrz}-frames. Let $\mathscr{B}(\mathsf{C})=\{\mathscr{B}(\mathfrak{F})\mid \mathfrak{F}\in\mathsf{C}\}$. We show that \textbf{QGrz} is strongly sound with respect to $\mathscr{B}(\mathsf{C})$. Let $\mathfrak{F}\in \mathsf{C}$, where $\mathfrak{F}=(X,R,E)$. We claim that $\mathscr{B}(\mathfrak{F})\models^+\hspace{-3pt}\textbf{QGrz}$. By \Cref{sound}, it suffices to show that the $n^{\text{th}}$ level of $\mathscr{B}(\mathfrak{F})$ 
is a Noetherian poset. From the definition of $\mathscr{B}$, we have $\mathscr{B}(\mathfrak{F})=((X,R),\pi_X,(X_0,R_0))$. Thus, the $n^{\text{th}}$ level of $\mathscr{B}(\mathfrak{F})$ is the frame $(X^n,R^n)$ for $n>1$, $(X,R)$ for $n=1$, and $(X_0,R_0)$ for $n=0$. Since $\mathfrak{F}$ is an \textbf{MGrz}-frame, $(X,R)$ and $(X_0,R_0)$ are Noetherian posets. The former gives that $(X^n,R^n)$ is also a Noetherian poset for $n>1$. Thus, by \Cref{useful}, \textbf{MGrz} is the monadic fragment of \textbf{QGrz}.
\end{proof}


\section{Conclusions}{\label{conclusion}}

The main contributions of this paper can be summarized as follows: 
\begin{enumerate}
    \item We proved that 
    the categories of Kripke bundles and \textbf{MK}-frames are equivalent, thus generalizing the results of \cite{szk} and \cite{BV}. 
    \item We provided a criterion to determine when a given monadic modal logic is the monadic fragment of a given predicate modal logic, thus generalizing the results of \cite{onoszk,szk}. This, in particular, yields a criterion for when the monadic extension of a propositional modal logic $\bf L$ is the monadic fragment of the predicate extension of $\bf L$. 
    \item We refined the selective filtration methods of 
    \cite{Grefe} and \cite{GBm} to show that \textbf{MGrz} has the fmp,  thus resolving in the positive the issue of completeness of \textbf{MGrz} (see \cite{esa}).  The key new ingredient is the notion of 
    a strongly maximal point of 
    a descriptive \textbf{MGrz}-frame, and an adaptation of the Fine-Esakia maximality principle. 
    As a consequence, we proved that \textbf{MGrz} 
    is the monadic fragment of \textbf{QGrz}.
\end{enumerate}


We point out that our construction, in particular, implies that the Grzegorczyk logic enriched with the universal modality also has the fmp. For a propositional modal logic \textbf{L}, we let  
\[
\textbf{L}_u=\textbf{ML}+\Diamond q\to\exists q .
\]
We refer to $\textbf{L}_u$ as ``the modal logic \textbf{L} enriched with the universal modality." A detailed study of such logics was initiated in \cite{goranko}, where it was shown (among other things) that if $\textbf{L}$ admits filtration, then so does $\textbf{L}_u$. 
Since $\textbf{Grz}$ does not admit standard filtration, the above result does not immediately yield that $\textbf{Grz}_u$ has the fmp. However, a simplified version of our construction does yield that $\textbf{Grz}_u$ has the fmp. 

Recall that an \textbf{MGrz}-frame $(X,R,E)$ is a $\textbf{Grz}_u$-frame iff $x\mathrel{R}y\implies x\mathrel{E}y$ for each $x,y\in X$. Therefore, $(X,R,E)$ is rooted iff $E=X^2$. Thus, our construction in \Cref{construction} simplifies considerably since the selection needs to be performed only within one $E$-cluster. Consequently, we obtain:

 
\begin{cor}
    $\textbf{Grz}_u$ has the fmp.
\end{cor}

We next observe that the criterion developed in \Cref{translation,useful} can be applied to a variety of mm- and pm-logics. Below we give some additional examples. 

It is well known (see, e.g., \cite[Sec.~3.9]{CZ}) that both \textbf{IPC} and \textbf{Grz} embed faithfully into the provability logic \textbf{GL}, and hence admit provability interpretations. 
Japaridze \cite{japaridze88,japaridze} 
proved that Solovay's translation \cite{solovay} 
of \textbf{GL} into Peano Arithmetic
extends to \textbf{MGL}. 
In doing so, he proved that \textbf{MGL} 
has the fmp. 

However, \textbf{MIPC} and \textbf{MGrz} no longer embed faithfully into \textbf{MGL}. To get a faithful embedding, we need to work with proper extensions of these systems (see \cite{GBm}). Let
\begin{align*}
    &\textbf{M}^+\textbf{IPC}=\textbf{MIPC}+\forall((p\to\forall p)\to\forall p)\to\forall p,\\
    &\textbf{M}^+\textbf{Grz}=\textbf{MGrz}+\Box\forall(\Box(\Box p\to \Box\forall p)\to\Box\forall p)\to\Box\forall p.
\end{align*}

%
%
%
Semantic characterizations of $\textbf{M}^+\textbf{IPC}$ and $\textbf{M}^+\textbf{Grz}$ are given in \cite{GBm}.
This yields the following characterization of finite $\textbf{M}^+\textbf{Grz}$-frames: a finite \textbf{MK}-frame $\mathfrak F=(X,R,E)$ is an  $\textbf{M}^+\textbf{Grz}$-frame iff $R$ is a partial order such that $x\mathrel{E}y$ and $x\mathrel{R}y$ imply $x=y$ for all $x,y\in X$. 
A similar characterization yields that $\mathfrak F$ is a finite \textbf{MGL}-frame iff $R$ is a strict partial order such that $x\mathrel{E}y$ 
implies $x \mathrel{\cancel{R}} y$ for all $x,y\in X$. This explains why it is $\textbf{M}^+\textbf{Grz}$ and not \textbf{MGrz} that embeds faithfully into \textbf{MGL}. 

Let $\textbf{Q}^+\textbf{Grz}=\textbf{QGrz}+\Box\forall x(\Box(\Box p(x)\to \Box\forall xp(x))\to\Box\forall xp(x))\to\Box\forall xp(x)$. As was pointed out in \cite[Rem.~5.19]{GBm}, $\textbf{M}^{+}\textbf{IPC}$ axiomatizes the monadic fragment of the corresponding predicate intuitionistic logic. That the same is true for $\textbf{M}^{+}\textbf{Grz}$ and $\textbf{MGL}$ 
is a consequence of 
\Cref{translation} and \Cref{useful}: 



%
%

\begin{samepage}
\begin{cor}{\label{MGL_M+Grz}}
    \begin{enumerate}
        \item[] 
        \item $\textbf{MGL}$ is the monadic fragment of $\textbf{QGL}$. \label{folk3}
        \item $\textbf{M}^{+}\textbf{Grz}$ is the monadic fragment of $\textbf{Q}^{+}\textbf{Grz}$. \label{folk4}
    \end{enumerate}
\end{cor}
\end{samepage}
    
\begin{proof}
(\labelcref{folk3}) By \cite[Lem.~9]{japaridze}, 
\textbf{MGL} is 
complete with respect to those finite \textbf{MGL}-frames $\mathfrak{F}$ for which $\mathscr{B}(\mathfrak{F})$ is a predicate Kripke frame (see, e.g., \cite[Def.~3.2.2]{GSS} for the definition). 
Since those are \textbf{QGL}-frames, the result follows from \Cref{useful}.

(\labelcref{folk4}) By \cite{GBm}, $\textbf{M}^+\textbf{Grz}$ is complete with respect to those finite $\textbf{M}^+\textbf{Grz}$-frames $\mathfrak{F}$ for which $\mathscr{B}(\mathfrak{F})$ is a predicate Kripke frame.  
We claim that each such $\mathscr{B}(\mathfrak{F})$ is a $\textbf{Q}^+\textbf{Grz}$-frame. 
The defining axiom of $\textbf{Q}^+\textbf{Grz}$ is the G\"{o}del translation of the Casari formula (see, e.g., \cite[p.~59]{ono}, where the formula is denoted by $W^*$), which
holds in every Noetherian predicate Kripke frame (see, e.g., \cite[Thm.~3(2)]{ono}).
Therefore, its G\"{o}del translation also holds in every such viewed as a \textbf{QS4}-frame. But the Noetherian condition holds trivially in each
$\mathscr{B}(\mathfrak{F})$ since these are finite. Thus, $\mathscr{B}(\mathfrak{F})\models \textbf{Q}^+\textbf{Grz}$, which implies that $\mathscr{B}(\mathfrak{F})\models^+\hspace{-2.8pt}\textbf{Q}^+\textbf{Grz}$ 
since $\mathscr{B}(\mathfrak{F})$ is a predicate Kripke frame.
It is also clear that $\textbf{M}^+\textbf{Grz}\vdash\varphi$ implies $\textbf{Q}^+\textbf{Grz}\vdash\varphi^t$ for each $\mathcal{L}_\exists$-formula $\varphi$ since the defining axiom of  $\textbf{Q}^+\textbf{Grz}$ is the translation of the defining axiom of  $\textbf{M}^+\textbf{Grz}$. 
Consequently, \Cref{translation} is applicable, and the result follows.  
\end{proof}

%
Our construction also yields the fmp for $\textbf{M}^+\textbf{Grz}$ since the only adjustment would be the omission of the \textbf{Horizontal Step} when adding $\Diamond$-witnesses. Thus, the main result in \cite[Sec.~6]{GBm} can be seen as a specific consequence of the selection technique we have developed here. This technique can further be adapted to show that \textbf{MGL} has the fmp: Given a descriptive \textbf{MGL}-frame, we can proceed by selecting irreflexive maximal points. For any $x$, $y$ in the $E$-cluster of an irreflexive maximal point, we have $x \mathrel{\cancel{R}} y$.
Therefore, 
we no longer need the \textbf{Horizontal Step} to introduce $\Diamond$-witnesses. A similar counting argument as in the case for \textbf{MGrz} shows that chains along the selected points are of finite length. 
Consequently, we obtain: 
\begin{cor}
\begin{enumerate}
    \item[] 
    \item \cite[Thm. 6.16]{GBm} $\textbf{M}^+\textbf{Grz}$ has the fmp.
    \item \cite[Lem. 9]{japaridze} \textbf{MGL} has the fmp.
\end{enumerate}
\end{cor}

An important extension of \textbf{MGrz} is obtained by adding the monadic version of the Barcan formula:
\[
\textbf{MGrzB}=\textbf{MGrz}+\Diamond\exists p\to \exists\Diamond p.
\]
The monadic Barcan formula $\Diamond\exists p\to \exists\Diamond p$ is valid on an \textbf{MK}-frame $\mathfrak F=(X,R,E)$ iff 
\[
(\forall x,y,z\in X) \, (x \mathrel{E} z \mbox{ and } y \mathrel{R} z \Longrightarrow \exists u \in X : u \mathrel{R} x \mbox{ and } u \mathrel{E} y).
\]
\[\begin{tikzcd}
	x && z \\
	\\
	u && y
	\arrow["E"{description}, tail reversed, from=1-1, to=1-3]
	\arrow["R"{description}, dashed, from=3-1, to=1-1]
	\arrow["E"{description}, dashed, tail reversed, from=3-1, to=3-3]
	\arrow["R"{description}, from=3-3, to=1-3]
\end{tikzcd}\]

In other words, $\mathfrak F \models \Diamond\exists p\to \exists\Diamond p$ iff $E\circ R=R\circ E$. We have the following proper inclusions
%
%
{
\begin{equation*}  \textbf{MGrz}\subset\textbf{MGrzB}\subset\textbf{Grz}_u.
\end{equation*}
}
As we have seen, both \textbf{MGrz} and $\textbf{Grz}_u$ possess the fmp. However, it is not immediately clear whether $\textbf{MGrzB}$ also has the fmp. Indeed, our construction would require an appropriate adjustment
to ensure that the condition $E\circ R=R\circ E$ is satisfied. The challenge is twofold --- firstly, the necessary arrows for backward commutativity (as defined above) should be introduced through strongly maximal points; secondly, it must be shown that this additional step does not affect the termination of the selection process. 

In view of the above, it is worth considering whether the systems $\textbf{M}^+\textbf{GrzB}$ and $\textbf{MGLB}$, obtained by adding the Barcan formula to $\textbf{M}^+\textbf{Grz}$ and $\textbf{MGL}$, also have the fmp. More generally, it is desirable to develop a selective filtration technique for mm-logics that contain the monadic Barcan formula.

\section*{Acknowledgments}

We are thankful to the two 
reviewers and Luca Carai for careful reading and 
constructive comments, which have improved both the clarity and presentation of the manuscript.

\printbibliography 

\end{document}